\documentclass[11pt,reqno]{amsart}

\usepackage{amsmath, amssymb, amsthm}
\usepackage{mathtools}
\usepackage{graphicx} % Required for inserting images
\usepackage{enumerate}
\usepackage[colorlinks=true,allcolors=blue]{hyperref}

\usepackage{tikz}
\usetikzlibrary{cd,arrows,calc}

\usepackage[capitalise]{cleveref}

\newcommand{\C}{\mathbb{C}}
\newcommand{\cU}{\mathcal{U}}
\newcommand{\cO}{\mathcal{O}}
\newcommand{\cR}{\mathcal{R}}
\newcommand{\bN}{\mathbb{N}}
\newcommand{\As}{A^s}
\newcommand{\R}{\mathbb{R}}
\newcommand{\Z}{\mathbb{Z}}
\newcommand{\N}{\mathbb{N}}

\newcommand{\cB}{\mathcal{B}}
\newcommand{\cD}{\mathcal{D}}
\newcommand{\covU}[1]{\widetilde{U}(#1)}

\newcommand{\bK}{\mathbb{K}}
\newcommand{\bT}{\mathbb{T}}

\newcommand{\Aut}[1]{\text{\normalfont Aut}(#1)}
\newcommand{\Autt}[1]{\text{\normalfont Aut}_\tau(#1)}
\newcommand{\Autz}[1]{\text{\normalfont Aut}_0(#1)}
\newcommand{\Aute}[1]{\text{\normalfont Aut}_{1\otimes e}(#1)}
\newcommand{\Outt}[1]{\text{\normalfont Out}_\tau(#1)}
\newcommand{\Out}[1]{\text{\normalfont Out}(#1)}
\newcommand{\Inn}[1]{\text{\normalfont Inn}(#1)}
\newcommand{\id}[1]{\operatorname{id}_{#1}}

\newcommand{\Fun}[2]{\text{\normalfont Fun}_{\text{\normalfont su}}(#1,#2)} % strictly unital functors
\newcommand{\stFun}[2]{\text{Fun}_{\text{str}}(#1,#2)} % strict functors
\newcommand{\stF}[2]{\mathcal{F}_{\text{str}}(#1,#2)} % strict functor category
\newcommand{\lxF}[2]{\mathcal{F}(#1,#2)} % lax functor category
\newcommand{\Flax}{\mathcal{F}} % pseudofunctor category
\newcommand{\Top}{\mathcal{T}\!op} % category of compactly generated Hausdorff spaces
\newcommand{\Deltainj}{\Delta_{\text{inj}}}

\newcommand{\cC}{\mathcal{C}}
\newcommand{\cG}{\mathcal{G}}
\newcommand{\cF}{\mathcal{F}}
\newcommand{\wcG}{\widetilde{\mathcal{G}}}
\newcommand{\cH}{\mathcal{H}}
\newcommand{\cV}{\mathcal{V}}

\DeclareMathOperator{\Ad}{Ad}

\DeclareMathOperator{\ob}{ob}
\DeclareMathOperator{\el}{el}

\newtheorem{theorem}{Theorem}[section]
\newtheorem{lemma}[theorem]{Lemma}
\newtheorem{corollary}[theorem]{Corollary}
\newtheorem{prop}[theorem]{Proposition}

\theoremstyle{definition}
\newtheorem{definition}[theorem]{Definition}
\newtheorem{remark}[theorem]{Remark}

\numberwithin{equation}{section}

\title{G-kernels and Crossed Modules}
\author{Sergio Girón Pacheco}
\address{\hskip-\parindent Sergio Girón Pacheco, Department of mathematics, KU Leuven, Celestijnenlaan 200B, 3001, Leuven, Belgium.}
\email{sergio.gironpacheco@kuleuven.be}

\author{Masaki Izumi}
\address{\hskip-\parindent Masaki Izumi, Graduate School of Science, Kyoto University, Sakyo-ku, Kyoto 606-8502, Japan}
\email{izumi@math.kyoto-u.ac.uk}

\author{Ulrich Pennig}
\address{\hskip-\parindent Ulrich Pennig, School of Mathematics, Cardiff University, Cardiff, CF24 4AG, Wales, UK}
\email{pennigu@cardiff.ac.uk}
\date{}

\begin{document}

\maketitle

%TODO: Abstract might need to be shortened
\begin{abstract}
%We study topological invariants of 
%group actions, cocycle actions and 
%$\Gamma$-kernels on $C^*$-algebras for discrete groups $\Gamma$. In particular, 
We develop a unified framework based on topological crossed modules for various lifting obstructions for $\Gamma$-kernels. It allows us to identify actions, cocycle actions and $\Gamma$-kernels up to their natural equivalence relations with cohomology sets. The obstructions then appear as boundary maps in corresponding exact sequences. Since topological crossed modules are topological $2$-groups (in the categorical sense), they have classifying spaces, which come with a natural transformation from the cohomology to a homotopy set. For the crossed module that gives cocycle actions we prove a weak equivalence of the classifying space of the crossed module with one from bundle theory. In case the algebra is strongly self-absorbing we show that the homotopy set is a group and that the above natural transformation is a group isomorphism on an appropriate restriction of the cohomology set.
%We develop a unified framework based on topological crossed modules for the various lifting obstructions for $\Gamma$-kernels that appear in the literature. It allows us to identify actions, cocycle actions and $\Gamma$-kernels up to their natural equivalence relations with cohomology sets $H^1(\Gamma, \cG)$. The obstructions then appear as boundary maps in corresponding exact sequences. Since topological crossed modules are topological $2$-groups (in the categorical sense), they have classifying spaces $B^D\cG$, which come with a natural transformation $H^1(\Gamma, \cG) \to [B\Gamma, B^D\cG]$. For the crossed module $\cG_A$ that gives cocycle action we prove a weak equivalence $B^D\cG_A \simeq B\Autz{A \otimes \bK}$. In case $A$ is strongly self-absorbing we show that $[B\Gamma, B^D\cG_A]$ is a group and that the above natural transformation is a group homomorphism when restricted appropriately.
\end{abstract}

\section*{Introduction}
An action of a discrete group $\Gamma$ on a $C^*$- or von Neumann algebra $A$ is a group homomorphism $\alpha \colon \Gamma \to \Aut{A}$. Several weaker notions of this definition have played a pivotal role in the classification of group actions on these algebras. A $\Gamma$-kernel, for example, is a group homomorphism
\[
    \bar{\alpha} \colon \Gamma \to \Out{A}\ ,
\] 
where $\Out{A} = \Aut{A}/\Inn{A}$ and $\Inn{A}$ denotes the normal subgroup of inner automorphisms, i.e.\ those that are of the form $\Ad_{u}$ for $u \in U(M(A))$.\footnote{Usually these are called $G$-kernels in the literature, as $G$ is the preferred symbol for a group. However we prefer to denote a discrete group by $\Gamma$.} In a similar spirit cocycle actions consist of pairs of maps $(\alpha,u)$ with
\[
    \alpha \colon \Gamma \to \Aut{A} \quad, \quad u \colon \Gamma \times \Gamma \to U(M(A))\ ,
\]
where $u$ measures the defect of $\alpha$ being a group homomorphism and satisfies a cocycle condition arising from associativity.

The problem of classifying group actions, $\Gamma$-kernels and cocycle actions on operator algebras has an extensive history. In the case of von Neumann algebras, early remarkable results have been achieved by Connes \cite{paper:ConnesOutR, paper:ConnesPeriodic}, Jones~\cite{paper:JonesActions} and Ocneanu~\cite{paper:OcneanuActions}. Showing for example that there is a unique pointwise outer action of any countable, amenable group $\Gamma$ on the Hyperfinite type {II}$_1$ factor $\cR$ up to outer conjugacy. More generally Connes, Jones and Ocneanu showed that injective $\Gamma$-kernel on $\cR$ are classified up to conjugacy by their lifting obstruction. For a $\Gamma$-kernel $\alpha$, its lifting obstruction $\ob(\alpha)\in H^3(\Gamma,\bT)$, is precisely the obstruction to lifting $\alpha$ to a cocycle action. Katayama, Sutherland and Takesaki have extended these results to other injective factors \cite{paper:KatayamaTakesaki1,paper:KatayamaTakesaki2,paper:KatayamaTakesaki3,KATASU98}. 

%culminating for example in a full classification of $\Gamma$-kernels for amenable groups $\Gamma$ on the hyperfinite type $II_1$-factor $R$ up to conjugacy. The invariant of a $\Gamma$-kernel $\alpha \colon \Gamma \to \Out{R}$ is the obstruction to lifting $\alpha$ to a cocycle action on $R$, which takes values in group cohomology with circle coefficients, more precisely in the group $H^3(\Gamma, \bT)$. Katayama and Takesaki have extended these results to other factors \cite{paper:KatayamaTakesaki1,paper:KatayamaTakesaki2,paper:KatayamaTakesaki3}. 

 \par In the context of group actions on C$^*$-algebras, early classification results were achieved in the presence of the Rokhlin property (\cite{paper:HermanJones, paper:HermanOcneanu, paper:IzumiRokhlin1, paper:IzumiRokhlin2, paper:Kishimoto4} see also \cite{IZ10} for a detailed survey). Although the Rokhlin property implies strong K-theoretical constraints in general, these often vanish in the case of torsion-free groups. For Kirchberg algebras, Nakamura proved in \cite{paper:NakamuraAperiodic} that aperiodic automorphisms automatically satisfy the Rokhlin property and are completely classified by their $KK$-classes up to $KK$-trivial cocycle conjugacy. Building on Nakamura's work, the second named author together with Matui developed a classification of a large class of $\Z^2$-actions on Kirchberg algebras \cite{paper:IzumiMatuiZ2}, which they later extended to Poly-$\Z$-group actions~\cite{IZMA21,paper:IzumiMatuiPolyZ2}.
 \par Motivated by their partial results, in \cite[Conjecture 1.1]{paper:IzumiMatuiPolyZ2} Matui and the second named author conjectured that for a torsion-free countable discrete amenable group $\Gamma$ and a stable Kirchberg algebra $A$ the transformation
\begin{align} \label{eqn:OABGamma}
    \mathcal{OA}(\Gamma,A)/\!\!\sim_{cc}\ &\to\ [B\Gamma, B\Aut{A}]\\
    [\alpha]_{cc} &\mapsto\ [B\alpha]_h\notag
\end{align}
is a bijection, where the left hand side denotes the outer actions of $\Gamma$ on $A$ up to cocycle conjugacy $\sim_{cc}$. The map above arises as follows: for every topological group $G$ there exists a universal principal $G$-bundle $EG \to BG$ over the classifying space $BG$ such that the pullback gives a bijection between the homotopy set $[X,BG]$ and the isomorphism classes of numerable principal $G$-bundles over $X$. A group action $\alpha \colon \Gamma \to \Aut{A}$ on a $C^*$-algebra $A$ gives rise to a principal $\Aut{A}$-bundle 
\[
    E\Gamma \times_{\alpha} \Aut{A} \to B\Gamma
\]
that corresponds to a map $B\alpha \colon B\Gamma \to B\Aut{A}$ (unique up to homotopy).
 
%In the context of group actions on $C^*$-algebras, early work includes papers by Herman and Jones \cite{paper:HermanJones}, and Herman and Ocneanu \cite{paper:HermanOcneanu} on the study of cyclic and finite group actions on UHF-algebras. Automorphisms of UHF, A$\mathbb{T}$- and Cuntz-algebras \cite{paper:Kishimoto2,paper:Kishimoto1,paper:Kishimoto3} were the subject of later research by Kishimoto. For Kirchberg algebras, Nakamura proved in \cite{paper:NakamuraAperiodic} that aperiodic automorphisms are completely classified by their $KK$-classes up to $KK$-trivial cocycle conjugacy. Building on that the second named author together with Matui developed a classification of a large class of $\Z^2$-actions on Kirchberg algebras \cite{paper:IzumiMatuiZ2}, which they later extended to Poly-$\Z$-group actions~\cite{IZMA21,paper:IzumiMatuiPolyZ2}. Together with Goldstein the second named author also studied quasi-free actions of finite groups on Cuntz algebras \cite{paper:GoldsteinIzumi} and finite group actions on $C^*$-algebras with the Rokhlin property \cite{paper:IzumiRokhlin1,paper:IzumiRokhlin2}.

 This conjecture was recently resolved affirmatively by combining the dynamical Kirchberg-Phillips theorem of Gabe and Szabo \cite[Theorem 6.2]{paper:GabeSzabo} with Meyer's work \cite[Theorem 3.10]{paper:MeyerKK}. Therefore confirming the fundamental importance of the homotopy type of $B\Aut{A}$ in understanding the structure of group actions on $A$.

In the case that $A = D \otimes \bK$ for a strongly self-absorbing Kirchberg algebra $D$, the space $\Aut{A}$ is an infinite loop space by the main result in \cite{DAPE16}, which implies that the right hand side of \eqref{eqn:OABGamma} is the first group in a generalised cohomology theory evaluated at $B\Gamma$. The group structure arises from the fibrewise tensor product of $C^*$-algebra bundles with fibre $D \otimes \bK$. 
%Dadarlat and the last named author have proven in \cite[Thm.~2.18]{DAPE16} that $\Aut{D \otimes \bK}$ has infinitely many non-zero homotopy groups and therefore cannot be an Eilenberg-MacLane space. 
The cohomology theory $E^*_D(X)$ with $E^1_D(X) = [X, B\Aut{D \otimes \bK}]$ obtained from the infinite loop space structure on $\Aut{D\otimes \bK}$, was identified in \cite{paper:DP-Units} as the one associated to the unit spectrum $gl_1(KU^D)$ of $K$-theory with coefficients in $K_0(D)$. This shows that even in the case of group actions on Kirchberg algebras the classification invariants have a much more intricate structure than in the case of von Neumann algebras. 

Further evidence of this was provided in \cite{paper:EvingtonPacheco} where Evington and the first named author studied $\Gamma$-kernels (called anomalous actions in that paper) on $C^*$-algebras classified by the Elliott programme (hereinafter called classifiable). They showed that all $\Gamma$-kernels on the Jiang-Su algebra $\mathcal{Z}$ have vanishing lifting obstruction and therefore arise from cocycle actions. Moreover, the lifting obstruction satisfies strong divisibility restrictions in the case of UHF-algebras by \cite[Theorem B]{paper:EvingtonPacheco}. Since these algebras can be seen as analogues of $\cR$ in the category of nuclear $C^*$-algebras, these results show that there is no straightforward generalisation of the von Neumann algebraic classification for $\Gamma$-kernels. 

Motivated by the work in \cite{paper:EvingtonPacheco} the second named author showed further restrictions to the possible values of lifting obstructions for $\Gamma$-kernels on classifiable C$^*$-algebras (\cite{IZ23}). For example showing that any $\Gamma$-kernel on $\cO_\infty$ for a finite group $\Gamma$ has trivial lifting obstruction. Both of the results in \cite{paper:EvingtonPacheco} and \cite{IZ23} arise by introducing new invariants of $\Gamma$-kernels of K-theoretic flavour. The restrictions to the values of the lifting obstruction then arise due to the relations between these invariants and the lifting obstruction. Moreover, in \cite{IZ23} the second named author also classified $\Z^n$-kernels on strongly self-absorbing Kirchberg algebras by making use of his newly introduced invariant.

The goal of this paper is to develop a unified picture to handle $\Gamma$-kernels, group actions and cocycle actions up to their respective natural equivalence relations. This picture will be cohomological by nature and all known $\Gamma$-kernel invariants appear as boundary maps in exact sequences of cohomology sets. We draw on tools from homotopical algebra developed by Whitehead \cite[Def.~(2.1)]{paper:Whitehead2}: Every unital $C^*$-algebra $A$ has several topological crossed modules (see \cref{def:crossed_mod}) associated to it. A list is shown in the first two columns of \cref{tab:list_cros_mod}.
\begin{table}[ht]
    \centering
    \begin{tabular}{|c|c|c|}
         \hline
         \multicolumn{2}{|c|}{Crossed Module} & $H^1(\Gamma, \cG)$ \\[1mm]
         \hline
         & $1 \to \Aut{A}$ & $\{ \text{actions of $\Gamma$ on $A$} \}/\sim_c$ \\[1mm]  
         $\cG_A$ & $U(A) \to \Aut{A}$  & $\{\text{cocycle actions of  $\Gamma$ on $A$} \}/\sim_{\text{c.c}}$ \\[1mm]
         $P\cG_A$ & $PU(A) \to \Aut{A}$ & $\{\text{$\Gamma$-kernels on $A$} \}/\sim_{\text{c}}$  \\[1mm]
         $\wcG_A$ & $\widetilde{U}(A) \to \Aut{A}$ &  \\[1mm]
         $S\cG^\tau_A$ & $SU_\tau(A) \to \Autt{A}$ & \\[1mm]
         $P\cG^\tau_A$ & $PU(A) \to \Autt{A}$ & $\{\text{$\tau$-preserving $\Gamma$-kernels on $A$} \}/ \sim_{\tau.\text{c}}$ \\
         \hline
    \end{tabular}
    \vspace{2mm}
    \caption{\label{tab:list_cros_mod}\ Crossed modules associated to unital $C^*$-algebras, the last two require the existence of a trace $\tau$ on $A$.}
    \vspace{-6mm}
\end{table}
% TODO: Explain SU_tau briefly and mention kernel of det, topology, etc.
% explain Aut_tau(A)

Where, assuming that $A$ has connected unitary group, $\widetilde{U}(A)$ denotes the universal cover of $U(A)$ and $SU_\tau(A)$ coincides with the kernel of the de la Harpe--Skandalis determinant equipped with an exponential length topology in the spirit of \cite{CHRO23} (see Section \ref{sec:SUA}). The low-degree group cohomology of $\Gamma$ with coefficients in a $\Gamma$-module generalises to cohomology with coefficients in a crossed module (see \cref{def:cohomology}). For the crossed modules shown in \cref{tab:list_cros_mod}, the first cohomology often has an interpretation in terms of $C^*$-dynamics. Examples of this are shown in the final column. While ordinary group cohomology gives abelian groups, it should be pointed out that $H^1(\Gamma, \cG)$ is in general only a pointed set. However, studying the interplay between the crossed modules leads to exact sequences of pointed sets that encode lifting obstructions. If $A$ is a unital $C^*$-algebra with $Z(A) = \C$ then the quotient map $\cG_A \to P\cG_A$ produces
\[
\begin{tikzcd}
    \ar[r] & H^2(\Gamma, \bT) \ar[r] & H^1(\Gamma, \cG_A) \ar[r] & H^1(\Gamma, P\cG_A) \ar[r,"\ob"] & H^3(\Gamma, \bT)\ ,
\end{tikzcd}
\]
where $\ob$ indeed coincides with the lifting obstruction from $\Gamma$-kernels to cocycle actions. Following this example, we recover in our first main result \cref{thm:operatoralgcrossedmoduleseq} not only the standard lifting obstruction outlined above, but also the one from \cite[Section 3]{IZ23} and the tracial lifting obstruction from \cite{paper:EvingtonPacheco}. 

Crossed modules are inherently 2-categorical objects. We will recall in \cref{sec:ClassifyingSpaces} how to associate a strict topological $2$-category with a single object and invertible $1$- and $2$-morphisms to a topological crossed module. These structures are known under the name topological $2$-groups \cite{paper:BaezStevenson} or as topological categorical groups in the category theory literature. With $2$-morphisms at hand the natural maps between $2$-categories turn out to be pseudofunctors. These do not preserve composition on the nose, but only up to a natural transformation (sometimes called compositor). If we view a discrete group $\Gamma$ as a $2$-category $\cC_\Gamma$ with a single object, $\Gamma$ as $1$-morphisms and trivial $2$-morphisms, then pseudofunctors $\cC_\Gamma \to \cG_A$ are exactly cocycle actions. Naturally isomorphic pseudofunctors correspond to cocycle conjugate actions. 

Just as ordinary groups, $2$-groups have classifying spaces. However, unlike in the $1$-categorical setting, for a crossed module $\cG$ there are several equally natural candidates for $B\cG$, two of which we study in this paper. The first one is based on the Duskin nerve $N_\ast^D(\cG)$, a simplicial space constructed from pseudofunctors $[n] \to \cG$, where $[n]$ denotes the poset category associated to the ordered set $\{0,\dots, n\}$. Its (fat) geometric realisation gives $B^D\cG$ (see \cref{def:Duskin_nerve}). The Duskin nerve of a $1$-category reproduces its ordinary nerve. We show in \cref{lem:map_on_BG} that functoriality with respect to pseudofunctors provides a transformation of pointed sets
\begin{equation}\label{eqn:intromap}
    H^1(\Gamma, \cG) \to [B\Gamma, B^D\cG]
\end{equation}
that is natural in $\Gamma$ and in $\cG$. We may also view the $2$-category associated to $\cG$ as a strict monoidal $1$-category $\cG_\otimes$, in which every object has an inverse. If the topology on $\cG$ is not too pathological, then the (thin) geometric realisation of the nerve $N_\ast\cG_\otimes$ is a monoid. We define the monoidal classifying space $B^\otimes \cG$ to be its geometric realisation (see \cref{def:monoidal_BG}). This gives another space that could be called the classifying space of $\cG$. For discrete $2$-categories it was shown in \cite{BUCE03} that $B^D\cG \simeq B^\otimes \cG$. It is straightforward to generalise their proof to the topological case; we do so in \cref{appendix}.

While  $B^D\cG$ is directly related to the cohomology set $H^1(\Gamma, \cG)$, the weak equivalence with $B^\otimes \cG$ allows us to determine its weak homotopy type. Combining the main result in \cref{appendix} with \cref{thm:weakhomeq} we obtain weak homotopy equivalences
\[
    B^D\cG_A \simeq B^\otimes\cG_A \simeq \cB\Autz{A \otimes \bK}\ .
\]
The group $\Autz{A \otimes \bK}$ is the group of the automorphism $\alpha$ such that the projections $\alpha(1\otimes e)$ and $1\otimes e$ are homotopic for a rank one projection $e\in \bK$ (this coincides with the connected component of the identity in $\Aut{A\otimes \bK}$ when $A$ is strongly self-absorbing). The right-hand side above denotes the classifying space of principal $\Autz{A\otimes \bK}$-bundles, obtained as the fat geometric realisation of the bar construction. For strongly self-absorbing $C^*$-algebras~$A$ the set $[B\Gamma, B^D \cG_A]$ has a natural group structure and the transformation $H^1_{ff}(\Gamma, \cG_A) \to [B\Gamma, B^D\cG_A]$ is a group isomorphism as we show in \cref{thm:BG_group_hom}, where $H^1_{ff}(\Gamma, \cG_A)$ is the restriction to fully faithful cocycles (see Definition \ref{def:ffcohomology}). 

We speculate that for strongly self-absorbing $C^*$-algebras $A$ the spaces $B^D\cG_A$, $B^DP\cG_A$, $B^D\wcG_A$, $B^DS\cG^\tau_A$ and $B^DP\cG^\tau_A$ will all be infinite loop spaces and that the transformation \eqref{eqn:intromap} will be a group homomorphism for each of these cases. Their infinite loop space structure would lead to interrelated cohomology theories that provide receptacles for topological invariants of $\Gamma$-kernels on $A$. 

\par The paper is structured as follows. In Section \ref{sec:prelim} we recall the notion of $\Gamma$-kernels, their invariants and the classifying space construction for topological groups. In Section \ref{sec:SUA} we introduce the special unitary group associated to each trace and show its relevant properties. In Section \ref{sec:crossedmodules} we introduce the operator algebraic crossed modules and show that every invariant of $\Gamma$-kernels arises as boundary maps in exact sequences of their cohomology sets. In Section \ref{sec:ClassifyingSpaces} we construct the map \eqref{eqn:intromap} and show the weak homotopy equivalence $B^{\otimes}\cG_A\simeq \cB\Autz{A\otimes\bK}$ along with its implications in the strongly self-absorbing case. In Section \ref{sec:covSU} we revisit $SU_\tau(A)$ and show that, under some structural assumptions on $A$, $\covU{A}\cong \R\times SU_\tau(A)$.\footnote{This is analogous to the well-known isomorphism $\covU{n}\cong \R\times SU(n)$.} Finally, in Appendix \ref{appendix} we show the weak homotopy equivalence $B^D\cG\simeq B^{\otimes}\cG$.
% TODO: Outlook?
% - future problems
\subsection*{Acknowledgements}
The authors' would like to thank Christian Voigt for interesting discussions. Part of this work was completed during the authors' stay in the International Centre for Mathematical Sciences (ICMS) for the \emph{Twinned Conference on $C^*$-Algebras and Tensor Categories} in November 2023. We thank ICMS and the organisers for the hospitality. The second and third author would also like to thank the Isaac Newton Institute for Mathematical Sciences, Cambridge, for support and hospitality during the programme \emph{Topological groupoids and their $C^*$-algebras}, where work on this paper was undertaken. This work was supported by EPSRC grant EP/V521929/1.

\par MI was supported by JSPS KAKENHI Grant Number JP25K00912. SGP was supported by projects G085020N and 1249225N funded by the Research Foundation Flanders (FWO). %\UP{funding acknowledgements Ulrich}
\section{Preliminaries}\label{sec:prelim}
Throughout this paper $A$ will be a unital, separable C$^*$-algebra with $A^{sa}$ its real subspace of self-adjoint elements and $Z(A)$ its centre. We denote its unitary group by $U(A)$ and its set of projections by $P(A)$. The connected component of the identity in $U(A)$ will be denoted by $U(A)_0$. We denote the automorphism group of $A$ by $\Aut A$ and our convention is that $\Ad \colon U(A)\rightarrow \Aut A $ is the group homomorphism defined by $\Ad_u(a)=uau^*$ for $a\in A$. Any automorphism of $A$ in the image of $\Ad$ is called \emph{inner} and forms a normal subgroup of $\Aut{A}$ denoted by $\Inn{A}$. We denote the quotient group $\Aut{A}/\Inn{A}$ by $\Out{A}$. We denote the space of tracial states of $A$ by $T(A)$. Whenever $\tau\in T(A)$ we will denote by $\Autt A$ the subgroup of $\Aut A$ of automorphisms that fix $\tau$.
\par Let $\Top$ be the category of compactly generated and weakly Hausdorff spaces. All topological spaces within this paper will be objects in $\Top$ and products will be formed within this category. For a group $G$ we will denote its centre by $ZG$ i.e. $ZG=\{g\in G: gh=hg\ \text{for all}\ h\in G\} $. If $(X,x_0)$ is a based topological space we denote by $\Omega_{x_0} X$ the space of continuous loops based at $x_0$. If $G$ is a topological group then $(G,1)$ is a based topological space and we denote $\Omega_1G$ simply by $\Omega G$, this is also a topological group under the pointwise operations. 
\par We define a \emph{short exact sequence of topological groups} to be an exact sequence of groups
\[
1\rightarrow H\rightarrow G\rightarrow K\rightarrow 1
\]
such that the connecting maps are continous and the map $G\rightarrow K$ is a Hurewicz fibration.
\subsection{The Skandalis de la Harpe determinants}
We recall the de la Harpe--Skandalis determinants; the reader is referred to \cite{SkdlH} for details. Let $A$ be a unital C$^*$-algebra. We denote by $U_n(A)=U(M_n(A))$ and by $U_{\infty}(A)=\bigcup_{n\in \N} U_n(A)$ where we view $U_n(A)\subset U_{n+1}(A)$ under the embedding $u\mapsto u\oplus 1_A$. We topologise the group $U_{\infty}(A)$ with the direct limit topology. We warn the reader that $U_{\infty}(A)$ endowed with this topology may not be a topological group. However, note that every compact subset $K\subset U_\infty(A)$ is contained in $U_n(A)$ for some $n\in \N$ (\cite[Lemma 1.7]{GLO05}) and so continuous paths $[0,1]\rightarrow U_{\infty}(A)$ and homotopies $[0,1]^2\rightarrow U_{\infty}(A)$ factor through $U_n(A)$ for some $n\in\N$.  We will denote by $U_{\infty}^{(0)}(A)$ the connnected component of $1_A$ in $U_{\infty}(A)$. For $\lambda\in \R$ we will often denote the path $e^{2\pi i\lambda t}$ for $t\in [0,1]$ by $e_{\lambda}$.
\par We denote by $C^{1}_{*}([0,1],U_{\infty}(A))$ the group of continuously differentiable paths $\gamma:[0,1]\rightarrow U_{\infty}(A)$  with $\gamma(0)=1_A$. Let $\tau:A\rightarrow E$ be a tracial continuous linear map into a Banach space $E$. Precisely $\tau$ is a bounded linear map and $\tau(ab)=\tau(ba)$ for all $a,b\in A$. In \cite{SkdlH} de la Harpe and Skandalis define a mapping $\tilde{\Delta}_\tau:C^{1}_{*}([0,1],U_{\infty}(A))\rightarrow E$ by
\begin{equation*}
    \tilde{\Delta}_\tau(\gamma)=\frac{1}{2\pi i}\int_{0}^1\tau(\gamma(t)^{-1}\gamma'(t))dt.\footnote{the trace $\tau$ is extended to $A= (a_{ij})\in M_n(A)$ by $\tau(A)=\sum_{k=1}^n\tau(a_{kk})$.}
\end{equation*}
In \cite[Lemme 1]{SkdlH} it is shown that $\tilde{\Delta}_\tau$ is a homomorphism of groups and that $\tilde{\Delta}_\tau(\gamma)=\tilde{\Delta}_\tau(f)$ for any two paths $f$ and $\gamma$ that are path homotopic to one another. Any two paths $f,\gamma\in C^{\infty}_{*}([0,1],U_{\infty}(A))$ with the same endpoint differ by a smooth loop in $U_{\infty}(A)$ based at $1_A$. Moreover, by Bott periodicity, any such loop is path homotopic to a loop of the form $\lambda_{p,q}(t)=e^{2\pi i tp}e^{-2\pi i tq}$ for some projections $p,q\in P_n(A):=P(M_n(A))$. Thus one defines the Skandalis de la Harpe determinant $\Delta_\tau: U_{\infty}^{(0)}(A)\rightarrow E/\tau_{*}(K_0(A))$ by picking for any $u\in U_{\infty}^{(0)}(A)$ a $C^{\infty}$ path $\xi$ with $\xi(0)=1_A$ and $\xi(1)=u$ and letting
\begin{equation*}
    \Delta_\tau(u)=\tilde{\Delta}_\tau(\xi)+\tau_{*}(K_0(A)).
\end{equation*}
Which is well defined group homomorphism by the preceding discussion. Note that when $u\in U^{(0)}_{\infty}(A)$ then $u=e^{ih_1}\ldots e^{ih_n}$ for $h_i\in M_n(A^{sa})$ for some large enough $n$. Choosing the path $u(t)=e^{ith_1}\ldots e^{ith_n}$ shows that $\Delta_\tau(u)=\frac{1}{2\pi}\sum_{i=1}^n\tau(h_i)$. We will be mainly interested in the case that $\tau:A\rightarrow \C$ is a tracial state. In this case, as $\tau(a^*)=\overline{\tau(a)}$ for any $a\in A$, the image of $\Delta_\tau$ is contained in $\R/\tau_{*}(K_0(A))$.
\subsection{$G$-kernels}
In this section we recall some invariants associated to $\Gamma$-kernels on C$^*$-algebras. We will refer to \cite{thesis:sergio} or \cite{IZ23} for more details on these constructions and the relevant proofs.
\par Let $\Gamma$ be a countable discrete group. A \emph{$\Gamma$-kernel} is a group homomorphism $\alpha:\Gamma\rightarrow \Out A$. To any $\Gamma$-kernel one can associate an invariant often called the \emph{lifting obstruction} of $\alpha$. We briefly recall this construction. Pick a set theoretic lifting $\tilde{\alpha}:\Gamma \rightarrow \Aut A $ of $\alpha$ and $u_{g,h}\in U(A)$ such that
\[
\tilde{\alpha}_g\tilde{\alpha}_h=\Ad_{u_{g,h}}\tilde{\alpha}_{gh}\ \text{for all}\ g,h\in \Gamma.
\]
We call the pair $(\tilde{\alpha},u)$ a \emph{lifting} of $\alpha$. It follows from the associativity of $\Aut A$ that 
\[
\tilde{\alpha}_{g}(u_{h,k})u_{g,hk}u_{gh,k}^*u_{g,h}^*=\omega_{g,h,k}\in \ker(\Ad)=ZU(A).
\]
One can show that $\omega$ satisfies the $3$-cocycle relation and that the cohomology class $[\omega]\in H^3(\Gamma,ZU(A))$ is independent of the choices of $\tilde{\alpha}$ and $u$. Where $ZU(A)$ is equipped with the $\Gamma$-module structure 
\[
g\cdot a=\tilde{\alpha}_g(a)
\]
for all $a\in ZU(A)$ and $g\in \Gamma$.
\begin{definition}\label{def:obalpha}
    Let $\alpha:\Gamma\rightarrow \Out A$ be a $\Gamma$-kernel with lifting $(\tilde{\alpha},u)$. We denote by $\ob(\alpha)$ the class in $H^3(\Gamma,ZU(A))$ defined by $\tilde{\alpha}_{g}(u_{h,k})u_{g,hk}u_{gh,k}^*u_{g,h}^*$ which is independent of the choice of lifting.
\end{definition}
\par In \cite{IZ23} the second named author introduces another two invariants of similar flavour. Let's assume for the remainder of this section that the unitary group of $A$ is path connected, i.e. $U(A)=U(A)_0$. Consider the group
\[
K_0^\#(A):=\{\gamma\in C([0,1],U(A))\ |\ \gamma(0)=1,\ \gamma(1)\in ZU(A)\}/\Omega_0 U(A)
\]
where $\Omega_0 U(A)$ is the subgroup of $C([0,1],U(A))$ consisting of loops in $U(A)$ based at $1_A$ that are path homotopic to the constant loop. The group $K_0^\#(A)$ is abelian under pointwise multiplication. Moreover, any $\Gamma$-kernel $\alpha:\Gamma\rightarrow \Out A $ induces a module structure on $K_0^\#(A)$ by choosing a lifting $\tilde{\alpha}:\Gamma\rightarrow \Aut A$ and pointwise evaluating $\tilde{\alpha}$ on an equivalence class of a path $\gamma\in K_0^\#(A)$ (see \cite[Section 5.2]{thesis:sergio}).
\begin{definition}\label{def:tildeobalpha}{(cf\ \cite[Definition 3.5]{IZ23})}
    Let $\alpha:\Gamma\rightarrow \Out A$ be a $\Gamma$-kernel with lifting $(\tilde{\alpha},u)$ and $\tilde{u}_{g,h}$ a choice of continuous paths starting at $1$ and ending at $u_{g,h}$ for each $g,h\in \Gamma$. We denote by $\widetilde{\ob}(\alpha)$ the class in $H^3(\Gamma,K_0^\#(A))$ defined by $\tilde{\alpha}_{g}(\tilde{u}_{h,k})\tilde{u}_{g,hk}\tilde{u}_{gh,k}^*\tilde{u}_{g,h}^*$. This is independent of the choice of lifting and path $\tilde{u}$.
\end{definition}
\begin{remark}
There is a precursor to $\widetilde{\ob}$ in the setting of cocycle actions. For a discrete group $\Gamma$, a pair $\alpha:\Gamma\rightarrow \Aut A$ and $u:\Gamma\times \Gamma\rightarrow U(A)$ such that 
\begin{align}
\alpha_g\alpha_h =\Ad_{u_{g,h}}\alpha_{gh}\\
\alpha_g(u_{h,k})u_{g,hk}=u_{g,h}u_{gh,k}\label{eqn:cocycleaction}
\end{align}
is called a \emph{cocycle action}.\footnote{Note that this is simply a lift of a $\Gamma$-kernel with trivial associated $3$-cocycle.} In \cite{IZMA21} the authors introduce an invariant $\kappa^3(\alpha,u)\in H^3(\Gamma,\pi_1(U(A)))$ by picking continuous paths $\tilde{u}_{g,h}$ starting at $1$ and ending at $u_{g,h}$ and setting $\kappa^3(\alpha,u)=\tilde{\alpha}_{g}(\tilde{u}_{h,k})\tilde{u}_{g,hk}\tilde{u}_{gh,k}^*\tilde{u}_{g,h}^*$ which is now valued in $\pi_1(U(A)))$ by \eqref{eqn:cocycleaction}.
\end{remark}
The natural notion of equivalence for $\Gamma$-kernels is that of conjugacy. We say two $\Gamma$-kernels $\alpha,\beta:\Gamma\rightarrow \Out A$ are \emph{conjugate} if there exists an automorphism $\varphi\in \Aut A$ such that $\varphi\alpha_g\varphi^{-1}=\beta_g\in \Out A$ for all $g\in G$, denoted $\alpha\sim_{c}\beta$. However, cocycle actions are usually considered up to cocycle conjugacy. Two cocycle morphisms $(\alpha,u)$ and $(\beta,v)$ are said to be \emph{cocycle conjugate}, denoted $(\alpha,u)\sim_{c.c} (\beta,v)$, if there exists an automorphism $\varphi\in \Aut A$ and a map $w:\Gamma\rightarrow U(A)$ such that
\begin{align}
    \Ad_{w_g}\varphi\alpha_g\varphi^{-1}&=\beta_g,\\
    w_{gh}\varphi(u_{g,h})w_g^*\beta_g(w_h^*)&=v_{g,h}
\end{align}
for all $g,h\in \Gamma$.
\par For the final invariant let $\tau\in  T(A)$. If $\alpha:\Gamma\rightarrow \Out A $ is a $\Gamma$-kernel on $A$ that preserves $\tau$, then one can always choose a lifting $(\tilde{\alpha},u)$ such that $\Delta_\tau(u)=0$; one may simply replace $u$ by $ue^{-2\pi\Delta_{\tau}(u)i}$ for any arbitrary lifting. For such a choice $\tilde{\alpha}_{g}(u_{h,k})u_{g,hk}u_{gh,k}^*u_{g,h}^*$ is now valued in $\ker(\Delta_\tau)$.
\begin{definition}
     Let $\tau\in T(A)$ and $\alpha:\Gamma\rightarrow \Out A$ a $\Gamma$-kernel that preserves $\tau$. Let $(\tilde{\alpha},u)$ be a lifting of $\alpha$ with $u\in \ker(\Delta_\tau)$. We denote by $\ob_\tau(\alpha)$ the class in $H^3(\Gamma,Z\ker(\Delta_\tau))$ defined by $\tilde{\alpha}_{g}(u_{h,k})u_{g,hk}u_{gh,k}^*u_{g,h}^*$ which is independent of the choice of lifting.\footnote{That $\ob_\tau$ is well-defined is only shown in \cite{IZ23} in the case that $ZU(A)=\bT$ but the proof follows in precisely the same way without this assumption.}
\end{definition}
Two $\tau$-preserving $\Gamma$-kernels are said to be $\tau$-conjugate (denoted $\sim_{\tau.c}$) if they are conjugate by an automorphism that fixes $\tau$.

\par Note that the above invariants may not be preserved up to conjugacy. For example, if $\alpha:\Gamma\rightarrow \Out A $ is a $\Gamma$-kernel and $\varphi\in \Aut A $ then $\ob(\varphi\alpha\varphi^{-1})=\varphi^*(\ob(\alpha))$ where 
    \begin{equation}\label{eqn: inducedmap}
    \varphi^*:H^3(\Gamma,ZU(A))\rightarrow H^3(\Gamma,\varphi^*(ZU(A)))
    \end{equation}
    
    is the induced map in cohomology. To introduce a well defined invariant for conjugacy classes of $\Gamma$-kernels we consider the direct sum $\bigoplus H^3(\Gamma,ZU(A))$ over all $\Gamma$-module structures on $Z(U(A))$ descending from $\Gamma$-kernels on $A$ and let
    \[\bigoplus_{\sim \Out A} H^3(\Gamma, ZU(A))
    \]the quotient by the canonical action of $\Out A$ on $\bigoplus H^3(\Gamma, ZU(A))$ given by \eqref{eqn: inducedmap}. The invariant $\ob$ now descends to a well defined mapping
    \begin{equation}\label{eqn: obmap}
    \{\Gamma\text{-kernels}\}/\sim_{c}\xrightarrow{\ob} \bigoplus_{\sim \Out A} H^3(\Gamma, ZU(A)).
    \end{equation}
     Similarly we get well defined maps
     \begin{align}
    \{\Gamma\text{-kernels}\}/\sim_{ c}&\xrightarrow{\widetilde{\ob}} \bigoplus_{\sim \Out A} H^3(\Gamma, K_{0}^\#(A))\label{eqn: tildeobmap}\\
    \{\tau\text{-preserving}\ \Gamma\text{-kernels}\}/\sim_{ \tau. c.}&\xrightarrow{\ob_\tau} \bigoplus_{\sim \Outt A} H^3(\Gamma, Z\ker\Delta_\tau). \label{eqn: tauobmap}
    \end{align}
    \begin{remark}\label{rmk:simplificationsobs}
    We will mostly care about the case that $A$ has centre $\C$, then the right hand sides of \eqref{eqn: obmap} and \eqref{eqn: tauobmap} are simply identified with $H^3(\Gamma,\bT)$ and $H^3(\Gamma,\tau(K_0(A))/\Z)$ with the coefficient modules carrying the trivial $\Gamma$-module structures. Also, when $A$ has a unique trace $\tau$, every $\Gamma$-kernel is automatically $\tau$-preserving. Note though that even when $A$ has trivial centre the induced module structure on $K_0^{\#}(A)$ may still be non-trivial. By \cite[Proposition 5.1.3]{thesis:sergio} (see also \cite[Lemma 3.4]{IZ23}) there is an isomorphism of $\Gamma$-modules
    \[
    K_0^{\#}(A)\cong \frac{\R \oplus \pi_1(U(A))}{\Z(-1,[e_1])}.
    \]
    If all automorphisms of $A$ are approximately inner, then the induced action on $\pi_1(U(A))$, and hence on $K_0^{\#}(A)$, is trivial for any $\Gamma$-kernel. Indeed, for an automorphism $\alpha\in \Aut A $ and $\tilde{u}\in \Omega U(A)$, there exists a unitary $v\in U(A)$ such that $\|\alpha(\tilde{u})-v\tilde{u}v^*\|<1$. Thus, $\alpha(\tilde{u})$ and $v\tilde{u}v^*$ are homotopic in $\Omega U(A)$. Also, as $v\in U(A)_{0}$, then the loop $v\tilde{u}v^*$ is homotopic to $\tilde{u}$ in $\Omega U(A)$. Thus all automorphism induce the trivial map on $\pi_1(U(A))$. In this case, the right hand side of \eqref{eqn: tildeobmap} reduces to $H^3(\Gamma, K_0^{\#}(A))$ with the trivial $\Gamma$-module structure. The assumptions that $A$ has trivial centre and that all its automorphisms are approximately inner are satisfied for all strongly self-absorbing C$^*$-algebras. This is the class of C$^*$-algebras predominantly considered in \cite{IZ23}. 
    \end{remark}

\subsection{Categories, simplicial spaces and classifying spaces}\label{subsec:simplicalspaces}
%\begin{itemize}
    %\item \nn{What is a simplicial space}
    %\item \nn{What is the nerve of a category}
    %\item \nn{What is the fat and thin geometric realization of a simplicial space}
%\end{itemize}
In this section we will recall how to turn a topological category into a simplicial space via the nerve construction. The geometric realisation of the nerve is known as the classifying space of the category. In case of a discrete group, considered as a one-object category, the cohomology of the classifying space is isomorphic to the algebraically defined group cohomology. 
\begin{definition} \label{def:simp_cat}
    For $n \in \N_0$ let $[n] = \{0, \dots, n\}$ considered as a linearly ordered set. The sets $[n]$ are the objects of the \emph{simplex category} $\Delta$. Its morphisms are the monotone functions $[n] \to [m]$. We will also need the subcategory $\Deltainj \subset \Delta$, which has the same objects, but whose morphisms are the injective monotone maps.
\end{definition}
The morphisms in $\Delta$ are generated by the coface maps $\partial_i \colon [n-1] \to [n]$ and $\sigma_i \colon [n+1] \to [n]$ with
\[
    \partial_i(k) = \begin{cases}
        k & \text{if } k < i \ ,\\
        k+1 & \text{if } k \geq i
    \end{cases}
    \quad \text{and} \quad
    \sigma_i(k) = \begin{cases}
        k & \text{if } k \leq i \ , \\
        k-1 & \text{if } k > i\ .
    \end{cases}
\]
Denote by $\Top$ the category of compactly generated Hausdorff spaces. 
\begin{definition} \label{def:simp_space}
    A \emph{simplicial space} $X$ is a presheaf on $\Delta$ with values in $\Top$, i.e.\ it is a functor $X \colon \Delta^{\rm op} \to \Top$. We will often denote $X([n])$ by $X_n$ and the simplicial space $X$ by $X_\ast$. Moreover, $d_i = X(\partial_i) \colon X_n \to X_{n-1}$ are called the \emph{face maps} of $X_\ast$ and $s_i = X(\sigma_i) \colon X_n \to X_{n+1}$ are the \emph{degeneracy maps}. A \emph{simplicial map} $f \colon X_\ast \to Y_\ast$ of simplicial spaces $X_\ast$ and $Y_\ast$ is a continuous natural transformation between the functors $X,Y \colon \Delta^{\rm op} \to \Top$. We denote the category consisting of simplicial spaces and simplicial maps between them $\Top^{\Delta^{\rm op}}$.
\end{definition}
Since all morphisms in $\Delta$ are compositions of $\partial_i$'s and $\sigma_i$'s, the functor~$X$ is on morphisms completely determined by defining what it associates to those. Diagrammatically, this can be viewed as follows:
\[
\begin{tikzcd}
     \dots \ar[r,yshift=15pt] \ar[r,yshift=5pt] \ar[r,yshift=-5pt] \ar[r,yshift=-15pt] & \ar[l,yshift=-10pt] \ar[l,yshift=0pt] \ar[l,yshift=10pt] X_2 \ar[r,yshift=10pt] \ar[r,yshift=0pt] \ar[r,yshift=-10pt] & X_1  \ar[r,yshift=5pt] \ar[r,yshift=-5pt] \ar[l,yshift=5pt] \ar[l,yshift=-5pt] & \ar[l] X_0
\end{tikzcd}
\]
with face maps pointing to the right and degeneracies pointing to the left. If each $X_n$ is a discrete set, then it should be understood as the set of $n$-simplices. The $i$th face map gives the $i$th face of the simplex as an $(n-1)$-simplex. This ``blueprint'' for a topological space can then be turned into a real one by the geometric realisation.

Recall that the standard $p$-simplex $\Delta^p$ is the following subspace of $\R^{p+1}$:
\[
    \Delta^p = \left\{ (t_0, \dots, t_p) \in \R^{p+1}\ |\  \sum_{i=0}^p t_i = 1\text{ and } t_i \geq 0 \text{ for all } i \right\}
\]
Each morphism $\varphi \in \hom_{\Delta}([p], [q])$ induces a continuous map $\varphi_* \colon \Delta^p \to \Delta^q$ defined by 
\[
    \varphi_*(t_0, \dots, t_p) = (s_0, \dots, s_q) \quad \text{with} \quad s_j = \sum_{i \in \varphi^{-1}(j)} t_i\ .
\]
In particular, $(\partial_i)_*$ enters a $0$ in the $i$th coordinate and $(\sigma_i)_*$ adds the $i^{\text{th}}$ and $(i+1)^{\text{st}}$ coordinate.

\begin{definition} \label{def:geom_real}
The \emph{(thin) geometric realisation} of a simplicial space $X_\ast$ is the quotient space
\[
    \lvert X_\ast \rvert = \left( \coprod_{n} X_n \times \Delta^n \right)/\!\!\sim\ ,
\]
where $\sim$ is the equivalence relation generated by $(x, \varphi_*(t)) \sim (X(\varphi)(x), t)$ for all morphisms $\varphi$ in $\Delta$. The \emph{fat geometric realisation} is defined as 
\[
    \lVert X_\ast \rVert = \left( \coprod_{n} X_n \times \Delta^n \right)/\!\!\sim_f\ ,
\]
where $\sim_f$ is the equivalence relation generated by $(x, \varphi_*(t)) \sim_f (X(\varphi)(x), t)$, but only for the morphisms $\varphi$ in $\Deltainj$ (i.e.\ ignoring the degeneracy maps).
\end{definition}
Both of these constructions are functorial with respect to simplicial maps and therefore provide functors $\Top^{\Delta^{\rm op}} \to \Top$ from the category of simplicial spaces to topological spaces. We refer the reader to \cite[Section 11]{book:MayIterated}, \cite[Section 1.2]{EBWI19} and \cite[Appendix A]{SE74} for references about the geometric realisation functors.

While the fat geometric realisation often has the homotopy type we are after, the thin one is better behaved from a categorical viewpoint. For example, for the constant simplicial space (i.e.\ $X_n = Y$ for all $n \in \N_0$ and all structure maps equal to identities) we have 
\[
    \lvert X_\ast \rvert \cong Y \quad , \quad \lVert X_\ast \rVert \cong Y \times \Delta^\infty\ ,
\]
where $\Delta^\infty$ is an infinite-dimensional simplex. Moreover, the thin geometric realisation preserves products by \cite[Theorem 11.5]{book:MayIterated}, while this is true for the fat one only up to weak equivalence \cite[Theorem 7.2]{EBWI19}. An important consequence is that the thin geometric realisation of a simplicial topological group is again a topological group \cite[Corollary 11.7]{book:MayIterated}. The following cofibrancy conditions appear in the literature \cite[Definition A.4]{SE74}, :
\begin{definition} \label{def:good_proper}
    A simplicial space $X_\ast$ is called \emph{good} if all degeneracy maps $X_n \to X_{n+1}$ are closed cofibrations. It is called \emph{proper} (or \emph{Reedy cofibrant} if the inclusion $sX_n \to X_n$ with $sX_n = \bigcup_i X(\sigma_i)(X_{n-1})$ is a closed cofibration.
\end{definition}
If the simplicial space $X_\ast$ is good, then the quotient map
\(
    \lVert X_\ast \rVert \to \lvert X_\ast \rvert
\)
is an equivalence by \cite[Proposition A.1 (iv)]{SE74}. 

We will make extensive use of the following facts: 
If $f_\ast \colon X_\ast \to Y_\ast$ is a map of simplicial spaces such that $f_n \colon X_n \to Y_n$ is a weak homotopy equivalence for all $n \in \N_0$, then $\lVert f_\ast \rVert \colon \lVert X_\ast \rVert \to \lVert Y_\ast \rVert$ is a weak homotopy equivalence as well \cite[Theorem 2.2]{EBWI19}. Moreover, if $f_\ast$ is a map of proper simplicial spaces such that each $f_n$ is a homotopy equivalence, then $\lvert f_\ast \rvert$ is also a homotopy equivalence \cite[Theorem A.4]{MAY74}. In fact, a good simplicial space is also proper. This follows for example from \cite[Lemma A.6]{book:MayIterated}.

One of the main sources of simplicial spaces are topological categories:
\begin{definition} \label{def:nerve}
    Let $\cC$ be a topological category. The \emph{nerve} of $\cC$ is the simplicial space
    \[
        N_n\cC = \text{Fun}([n], \cC) \ ,
    \]
    where we consider $[n]$ as the poset category with a unique arrow $i \leftarrow j$ if and only if $i \leq j$.\footnote{The choice of direction for the arrows in the category was made in such a way that $(0 \leftarrow 1) \circ (1 \leftarrow 2) = (0 \leftarrow 2)$.} Each such functor is determined by a starting object and a chain of composable morphisms of length $n$. The topology is then fixed by considering $N_n\cC$ as a subspace of $\ob(\cC) \times \text{Mor}(\cC)^n$.
\end{definition}

Let $\Gamma$ be a topological group. It gives rise to a topological category $\cC_\Gamma$ with the one-point space as the objects and the group as the morphism space. Consider the spaces
\[
    \cB\Gamma = \lVert N_\ast \cC_{\Gamma} \rVert \quad, \quad B\Gamma = \lvert N_\ast \cC_\Gamma \rvert  \ .
\]
The space $\cB\Gamma$ is called the \emph{classifying space} of $\Gamma$. By \cite[Proposition 2]{paper:tomDieckBG} and the preceding remark there, this space is the one that appeared as $B\Gamma_{\N}$ in \cite[page 107]{paper:SegalBG}, which is homeomorphic to Milnor's classifying space construction (\cite{paper:MilnorBG}, \cite[Section 14.4.3]{book:TomDieckTopology}). Therefore the homotopy set $[X,\cB\Gamma]$ is in bijection with the set of isomorphism classes of numerable principal $\Gamma$-bundles over $X$. The space $B\Gamma$ is a special case of the bar construction. It is homeomorphic to $B(\ast, \Gamma, \ast)$ \cite[Section 7]{book:MayBGAndFib}. The space $E\Gamma = B(\ast, \Gamma, \Gamma)$ comes with a natural quotient map to $B\Gamma$. In case the group $\Gamma$ is \emph{well-pointed} in the sense that the inclusion $\{1\} \to \Gamma$ is a cofibration, the map $E\Gamma \to B\Gamma$ is a numerable principal $\Gamma$-bundle with contractible total space \cite[Theorem 8.2]{book:MayBGAndFib}. In this case $\cB\Gamma \to B\Gamma$ is a weak homotopy equivalence and $B\Gamma$ is another model for the classifying space.

\section{The special unitary group associated to each trace}\label{sec:SUA}
In \cite{CHRO23} the authors study the special unitary group of a unital C$^*$-algebra defined by
\[
SU(A)=\{u\in U(A): u=e^{ih_1}e^{ih_2}\ldots e^{ih_n}\ \text{for}\ n\in\N, h_i\in A^{sa}\cap \ker(T)\},
\]
where $T:A\rightarrow A/\overline{[A,A]}$ is the quotient map (often called the \emph{universal trace}). The topology given by the basis of open sets at the identity
\[
\mathcal{V}_{\varepsilon}=\{e^{ih}: h\in A^{sa}\cap \ker(T),\ \|h\|<\varepsilon\}
\]
for $\varepsilon>0$ make $SU(A)$ a Banach Lie subgroup of $U(A)_0$ associated to the Lie algebra of skewadjoint commutators $iA^{sa}\cap\overline{[A,A]}$. In fact $SU(A)$ is a Polish group (see \cite[Section 5]{CHRO23}).
\par In this section we generalise this construction to any tracial continuous linear map $\tau:A\rightarrow E$ for some Banach space $E$. Therefore, for the remainder of this section, we denote by $\tau$ an arbitrary such tracial, continuous linear map. We will be particularly interested in the case that $\tau\in T(A)$, as this case will be required in the next section to observe the invariant $\ob_\tau$ in the language of crossed modules.
\begin{definition}
    Let $\tau:A\rightarrow E$ be a tracial continuous linear map. We denote by
\[
    SU_{\tau}(A)=\{u\in U(A): u=e^{ih_1}e^{ih_2}\ldots e^{ih_n}\ \text{for}\ n\in\N,\ h_i\in A^{sa}\cap \ker(\tau)\}
\]
 which is a subgroup of $U(A)_0$.
\end{definition}
If one equips $SU_\tau(A)$ with the topology given by the basis of open sets at the identity
\begin{equation*}
   \mathcal{U}_{\varepsilon}=\{e^{ih}: h\in A^{sa}\cap \ker(\tau),\ \|h\|<\varepsilon\}
\end{equation*}
for $\varepsilon>0$ the group $SU_{\tau}(A)$ is the Banach Lie subgroup of $U(A)_0$ associated to the Lie algebra $iA^{sa}\cap \ker(\tau)$. Similarly, as in \cite{CHRO23}, we will show that this topology of $SU_{\tau}(A)$ is metrisable and that $SU_{\tau}(A)$ is a Polish group.
\begin{lemma}\label{lem:expinequalities}
    Let $h_1,h_2\in A^{sa}$. Then 
    \begin{equation*}
        \|h_1-h_2\|\left(1-\frac{\|h_1\|+\|h_2\|}{2}\right)\leq \|e^{ih_1}-e^{ih_2}\|\leq \|h_1-h_2\|.
    \end{equation*}
\end{lemma}
\begin{proof}
    First we show the rightmost inequality. Let $f(t)=e^{ith_1}e^{-ith_2}$. Then
    \begin{align*}
        \|e^{ih_1}-e^{ih_2}\|=\|f(1)-f(0)\|&=\|\int_0^1 f'(t)dt\|\\
        &=\|\int_0^1 e^{ith_1}(h_1-h_2)e^{-ith_2}\|\\
        &\leq \int_0^1\|e^{ith_1}(h_1-h_2)e^{-ith_2}\|dt\\
        &=\int_0^1\|h_1-h_2\|dt\\
        &=\|h_1-h_2\|.
    \end{align*}
    Now for the leftmost inequality we may make use of the inequalities above along with standard trigonometric inequalities.
    \begin{align*}
\|e^{ih_1}-e^{ih_2}\|&=\|\int_0^1(e^{ith_1}-1)(h_1-h_2)e^{-ith_2}+(h_1-h_2)(e^{-ith_2}-1)+(h_1-h_2)dt\|\\
&\geq \|h_1-h_2\|-\|\int_0^1(h_1-h_2)(e^{-ith_2}-1)dt\|-\|\int_0^1 (e^{ith_1}-1)(h_1-h_2)dt\|\\
&\geq \|h_1-h_2\|\left(1-\int_0^1\|e^{-ith_2}-1\|dt-\int_0^1\|e^{ith_1}-1\|dt\right)\\
&= \|h_1-h_2\|\left(1-\int_0^12\left(\sin{\left(\frac{t\|h_2\|}{2}\right)}\right)dt-\int_0^1 2\left(\sin{\left(\frac{t\|h_1\|}{2}\right)}\right)dt\right)\\
&\geq \|h_1-h_2\|\left(1-\int_0^1t\|h_2\|dt-\int_0^1t\|h_1\|dt\right)\\
&=\|h_1-h_2\|\left(1-\frac{\|h_1\|+\|h_2\|}{2}\right).
\end{align*}
\end{proof}
\par Consider the exponential length function
\begin{align*}
\el_{\tau}:SU_{\tau}(A)&\rightarrow [0,\infty)\\
u&\mapsto \inf\left\{\sum_{j=1}^n\|h_j\|: h_j\in A^{sa}, \tau(h_j)=0, u=e^{ih_1}e^{ih_2}\ldots e^{ih_n}\right\}.
\end{align*}
Firstly, note that $d_{\tau}(u,v)=\el_{\tau}(u^*v)$ is a pseudometric on $SU_{\tau}(A)$ and that $\el_\tau(uvu^*)=\el_\tau(v)$ for all $u\in U(A)$ and $v\in SU_\tau(A)$ and hence $d_{\tau}$ is both left and right invariant. We proceed to show that $d_\tau$ is a metric.
\begin{lemma}\label{lemma:exponentillengthsmall}
Let $A$ be a unital C$^*$-algebra with a tracial continuous function $\tau:A\rightarrow E$ and $u\in SU_\tau(A)$ such that $el_\tau(u)<\pi$. Then $el_\tau(u)=\|\log (u)\|$ and $\tau(\log (u))=0$.
\end{lemma}
\begin{proof}
    Let $u=e^{ih_1}e^{ih_2}\ldots e^{ih_n}$ for $h_i\in A^{sa}\cap \ker(\tau)$. Firstly, as $\el_\tau(u)<\pi$ it follows from \cite[Corollary 2.2]{RI92} (see also \cite[Remark 2.3]{RI92}) that $\sigma(u)\subset \bT \setminus\{-1\}$ and that $\|\log(u)\|\leq \sum_{i=1}^n\|h_i\|$. Let 
    \[
    u(t)=e^{-t\log(e^{ih_1}e^{ih_2}\ldots e^{ih_n})}e^{ith_1}e^{ith_2}\ldots e^{ith_n}
    \]
    \par then $u(t)$ is a $C^{\infty}$ loop in $U(A)_0$ based on $1_A$. Let
    \[
    u_s(t)=e^{-t\log(e^{ish_1}e^{ish_2}\ldots e^{itsh_n})}e^{itsh_1}e^{itsh_2}\ldots e^{itsh_n}
    \]
    for $(s,t)\in [0,1]^2$. Then $u_s$ is a homotopy of based loops with $u_0=1$ and $u_1=u$. As $\tilde{\Delta}_\tau$ is equal on homotopic paths (see \cite[Lemme 1]{SkdlH})
    \[
    0=\tilde{\Delta}_\tau(1)=\tilde{\Delta}_\tau(u)=\frac{1}{2\pi}(\sum_{j=1}^n\tau(h_j))-\frac{1}{2\pi i}\tau(\log(u)).
    \]
    Therefore $\tau(\log(u))=0$. As $h_1,h_2,\ldots ,h_n$ are arbitrary $\el_\tau(u)=\|\log(u)\|$.
    \end{proof}
    It is immediate from Lemma \ref{lemma:exponentillengthsmall} that $d_\tau$ is a left and right invariant metric. In fact $d_\tau$ equips $SU_\tau(A)$ with the structure of a Polish group.
%\begin{corollary}
    %The mapping $d_\tau:SU_\tau(A)\times SU_\tau(A)\rightarrow [0,\infty)$ is a left and right invariant metric.
%\end{corollary}
%\begin{proof}
    %It is clear that $d_\tau$ is a left and right invariant pseudometric. Moreover, if $d_\tau(u,v)=\el_\tau(u^*v)=0$, then by Lemma \ref{lemma:exponentillengthsmall} it follows that $\log(u^*v)=0$ and hence $u=v$.
%\end{proof}
\begin{theorem}
    Let $A$ be a separable C$^*$-algebra, then $SU_\tau(A)$ is a Polish group.
\end{theorem}
\begin{proof}
This follows as the topology given by the basis of open sets at the identity $\mathcal{V}_{\varepsilon}=\{e^{ih}: h\in A^{sa}\cap \ker(\tau),\ \|h\|<\varepsilon\}$ coincides with that given by the metric $d_\tau$. Indeed, for any $\varepsilon<\pi$, Lemma \ref{lemma:exponentillengthsmall} implies that $\mathcal{V}_{\varepsilon}=\{u\in SU_\tau(A): el_\tau(u)<\epsilon\}=B_{d_\tau}(1,\varepsilon)$. That $SU_\tau(A)$ is complete with respect to the metric $d_\tau$ also follows from Lemma \ref{lemma:exponentillengthsmall}.
\end{proof}
By the following lemma the $d_\tau$ topology is stronger than the norm topology.
\begin{lemma}\label{lem:estimatenorms}
    For $0<\varepsilon<\frac{\pi}{2}$ and  for any $u,v\in \mathcal{V}_{\varepsilon}=\{e^{ih}: h\in A^{sa}\cap \ker(\tau),\ \|h\|<\varepsilon\}$ we have that
    \[
    (1-\varepsilon)\|\log(u)-\log(v)\|\leq \|u-v\|\leq d_\tau(u,v)\leq \frac{\pi}{2}\|u-v\|\leq \frac{\pi}{2}\|\log(u)-\log(v)\|.
    \]
\end{lemma}
\begin{proof}
    As $u,v \in \mathcal{V}_{\epsilon}$ one has that $u=e^{\log(u)}$ and $v=e^{\log(v)}$. Therefore the leftmost and rightmost inequalities follow immediately from Lemma \ref{lem:expinequalities}. We also have that $\el_\tau(u^*v)<\pi$ and hence by Lemma \ref{lemma:exponentillengthsmall} and making use of standard trigonometric inequalities one has that
    \begin{equation*}
     \|u-v\|=\|1-u^*v\|=\|1-e^{\log(u^*v)}\|=2\sin\left(\frac{d_\tau(u,v)}{2}\right)\leq d_\tau(v,v)
    \end{equation*}
    and
    \begin{equation*}
        d_\tau(u,v)=\pi\left(\frac{2}{\pi}\cdot \frac{d_\tau(u,v)}{2}\right)\leq \pi \left(\sin\left(\frac{d_\tau(u,v)}{2}\right)\right)=\frac{\pi}{2}\|u-v\|.
    \end{equation*}
\end{proof}
\par Observe that $SU_{\tau}(A)\subset \ker(\Delta_\tau)\cap U(A)_0$. We will care about the case when this inclusion is an equality. This is shown to be the case when $\tau$ is the universal trace and $A$ has stable rank one in \cite{CHRO23}. Following a similar argument, we will show the equality of this inclusion for any tracial state $\tau$ on any C$^*$-algebra $A$ that satisfies a $K$-theoretic condition which is implied whenever $A$ has stable rank one. First note that, as is the case for the determinant associated to the universal trace, the following holds.
\begin{lemma}[{cf.\ \cite[Lemma 5.1]{CHRO23}}]\label{lem:pathSUtau}
Let $\tau\in T(A)$. There exists a continuous path $\eta:[0,1]\rightarrow U(A)$ with $u(0)=1_A$ and $u(1)=u$ and $\tilde{\Delta}_\tau(\eta)=0$ if and only if $u\in SU_\tau(A)$.
\end{lemma}
\begin{proof}
        The if direction is clear. Choose a path $\eta$ as in the hypothesis. Following \cite[Lemma 5.1]{CHRO23} there exist $h_j\in A^{sa}$ for $1\leq j \leq n$ such that $u=e^{ih_1}e^{ih_2}\ldots e^{ih_n}$ and $\tau(\sum_{j=1}^n h_j)=0$. Then $u=e^{i(h_1-\tau(h_1))}\ldots e^{i(h_n-\tau(h_n))}$ exhibits $u\in SU_\tau(A)$.
\end{proof}
   By Bott periodicity the Bott map $K_0(A)\rightarrow \pi_1(U_{\infty}(A))$ is an isomorphism. By precomposing the inverse of the Bott map with the obvious mapping $\pi_1(U(A))\rightarrow \pi_1(U_\infty(A))$ we get a canonical map 
   \[
   B_A: \pi_1(U(A))\rightarrow K_0(A).
   \]
\begin{prop}\label{prop: kerdetcoincidewithSU}
    Let $\tau\in T(A)$ and suppose that the map $B_A:\pi_1(U(A))\rightarrow K_0(A)$ is surjective then $\ker(\Delta_\tau)\cap U(A)_0=SU_{\tau}(A)$.\footnote{The map $\pi_1(U(A))\rightarrow K_0(A)$ is surjective in many cases of interest. For example whenever $A$ has stable rank one (see \cite[Theorem 3.3]{RIE87}), when $A$ is Jiang-Su stability (see \cite[Theorem 3]{JIA97},\cite[Theorem 4.8]{Class1}) or when $K_0(A)$ is generated by the projections in $A$.} 
\end{prop}
\begin{proof}
    Let $u\in \ker(\Delta_\tau)\cap U(A)_0$. Then there exists $h_1,h_2,\ldots h_n\in A^{sa}$ such that $u=e^{ih_1}e^{ih_2}\ldots e^{ih_n}$. Let $l_j=h_j-\tau(h_j)$ for $1\leq j\leq n$, then 
    \[
    u=e^{il_1}e^{il_2}\ldots e^{il_n}e^{i\tau(\sum_{j=1}^n h_i)}.
    \]
    As $u\in \ker(\Delta_\tau)$ it follows that $\tau(\sum_{j=1}^n h_i)\in 2\pi \tau_{*}(K_0(A))$ and so there exists $p,q\in P_n(A)$ for some $n\in \N$ with $\tau(\sum_{j=1}^n h_j)=2\pi \tau(p)-2\pi \tau(q)$. Hence it suffices to show that $e^{2\pi i \tau(p)}\in SU_\tau(A)$ for any $p \in P_n(A)$ and $n\in \N$. If the canonical map $\pi_1(U(A))\rightarrow K_0(A)$ is surjective there exists a continuous loop $\eta:[0,1]\rightarrow U(A)_0$ that is homotopic to the loop $e^{2\pi itp}$. Therefore, the path $\xi(t)=e^{2\pi i t\tau(p)}\eta(t)^{-1}$ is a continuous path contained in $U(A)$ from $1$ to $e^{2\pi i \tau(p)}$. Then
    $\tilde{\Delta}_\tau(\xi)=\tilde{\Delta}_\tau(e^{2\pi i t\tau(p)})-\tilde{\Delta}_{\tau}(\eta)=0$. So by Lemma \ref{lem:pathSUtau} the unitary $e^{2\pi i \tau(p)}\in SU_\tau(A)$. 
    %If instead the second condition holds there exist $p_j,q_j\in P(A)$ for $1\leq j\leq m$ such that $\tau(\sum_{j=1}^n h_j)=\sum_{j=1}^n \tau(p_j)-\tau(q_j)$. Thus
    %\begin{align*}
    % u&=e^{il_1}\ldots e^{il_n}e^{2\pi i \tau(p_1)}\ldots e^{2\pi i \tau(p_m)}e^{-2\pi i \tau(q_1)}\ldots e^{-2\pi i \tau(q_m)}\\
     %&=e^{il_1}\ldots e^{il_n}e^{2\pi i (\tau(p_1)-p_1)}\ldots e^{2\pi i (\tau(p_m)-p_m)}e^{-2\pi i (\tau(q_1)-q_1)}\ldots e^{-2\pi i (\tau(q_m)-q_m)}
    % \end{align*}
     %as $l_j$ for $1\leq j\leq n$ and $p_k-\tau(p_k), q_k-\tau(q_k)$ for $1\leq k\leq m$ are contained in $A^{sa}\cap \ker(\tau)$ one has that $u\in SU_\tau(A)$.
\end{proof}
    \begin{remark}\label{rem:discrete}
        In most cases the $d_\tau$ topology is strictly stronger than the norm topology. For example, assuming that $\tau\in T(A)$ and that $\pi_1(U(A))\rightarrow K_0(A)$ is surjective, then $\ker(\Delta_\tau)\cap U(A)_0=SU_\tau(A)$ and so $e^{2\pi i \tau_{^*}(K_0(A))}$ is a subgroup of $SU_\tau(A)$. Moreover, for any non-zero $r\in \tau_{*}(K_0(A))/\Z$ the evaluation $\tau(\log(e^{2\pi i r}))\neq 0$. Thus it follows immediately from Lemma \ref{lemma:exponentillengthsmall} that $e^{2\pi i \tau_{^*}(K_0(A))}$ is a discrete subgroup of $SU_\tau(A)$. However, if $A$ is a UHF-algebra of infinite type, then the subgroup $e^{2\pi i \tau_{^*}(K_0(A))}$ is a countable dense subgroup of $\bT$ in the norm topology.
    \end{remark}
    Denote by $PU_0(A)$ the quotient topological group $U(A)_0/ZU(A)$. It is clear that the groups $SU_\tau(A)/ZSU_\tau(A)$ and $PU_0(A)$ are isomorphic as discrete groups. In fact, they are isomorphic as topological groups from the $d_\tau$ topology to the norm topology.
\begin{lemma}\label{lem:topologyproj}
    The norm topology on $PU_0(A)$ coincides with the $d_\tau$ topology after the canonical identification of $PU_0(A)$ with $SU_\tau(A)/ZSU_\tau(A)$.
\end{lemma}
\begin{proof}
    Let 
    \begin{align*}
    q_1&:U(A)_0\rightarrow PU_0(A),\\
    q_2&:SU_\tau(A)\rightarrow SU_\tau(A)/ZSU_\tau(A)
    \end{align*}
    be the quotient maps. As the $d_\tau$ topology is stronger than the norm topology, for any norm open $V\subset PU_0(A)$ one has that $q_2^{-1}(V)=q_1^{-1}(V)\cap SU_\tau(A)$ is open in $d_\tau$ topology.
    \par As both topologies are metrisable, it suffices to show that for any sequence $\overline{u}_n\in PU_0(A)$ converging to 1 the sequence $\overline{u}_n$ converges also in $d_\tau$ to 1. One may choose $u_n\in U(A)_0$ such that $q_1(u_n)=\overline{u}_n$ and $\|u_n-1\|\rightarrow 0$. Then for sufficiently large $n$ one has that $\|u_n-1\|< 1$ and so $v_n=u_ne^{-\tau(\log(u_n))}\in \ker(\Delta_\tau)$. Now $q_2(v_n)=\overline{u}_n$ and 
    \[
    d_\tau(1,v_n)\leq \|\log(u_n)-\tau(\log(u_n))\|\leq 2\|\log(u_n)\|\rightarrow 0 \qedhere
    \]
\end{proof}

\section{Crossed modules associated to C$^*$-algebras}\label{sec:crossedmodules}
We start by recalling the definition of a topological crossed module.
\begin{definition} \label{def:crossed_mod}
	A \emph{(topological) crossed module} consists of a continuous homomorphism $\partial \colon H \to G$ of topological groups, and a continuous left action of $G$ on $H$ such that 
	\begin{enumerate}[(i)]
		\item \label{cm:i} $\partial(\alpha(u)) = \alpha\,\partial(u)\,\alpha^{-1}$ for all $\alpha \in G$ and $u \in H$,
		\item \label{cm:ii} $\partial(u)(v) = u\,v\,u^{-1}$ for all $u,v \in H$.
	\end{enumerate}
Where we denote the action of $\alpha\in G$ on $u\in H$ by $\alpha(u)$.
\end{definition}
Any topological group $G$ defines a crossed module $1\rightarrow G$. Similarly, a topological group $H\rightarrow 1$ defines a crossed module if and only if $H$ is abelian. There are canonical crossed modules associated to C$^*$-algebras too. The map
\[
	\partial = \Ad \colon U(A) \to \Aut{A} \quad , \quad u \mapsto \Ad_u
\]
is a continuous group homomorphism. Moreover, $\Aut{A}$ acts continuously on $U(A)$ by mapping $u$ to $\alpha(u)$ for $\alpha \in \Aut{A}$ and $u \in U(A)$. It is straightforward to check that conditions \eqref{cm:i} and \eqref{cm:ii} of \cref{def:crossed_mod} are satisfied. Indeed,
\begin{align*}
	\partial(\alpha(u)) &= \Ad_{\alpha(u)} = \alpha \circ \Ad_u \circ \alpha^{-1} = \alpha\,\partial(u)\,\alpha^{-1} \ ,\\
	\partial(u_1)(u_2) &= \Ad_{u_1}(u_2) = u_1\,u_2\,u_1^{-1}\ .
\end{align*}
We will denote this crossed module by $\cG_A$. We can modify this construction slightly to construct another crossed module. Let
\[
	PU(A) = U(A) / ZU(A)\ .
\]
As a quotient of a Banach Lie group by a closed normal Lie subgroup $PU(A)$ is again a Banach Lie group (see for example \cite[Thm.~1]{paper:GloecknerNeeb}). Since $\Ad \colon U(A) \to \Aut{A}$ factors through a homomorphism $PU(A) \to \Aut{A}$, which we continue to denote $\Ad$, and the action of $\Aut{A}$ is also well-defined on $PU(A)$, we obtain another crossed module that will be denoted by $P\cG_A$. 
\par For the remainder of the section we assume that $U(A)=U(A)_0$. Denote by $\covU{A}$ the universal cover of $U(A)$, which is again a Banach Lie group modelled on the same Banach Lie algebra of self-adjoint elements in $A$. This universal cover has a concrete description. Let
\[
	\mathcal{P}_{U(A)} = \{ \gamma \in C([0,1],U(A))\ |\ \gamma(0) = 1_A \}
\] 
This is a group with respect to pointwise multiplication of paths. Recall that $\Omega_0 U(A) \subset \mathcal{P}_{U(A)}$ is the normal subgroup of those loops in $U(A)$ based at $1_A$ that are based homotopic to the constant loop at $1_A$, i.e.\ the path component of the identity in $\Omega U(A)$. We have  
\[
	\covU{A} = \mathcal{P}_{U(A)}/ \Omega_0U(A)\ .
\] 
Note that $\Aut{A}$ acts continuously on $\mathcal{P}_{U(A)}$ via $\alpha(\gamma) = \alpha \circ \gamma$. This actions maps $\Omega_0U(A)$ into itself and therefore induces a continuous action of $\Aut{A}$ on $\covU{A}$. Let $q \colon \covU{A} \to U(A)$ be the bundle map given by $q([\gamma]) = \gamma(1)$, then 
\[
	\widetilde{\Ad} = \Ad \circ q \colon \covU{A} \to \Aut{A} 
\]
is a continuous group homomorphism. 
\begin{lemma} \label{lem:adjoint_action}
Let $\gamma_1, \gamma_2 \in \mathcal{P}_{U(A)}$ be paths representing $[\gamma_1], [\gamma_2] \in \covU{A}$. In~$\covU{A}$ the following identity holds:
\begin{equation} \label{eqn:Ad-eq}
	\Ad_{\gamma_1(1)}([\gamma_2]) = \Ad_{[\gamma_1]}([\gamma_2])\ .
\end{equation}
Thus $\widetilde{Ad}:\covU A\rightarrow \Aut A$ defines a crossed module. We denote this crossed module by $\wcG_A$.
\end{lemma}

\begin{proof}
Consider the homotopy $H \colon [0,1] \times [0,1] \to U(A)$ defined by 
\[
	H(s,t) = \gamma_1(s + (1-s)t)\cdot \gamma_2(s) \cdot \gamma_1(s + (1-s)t)^{-1}\ .
\]
This satisfies $H(0,t) = 1_A$, $H(\cdot,0) = \Ad_{\gamma_1}(\gamma_2)$, $H(\cdot, 1) = \Ad_{\gamma_1(1)}(\gamma_2)$ and $H(1,t) = \Ad_{\gamma_1(1)}(\gamma_2(1))$, which implies $H(\cdot, 0)^{-1} \cdot H(\cdot, 1) \in \Omega_0U(A)$.	
We now check that $\wcG_A$ is a crossed module. Let $\alpha \in \Aut{A}$ and $\gamma, \gamma_1,\gamma_2 \in \mathcal{P}_{U(A)}$. By Equation \eqref{eqn:Ad-eq} we have
\begin{align*}
	\widetilde{\Ad}_{\alpha([\gamma])} &= \Ad_{\alpha(\gamma(1))} = \alpha \circ \Ad_{\gamma(1)} \circ \alpha^{-1} = \alpha \circ \widetilde{\Ad}_{[\gamma]} \circ \alpha^{-1} \\
	\widetilde{\Ad}_{[\gamma_1]}([\gamma_2]) &= \Ad_{\gamma_1(1)}([\gamma_2]) = \Ad_{[\gamma_1]}(\gamma_2)=[\gamma_1] \cdot [\gamma_2] \cdot [\gamma_1]^{-1}.
\end{align*}
\end{proof}

Note that Lemma \ref{lem:adjoint_action} also proves that we have a central extension
\[
	1 \to \pi_1(U(A)) \to \covU{A} \to U(A) \to 1\ .
\]

\par We now turn to crossed modules associated to tracial C$^*$-algebras.
\begin{lemma}
    Let $\tau\in T(A)$. The canonical action of $\Autt A$ on $SU_\tau(A)$ is continuous. Hence the mapping $\Ad: SU_\tau(A)\rightarrow \Autt A$ defines a crossed module denoted by $S\cG_A^{\tau}$.
\end{lemma}
\begin{proof}
    Let $(\alpha_n,u_n)$ be a sequence in $\Autt A \times SU_\tau(A)$ such that $\alpha_n$ converges to $\alpha$ and $u_n$ converges to $u$ in point norm and $d_\tau$ topology respectively. We want to show that $\alpha_n (u_n)$ converges to $\alpha(u)$ in $d_\tau$. Firstly, as $\alpha$ preserves the $d_\tau$ metric, we may assume that $\alpha=\id A$. Moreover, as $\alpha_n (u_n)=\alpha_n (u_n u^*)\alpha_n (u)$, and $\alpha_n (u_n u^*)$ converges to 1 in $d_\tau$, then it suffices to show that whenever $\alpha_n$ converges to $\id A$, then $\alpha_n(u)$ converges to $u$ in $d_\tau$.

    Let $u=e^{ih_1}\ldots e^{ih_k}$ with $h_i\in A^{sa}\cap \ker(\tau)$ for $1\leq i\leq k$. Then $\alpha_n(u)=e^{i\alpha_n(h_1)}\ldots e^{i\alpha_n(h_k)}$. Since $\alpha_n(u)u^*$ converges to $1$ in norm, then for large enough $n$ one can pick a cntinuous branch of the logarithm such that
    \[
    \alpha_n(u)u^*=e^{\log(\alpha_n(u)u^*)}
    \]
    \[
    \|\log(\alpha_n(u)u^*)\|<\pi/2
    \]
    and following an argument as in Lemma \ref{lemma:exponentillengthsmall}
    \[
    \tau(\log(\alpha_n(u)u^*))=\sum_{i=1}^k \tau(h_i)+\tau(\alpha_n(h_i))=0.
    \]
    Thus it follows from Lemma \ref{lem:estimatenorms} that $\alpha_n(u)$ converges to $u$ in $d_\tau$.
\end{proof} 
The map $\Ad:PU(A)\rightarrow \Autt A$ also defines a crossed module which we denote by $P\cG^{\tau}_A$.
\begin{definition}\label{defn:morphicrossedmodules}
Let $(\partial_1, H_1,G_1)$ and $(\partial_2, H_2, G_2)$ be crossed modules. A morphism of crossed modules $\Phi=(\phi_1,\phi_2): (\partial_1, H_1,G_1)\rightarrow (\partial_2, H_2, G_2)$ consists of a pair of continuous group homomorphisms $\phi_1:H_1\rightarrow H_2$ and $\phi_2: G_1\rightarrow G_2$ satisfying
\begin{enumerate}[(i)]
		\item \label{mcm:i} $\partial_2\circ \phi_1=\phi_2 \circ \partial_1$,
		\item \label{mcm:ii} $\phi_1(\alpha(u))=\phi_2(\alpha)(\phi_1(u))$ for all $\alpha\in G_1$, $u\in H_1$.
	\end{enumerate}
\end{definition}
Moreover, if $\cG_0$, $\cG_1$ and $\cG_2$ are crossed modules we call a sequence 
\[
\cG_0\xrightarrow{\Phi} \cG_1 \xrightarrow{\Psi} \cG_2
\]
a \emph{short exact sequence} if the two canonically associated sequences of groups are short exact sequences of topological groups in the sense of Section \ref{sec:prelim}.
\par In particular there are short exact sequences of crossed modules 
\begin{gather*}
(ZU(A) \to 1)\longrightarrow \cG_A\longrightarrow P\cG_A, \\
(ZSU_\tau(A) \to 1)\longrightarrow S\cG_\tau(A)\rightarrow P\cG_A,
\end{gather*}
as $ZU(A)\rightarrow U(A)\rightarrow PU(A)$ and $ZSU_\tau(A)\rightarrow SU_\tau(A)\rightarrow PU(A)$ are Hurewicz fibrations by \cite[II.7 Example 4, II.7 Corollary 14]{SPA66} and Lemma \ref{lem:topologyproj}. Recall that $q:\covU{A}\rightarrow U(A)$ is the bundle map $q([\lambda])=\lambda(1)$ and denote by $r:U(A)\rightarrow PU(A)$ the canonical quotient map.
\begin{lemma} \label{lem:ses-covU}
	The composition $\pi = r \circ q \colon \covU{A} \to PU(A)$ gives rise to a short exact sequence
	\[
	\begin{tikzcd}
		K_0^\#(A) \ar[r] \ar[d] & \covU{A} \ar[d,"\Ad{}"] \ar[r,"\pi"] & PU(A) \ar[d,"\Ad{}"] \\
		1 \ar[r] & \Aut{A} \ar[equal]{r}	 & \Aut{A}
	\end{tikzcd}
	\]
	of crossed modules, where the first row is a central extension of Banach Lie groups. Moreover, $\pi \colon \covU{A} \to PU(A)$ is a principal $K_0^\#(A)$-bundle. 
\end{lemma}

\begin{proof}
	Let $[\delta] \in K_0^\#(A)$ and $[\gamma] \in \covU{A}$. By Lemma \ref{lem:adjoint_action} we have that
	\[
		\Ad_{[\delta]}([\gamma]) = \Ad_{\delta(1)}([\gamma]) = [\gamma],
	\]
	since $\delta(1) \in ZU(A)$. This shows that $K_0^\#(A)$ is a central in $\covU{A}$. It is also a normal Lie subgroup of $\covU{A}$ (disconnected if $\pi_1(U(A)) \neq 1$) with the same Banach Lie algebra as $ZU(A)$. This implies that the first row of the diagram in the lemma is a central extension and reveals the whole diagram as a short exact sequence of crossed modules. The statement about principal bundles follows from \cite[Remark II.5]{paper:GloecknerNeeb}.
\end{proof}
\subsection{Cohomology with coefficients in a crossed module} \label{ssec:Cohomology}
Let $\Gamma$ be a discrete group. A $1$-cocycle on $\Gamma$ with values in $\cG$ consists of a pair $(\alpha, u)$, where $\alpha \colon \Gamma \to G$ and $u \colon \Gamma \times \Gamma \to H$ satisfying the following conditions
\begin{align}
	\alpha_{g} \, \alpha_{h} &= \partial(u_{g,h}) \, \alpha_{gh}\ ,\label{eqn:alpha}\\
	\alpha_g(u_{h,k})\,u_{g,hk} &= u_{g,h}\,u_{gh,k} \label{eqn:u}\ .
\end{align}
Denote the set of all $1$-cocycles on $\Gamma$ with values in $\cG$ by $Z^1(\Gamma, \cG)$. Two such cocycles $(\alpha^{(0)}, u^{(0)})$ and $(\alpha^{(1)}, u^{(1)})$ are cohomologous if there exists an element $\gamma \in G$ and a map $w \colon \Gamma \to H$ with the properties:
\begin{align}
	\alpha^{(1)}_g &= \partial w_g\,\gamma\,\alpha^{(0)}_g\,\gamma^{-1} \label{eqn:alpha12}\\
	u^{(1)}_{g,h} &= w_g \gamma(\alpha^{(0)}_g(\gamma^{-1}(w_h))u^{(0)}_{g,h})w_{gh}^{-1} \label{eqn:u12}  	
\end{align}
It is straightforward to check that being cohomologous defines an equivalence relation on the set of all $1$-cocycles. 

\begin{definition} \label{def:cohomology}
	Let $\cG = (H,G, \partial)$ be a topological crossed module. The pointed set defined by
	\[
		H^1(\Gamma, \cG) = Z^1(\Gamma, \cG)/\!\sim
	\]
	is called the cohomology of $\Gamma$ with coefficients in the crossed module $\cG$. We usually think of $H^1(\Gamma,\cG)$ as a pointed set with the trivial cocycle as the distinguished point.
\end{definition}
\begin{remark}
	The definition we have given here is compatible with the one in \cite{paper:Noohi-coh}. To turn as cocycle $(\alpha, u)$ into a cocycle $(p, \varepsilon)$ as in \cite[Section 4]{paper:Noohi-coh} we define
	\[
		p_g = \alpha_{g^{-1}}^{-1} \quad \text{and} \quad \varepsilon_{g,h} = u_{h^{-1}, g^{-1}}^* .
	\] 
	Since the equivalence relations are compatible as well, the above provides a bijection between the cohomology sets $H^1(\Gamma, \cG)$ from \cref{def:cohomology} and the ones in  \cite[Section 4]{paper:Noohi-coh}.
\end{remark}
In case of $\cG_A$ a cocycle $(\alpha, u)$ with $\alpha \colon \Gamma \to \Aut{A}$ and $u \colon \Gamma \times \Gamma \to U(A)$ is exactly the same as a cocycle action of $\Gamma$ on $A$. It is easy to check that the equivalence relation used to define $H^1(\Gamma, \cG_A)$ corresponds to cocycle conjugacy with respect to this identification. Therefore we have a natural bijection of pointed sets
\[
	H^1(\Gamma,\cG_A) \cong \{\text{cocycle actions of }\Gamma \text{ on }A\}/\sim_{c.c}\ .
\]
We have a similar interpretation of cohomology with coefficients in $P\cG_A$.

\begin{prop}
	Let $\Gamma$ be a discrete group and let $A$ be a unital $C^*$-algebra. Let $q \colon \Aut{A} \to \Out{A}$ be the quotient homomorphism. The map $(\alpha, [u]) \mapsto q\circ\alpha$ provides a natural isomorphism of pointed sets
	\[
		H^1(\Gamma, P\cG_A) \cong \{ \text{$\Gamma$-kernels of $A$} \}/\!\sim_c.
	\] 
	Similarly, if $\tau\in T(A)$ then $(\alpha, [u]) \mapsto q\circ\alpha$ provides a natural isomorphism of pointed sets
    \[ 
    H^1(\Gamma, P\cG_A^{\tau})\cong \{\text{$\tau$-preserving $\Gamma$-kernels of $A$} \}/\!\sim_{\tau.c}.
    \]
\end{prop}

\begin{proof}
We only prove the first statement as the proof of the second statement follows in the same way. Let $(\alpha,[u])$ be a $1$-cocycle with values in $P\cG_A$. By Equation \ref{eqn:alpha} the map $q \circ \alpha \colon \Gamma \to \Out{A}$ is a group homomorphism. By Equation \ref{eqn:alpha12} the equivalence relation on cocycles is compatible with conjugation. Hence,
	\[
		H^1(\Gamma, P\cG_A) \to \{ \text{$\Gamma$-kernels of $A$} \}/\!\sim_c
	\]  
	is well-defined. Let $\bar{\alpha} \colon \Gamma \to \Out{A}$ be a group homomorphism. For each $g \in \Gamma$ choose a lift $\alpha_g \in \Aut{A}$. For each pair $g,h \in \Gamma$ there is projective unitary $[u_{g,h}] \in PU(A)$ such that 
	\[
		\alpha_g \circ \alpha_h = \Ad_{u_{g,h}} \circ\, \alpha_{gh}\ .
	\]
	Notice that $[u_{g,h}]$ is uniquely fixed by this requirement. A comparison of $(\alpha_g \circ \alpha_h) \circ \alpha_k$ with $\alpha_g \circ (\alpha_h \circ \alpha_k)$ shows that the cocycle condition \cref{eqn:u} is satisfied. Therefore $(\alpha,u)$ is a $1$-cocycle mapping to $\bar{\alpha}$ showing that the above map is surjective.
	
	Let $(\alpha, [u])$ and $(\beta,[v])$ be $1$-cocycles. Assume that there exists an element $\bar{\gamma} \in \Out{A}$ such that 
	\[
		q \circ \beta = \bar{\gamma} \circ (q \circ \alpha) \circ \bar{\gamma}^{-1}\ .
	\]
	Choose a lift of $\bar{\gamma}$ to $\gamma \in \Aut{A}$. The equation $q \circ \beta = q \circ (\gamma \circ \alpha \circ \gamma^{-1})$ implies the existence of a map $[w] \colon \Gamma \to PU(A)$ such that 
	\[
		\beta_g = \Ad_{w_g} \circ\ (\gamma \circ \alpha_g \circ \gamma^{-1})
	\]
	for all $g \in \Gamma$. Observe that
	\begin{align*}
		\beta_g \circ \beta_h &=  \Ad_{w_g} \circ\ (\gamma \circ \alpha_g \circ \gamma^{-1}) \circ \Ad_{w_h} \circ\ (\gamma \circ \alpha_h \circ \gamma^{-1}) \\
		&= \Ad_{w_g \gamma(\alpha_g(\gamma^{-1}(w_h))u_{g,h})} \gamma\circ \alpha_{gh}\circ \gamma^{-1} \\
		&= \Ad_{w_g \gamma(\alpha_g(\gamma^{-1}(w_h))u_{g,h})w_{gh}^*} \beta_{gh}\ .
	\end{align*}
	Therefore we must have $[v_{g,h}] = [w_g \gamma(\alpha_g(\gamma^{-1}(w_h))u_{g,h})w_{gh}^*] \in PU(A)$ and the pair $(\gamma, [w])$ satisfies Equations \eqref{eqn:alpha12} and \eqref{eqn:u12}. Hence, $(\alpha, [u])$ and $(\beta, [v])$ are equivalent and the map $(\alpha,[u]) \mapsto q \circ \alpha$ is injective as well.
\end{proof}
Let $H$ be an abelian group. It is a consequence of Equation \ref{eqn:u} that a $1$-cocycle in $Z^1(\Gamma, H \to 1)$ is actually a $2$-cocycle on $\Gamma$ with values in $H$. Moreover, Equation \ref{eqn:u12} implies that two such $1$-cocycles are equivalent if and only if they differ by a $2$-coboundary in group cohomology with coefficients in $B$. Hence, we have a natural isomorphism  of pointed sets
\begin{equation} \label{eqn:abel-iso}
	H^1(\Gamma, H \to 1) \cong H^2(\Gamma, H)\ .
\end{equation}
 As discussed for example in \cite{paper:Noohi-coh}, one can also make sense of lower cohomology pointed sets in degree $-1$ and $0$. Moreover, there is an analogue of the classical long exact sequence of cohomology also in the case of crossed module coefficients. However, as the pointed sets $H^1(\Gamma, \cG)$ are not a group in general, we first recall the right notion of an exact sequence in the category of pointed sets.
 \begin{definition}
     Let $f:(X,x_0)\rightarrow (Y,y_0)$ and $g:(Y,y_0)\rightarrow (Z,z_0)$ be maps of pointed sets then a sequence
     \[
	\begin{tikzcd}
		X \ar[r,"f"] & Y \ar[r,"g"] & Z
	\end{tikzcd}
\]
is said to be exact at $Y$ if and only if $g^{-1}(z_0) = f(X)$.
 \end{definition}
 If $\cG_{-1}\xrightarrow{\Psi} \cG_0\xrightarrow{\Phi} \cG_1$ is a short exact sequence of crossed modules, then it follows from \cite[Proposition 11.3]{paper:Noohi-coh} that one has a long exact sequence of pointed sets finishing at $H^1(\Gamma,\cG_1)$. When $\cG_{-1}$ is an abelian crossed module one can extend the associated long exact sequence one more step (\cite[Remark 11.6]{paper:Noohi-coh}). Before we show this we need to introduce some notation.
\par In the short exact sequence above we let $\cG_0= (\partial_0:H_0\rightarrow G_0)$, $\cG_1= (\partial_1: H_1\rightarrow G_1)$, $\cG_{-1}=(H_{-1}\rightarrow 1)$ and $\Phi=(\phi_0,\phi_1): \cG_0\rightarrow \cG_1$. Observe that $\phi_1$ is an isomorphism. For simplicity, we will assume that $G_0=G_1$ and that $\phi_1$ is the identity. We identify $H_{-1}$ with the kernel of $\phi_0$ and thus one may assume that $\Psi$ is the inclusion of $H_{-1}$ into $H_0$. Note that $\partial_0(k)=\partial_1(\phi_0(k))=0$  for all $k\in H_{-1}$ so $H_{-1}\rightarrow H_0\rightarrow H_1$ is a central extension. 
\par Any $1$-cocycle $(\alpha,u)\in Z^1(\Gamma,\cG_1)$ induces an action of $\Gamma$ on $H_{-1}$ by $g\cdot k=\alpha_g(k)$. Indeed $\phi_0(\alpha_g(k))=\alpha_g(\phi_0(k))=0$ so $\alpha_g(k)\in H_{-1}$ for all $g\in \Gamma$. We will denote by $H_{-1}^{(\alpha,u)}$ the $\Gamma$ module induced by $(\alpha,u)$. The abelian group 
     \[
     \underset{\sim \cG_1}{\bigoplus}H^3(\Gamma, H_{-1}).
     \]
     is given by first taking the direct sum of $H^3(\Gamma,H_{-1}^{(\alpha,u)})$ over all $(\alpha,u)\in Z^1(\Gamma,\cG_1)$, and then taking the quotient of this direct sum by the action of $G_1$, where $\gamma\in G_1$ acts by \footnote{We identify $H_{-1}^{(\alpha,u)}$ with $H_{-1}^{(\beta,v)}$ if the induced actions by $\Gamma$ are the same. So strictly speaking we only take the direct sum over equivalence classes of elements of $Z^1(\Gamma,\cG_1)$ where two cocycles are identified if they induce the same module structure on $H_{-1}$.}
     \begin{align*}
     H^3(\Gamma, H_{-1}^{(\alpha,u)})&\rightarrow H^3(\Gamma, H_{-1}^{(\gamma\alpha\gamma^{-1},\gamma(u))})\\
     \omega_{g,h,k}&\mapsto (\gamma\cdot \omega)_{g,h,k}=\gamma(\omega_{g,h,k}). 
     \end{align*}
     This construction comes equipped with canonical group homomorphisms
     \[
     \iota_{(\alpha,u)}:H^3(\Gamma, H_{-1}^{(\alpha,u)})\rightarrow \underset{\sim \cG_1}{\bigoplus}H^3(\Gamma, H_{-1})
     \]
     which satisfy $\iota_{\gamma\alpha\gamma^{-1},\gamma(u)}(\gamma \cdot \omega)=\gamma\cdot \iota_{\alpha,u}(\omega)$ for every $\gamma\in G_1$ and $\omega\in H^3(\Gamma,H_{-1}^{(\alpha,u)})$.
 \begin{lemma}\label{lem:longexact}
     Let $\cG_0= (\partial_0:H_0\rightarrow G_0)$ and $\cG_1= (\partial_1: H_1\rightarrow G_0)$ be crossed modules and $\Gamma$ a discrete group. Let $\Phi=(\phi_0,\id): \cG_0\rightarrow \cG_1$ be a morphism of crossed modules and $H_{-1}$ an abelian group fitting into a short exact sequence of crossed modules
     \[
	\begin{tikzcd}
		H_{-1} \ar[r] \ar[d] & H_0 \ar[d,"\partial_0"] \ar[r,""] & H_1 \ar[d,"\partial_1"] \\
		1 \ar[r] & G_0 \ar[equal]{r}	 & G_0
	\end{tikzcd}
     \]
     then there is an exact sequence of pointed sets
     \[
		\begin{tikzcd}
			H^2(\Gamma, H_{-1}) \ar[r] & H^1(\Gamma, \cG_0) \ar[r] & H^1(\Gamma,\cG_1) \ar[r] & \underset{\sim \cG_1}{\bigoplus}H^3(\Gamma, H_{-1}) \ ,
		\end{tikzcd}
	\]
     which is both natural in $\Gamma$ and in the coefficient exact modules.
 \end{lemma} 
 \begin{proof}
     The exactness in $H^1(\Gamma,\cG_0)$ follows from \cite[Proposition 11.3]{paper:Noohi-coh} and the isomorphism in \eqref{eqn:abel-iso}. We now define the final map in the exact sequence of cohomology sets. Let
     \begin{align}
     \theta: H^1(\Gamma,\cG_1)&\rightarrow \underset{\sim \cG_1}{\bigoplus}H^3(\Gamma, H_{-1}).\label{eqn:theta}\\
     (\alpha,u)&\mapsto \iota_{(\alpha,u)}(\alpha_g(v_{h,k})v_{g,hk}v_{gh,k}^{-1}v_{g,h}^{-1})\notag
     \end{align}
     where $v_{g,h}\in H_0$ for $g,h\in \Gamma$ is a choice of lift of $u_{g,h}$ i.e. $\phi_0(v_{g,h})=u_{g,h}$. First note that by precisely the same argument  as in \cite[Lemma 7.1]{EIMA47} the $3$-cochain $\omega_{g,h,k}=\alpha_g(v_{h,k})v_{g,hk}v_{gh,k}^{-1}v_{g,h}^{-1}$ defines an element of $Z^3(\Gamma, H_{-1}^{(\alpha,u)})$. 
     %This follows precisely as in \cite[Lemma 7.1]{EIMA47}. Indeed, letting $J=\alpha_g(\alpha_h(v_{l,k})v_{h,lk})v_{g,hlk}$ one has that
     %\begin{align*}
         %J&=\alpha_g(\omega_{h,l,k}v_{h,l}v_{hl,k})v_{g,hlk}\\
        % &= \alpha_g(\omega_{h,l,k})\omega_{g,hl,k}\alpha_g(v_{h,l})v_{g,hl}v_{ghl,k}\\
        % &=\alpha_g(\omega_{h,l,k})\omega_{g,hl,k}\omega_{g,h,l}v_{g,h}v_{gh,l}v_{ghl,k}
     %\end{align*}
     %but also that 
     %\begin{align*}
        % J&=v_{g,h}\alpha_{gh}(v_{l,k})v_{g,h}^{-1}\alpha_g(v_{h,lk})v_{g,hlk}\\
        % &=\omega_{g,hl,k}v_{g,h}\alpha_{gh}(v_{l,k})v_{g,h}^{-1}v_{g,h}v_{gh,lk}\\
         %&=\omega_{g,hl,k}v_{g,h}\alpha_{gh}(v_{l,k})v_{gh,lk}\\
         %&=\omega_{g,hl,k}\omega_{gh,l,k}v_{g,h}v_{gh,l}v_{ghl,k}.
     %\end{align*}
     %Comparing these two formulas for $J$ it follows that $\omega\in Z^3(\Gamma, H_{-1}^{(\alpha,u)})$. 
     Note also that if $v'_{g,h}\in H_0$ is another family such that $\phi_0(v_{g,h}')=u_{g,h}$, then there exists $\lambda_{g,h}\in H_{-1}$ such that $\lambda_{g,h}v_{g,h}=v_{g,h}'$. Therefore 
     \[
     \alpha_g(v_{h,k}')v_{g,hk}'v_{gh,k}'^{-1}v_{g,h}'^{-1}=\partial \lambda_{g,h,k} \omega_{g,h,k}
     \]
     and so $[\omega]\in H^3(\Gamma, H_{-1}^{(\alpha,u)})$ does not depend on the choice of lift $v$. 
     \par We now show that $\theta$ is well defined. If $(\beta,\rho)\sim (\alpha,u)$ then there exists $\varphi\in G_1$ and $\sigma:\Gamma\rightarrow H_1$ such that the following two equations hold
    \begin{align*}
     \beta_g &=\partial_1(\sigma_g)\varphi\alpha_g\varphi^{-1}, \\
     \rho_{g,h} &=\sigma_g\varphi\alpha_g\varphi^{-1}(\sigma_h)\varphi(u_{g,h})\sigma_{gh}^{-1}
    \end{align*}
     Let $s:\Gamma\rightarrow H_0$ be such that $\phi_0(s_g)=\sigma_g$ and $t_{g,h}\in H_0$ such that $\phi_0(t_{g,h})$. Then $r_{g,h}=s_g\varphi\alpha_g\varphi^{-1}(s_h)\varphi(t_{g,h})s_{gh}^{-1}$ is a lift of $\rho_{g,h}$ under $\phi_0$. We compute that
     \begin{align*}
          &\beta_g(r_{h,k})r_{g,hk}r_{gh,k}^{-1}r_{g,h}^{-1}\\
         =\ &s_g\varphi\alpha_g\varphi^{-1}(s_h)\varphi\alpha_g\alpha_h\varphi^{-1}(s_k)\varphi\alpha_g(t_{h,k}) \varphi\alpha_g\varphi^{-1}(s_{hk}^{-1})s_g^{-1}r_{g,hk}r_{gh,k}^{-1}r_{g,h}^{-1}\\
         =\ &s_g\varphi\alpha_g\varphi^{-1}(s_h)\varphi\alpha_g\alpha_h\varphi^{-1}(s_k)\varphi\alpha_g(t_{h,k}) \varphi(t_{g,hk})s_{ghk}^{-1}r_{gh,k}^{-1}r_{g,h}^{-1} \\
         =\ &s_g\varphi\alpha_g\varphi^{-1}(s_h)\varphi\alpha_g\alpha_h\varphi^{-1}(s_k)\varphi\alpha_g(t_{h,k}) \varphi(t_{g,hk}t_{gh,k}^{-1})\varphi\alpha_{gh}\varphi^{-1}(s_k^{-1})\varphi(t_{g,h}^{-1})\\
         &\varphi\alpha_g\varphi^{-1}(s_h^{-1})s_g^{-1}\\
         =\ &\partial_0(s_g\varphi\alpha_g\varphi^{-1}(s_h)\varphi\alpha_g\alpha_h\varphi^{-1})\varphi(\omega_{g,h,k}) = \varphi(\omega_{g,h,k})
     \end{align*}
     where in the final step we have used that $\varphi(\omega_{g,h,k})$ is still $H_{-1}$ which is contained in the centre of $H_0$.
     \par Hence \begin{align*}
     \theta(\beta,\rho)&=\iota_{(\beta,\rho)}(\beta_g(r_{h,k})r_{g,hk}r_{gh,k}^{-1}r_{g,h}^{-1})\\
     &=\iota_{(\varphi\alpha\varphi^{-1},\varphi(u))}(\beta_g(r_{h,k})r_{g,hk}r_{gh,k}^{-1}r_{g,h}^{-1})\\
     &=\iota_{(\varphi\alpha\varphi^{-1},\varphi(u))}(\varphi(\omega_{g,h,k}))\\
     &=\varphi\cdot \iota_{(\alpha,u)}(\omega_{g,h,k})\\
     &=\theta(\alpha,u).
     \end{align*}
     To show the exactness at $H^1(\Gamma,\cG_1)$ take $(\alpha,u)\in H^1(\Gamma,\cG_1)$ such that $\theta(\alpha,u)=0$. As $\iota_{(\alpha,u)}$ is injective there exists $\eta:\Gamma\times \Gamma\rightarrow H_{-1}^{(\alpha,u)}$ such that using the notation in Equation \ref{eqn:theta}
     \[
     \alpha_g(v_{h,k})v_{g,hk}v_{gh,k}^{-1}v_{g,h}^{-1}=\alpha_g(\eta_{h,k})\eta_{g,hk}^{-1}\eta_{g,hk}\eta_{g,h}
     \]
     As $\eta$ is valued in $H_{-1}$ the elements $s_{g,h}=v_{g,h}\eta_{g,h}^{-1}$ lift the unitaries $u_{g,h}$ through $\phi_0$. But by construction $(\alpha,s)\in Z^{1}(\Gamma,\cG_0)$ and maps onto $(\alpha,u)$ under the map $H^1(\Gamma,\cG_0)\rightarrow H^1(\Gamma,\cG_1)$. The naturality of the exact sequence follows from a straightfowards computation.
\end{proof}
     We now apply the above lemma to our examples of operator algebraic crossed modules.
     \begin{theorem}\label{thm:operatoralgcrossedmoduleseq}
         Let $A$ be a unital C$^*$-algebra, and let $Z_A=Z(U(A))$. Then there exist exact sequence of pointed sets
     \[
		\begin{tikzcd}
			H^2(\Gamma, Z_A) \ar[r] & H^1(\Gamma, \cG_A) \ar[r] & H^1(\Gamma,P\cG_A) \ar[r,"\ob"] & \underset{\sim P\cG_A}{\bigoplus}H^3(\Gamma, Z_A),
		\end{tikzcd}
	\]
    \[
		\begin{tikzcd}[column sep=4ex]
			H^2(\Gamma, K_0^\#(A)) \ar[r] & H^1(\Gamma, \wcG_A) \ar[r] & H^1(\Gamma, P\cG_A) \ar[r,"\tilde{\ob}"] & \underset{\sim P\cG_A}{\bigoplus} H^3(\Gamma, K_0^\#(A)),
		\end{tikzcd}
	\]
    \[
    \begin{tikzcd}[column sep=4ex]
			H^2(\Gamma, \pi_1(U(A))) \ar[r] & H^1(\Gamma, \wcG_A) \ar[r] & H^1(\Gamma, \cG_A) \ar[r,"\kappa^3"] & \underset{\sim \cG_A}{\bigoplus} H^3(\Gamma, \pi_1(U(A)))
		\end{tikzcd}
    \]
and if $\tau \in T(A)$ then also
\[
		\begin{tikzcd}
			H^2(\Gamma, Z_A^\tau) \ar[r] & H^1(\Gamma, S\cG_A^{\tau}) \ar[r] & H^1(\Gamma, P\cG_A^{\tau}) \ar[r,"\ob_\tau"] & \underset{\sim P\cG_A^{\tau}}{\bigoplus} H^3(\Gamma,Z_A^\tau)
		\end{tikzcd}
	\]
denoting $ZSU_\tau(A)$ by. These sequences are natural in $\Gamma$ and in the coefficient crossed modules.
     \end{theorem}
     \begin{proof}
         This follows by applying Lemma \ref{lem:longexact} to the short exact sequence of crossed modules
         \begin{equation} \label{eqn:short-ex-UA-PUA}
	\begin{tikzcd}
		Z_A \ar[r] \ar[d] & U(A)\ar[d,"\Ad{}"] \ar[r] & PU(A) \ar[d,"\Ad{}"] \\
		1 \ar[r] & \Aut{A} \ar[equal]{r}	 & \Aut{A},
	\end{tikzcd}
    \end{equation}
    \begin{equation} \label{eqn: short-ex-tilUA-PUA}
        \begin{tikzcd}
		K_0^{\#}(A) \ar[r] \ar[d] & \tilde{U}(A)\ar[d,"\widetilde{\Ad}{}"] \ar[r] & PU(A) \ar[d,"\Ad{}"] \\
		1 \ar[r] & \Aut{A} \ar[equal]{r}	 & \Aut{A},
	\end{tikzcd}
\end{equation}
\begin{equation} \label{eqn: short-ex-tilUA-UA}
        \begin{tikzcd}
		\pi_1(U(A)) \ar[r] \ar[d] & \tilde{U}(A)\ar[d,"\widetilde{\Ad}{}"] \ar[r] & U(A) \ar[d,"\Ad{}"] \\
		1 \ar[r] & \Aut{A} \ar[equal]{r}	 & \Aut{A},
	\end{tikzcd}
\end{equation}
\begin{equation}\label{eqn: short-ex-SU-PU}
        \begin{tikzcd}
		Z_A^\tau \ar[r] \ar[d] & SU_\tau(A)\ar[d,"\Ad{}"] \ar[r] & PU(A) \ar[d,"\Ad{}"] \\
		1 \ar[r] & \Aut{A} \ar[equal]{r}	 & \Aut{A}.
	\end{tikzcd}
\end{equation}
That these are short exact sequences of crossed modules follows from Lemma \ref{lem:ses-covU} and the discussion preceeding this lemma.
\end{proof}
When $A$ has trivial centre one can remove the direct sum in the final term of the first exact sequence of Theorem \ref{thm:operatoralgcrossedmoduleseq} (see Remark \ref{rmk:simplificationsobs}). Under this assumption one can remove the direct sum in the final terms of the second and third exact sequence when $A$ is has only approximately inner automorphisms and in the fourth when $A$ is monotracial.
 \section{Classifying spaces} \label{sec:ClassifyingSpaces}
 Each topological crossed module $\partial \colon H \to G$ gives rise to a strict topological $2$-ca\-te\-go\-ry internal to topological spaces, which we will also denote by $\cG$. Its space of objects $\cG_0$ consists of a single point. Its space of $1$-morphisms $\cG_1$ is the topological group $G$ with composition given by the group multiplication in $G$. The space of $2$-morphisms $\cG_2$ is the semidirect product $G \rtimes H$. More precisely, an element $(\alpha, u) \in G \rtimes H$ gives a $2$-morphism $\alpha \Rightarrow \partial(u)\alpha$ and the vertical composition is given by 
\[
	(\partial(u_2)\alpha, u_1) \circ_v (\alpha, u_2) = (\alpha, u_1u_2)\ ,
\]
i.e.\ the multiplication in $H$. The horizontal composition of the $2$-morphisms $(\alpha, u)$ and $(\beta, v)$ is defined to be
\[
	(\alpha, u) \circ_h (\beta, v) = (\alpha \cdot \beta, u\,\alpha(v) )
\]
and uses the action of $G$ on $H$. Using arrows for $1$-morphisms and bold arrows for $2$-morphisms the two compositions are pictured below:
\[
	\begin{tikzcd}[row sep=2cm]
		\ast & & \ar[ll,bend right=80,""{name=U},"\alpha" above] \ar[ll,""{name=M}] \ar[ll,bend left=80,""{name=D},"\partial(u_1u_2)\alpha" below] \ast
		\arrow[Rightarrow, from=U, to=M, end anchor={[shift={(0pt,8pt)}]south},"u_2"] 
		\arrow[Rightarrow, from=M, to=D, end anchor={[shift={(0pt,8pt)}]south},"u_1"] 
	\end{tikzcd}
	\rightsquigarrow
	\begin{tikzcd}[row sep=2cm]
		\ast & & \ar[ll,bend right=80,""{name=U2},"\alpha" above] \ar[ll,bend left=80,""{name=D2},"\partial(u_1u_2)\alpha" below] \ast
		\arrow[Rightarrow, from=U2, to=D2, end anchor={[shift={(0pt,10pt)}]south},"u_1u_2"] 
	\end{tikzcd} 
\]
\[
	\begin{tikzcd}[row sep=2cm]
		\ast & & \ar[ll,bend right=60,""{name=U1},"\alpha" above] \ar[ll,bend left=60,""{name=D1},"\partial(u)\alpha" below] \ast & & \ar[ll,bend right=60,""{name=U2},"\beta" above] \ar[ll,bend left=60,""{name=D2},"\partial(v)\beta" below] \ast
		\arrow[Rightarrow, from=U1, to=D1, end anchor={[shift={(0pt,10pt)}]south},"u"] 
		\arrow[Rightarrow, from=U2, to=D2, end anchor={[shift={(0pt,10pt)}]south},"v"] 
	\end{tikzcd}
	\rightsquigarrow
	\begin{tikzcd}[row sep=2cm]
		\ast & & \ar[ll,bend right=60,""{name=U3},"\alpha \cdot \beta" above] \ar[ll,bend left=60,""{name=D3},"\partial(u\alpha(v))\,\alpha \cdot \beta" below] \ast 
		\arrow[Rightarrow, from=U3, to=D3, end anchor={[shift={(0pt,10pt)}]south},"u\alpha(v)"] 
	\end{tikzcd}	
\]
Strict topological $2$-categories with one object and invertible $1$- and $2$-mor\-phisms are called topological $2$-groups. As explained in \cite{paper:BaezStevenson} the above construction gives a bijective correspondence between topological $2$-groups and topological crossed modules.

Neglecting the topology, any topological $2$-group is an example of a $2$-category in the sense of \cite[Definition 2.1.3 and Definition 2.3.2]{book:JohnsonYau}. We recall the definition of strictly unital pseudofunctors between $2$-categories as stated in \cite[Definition 4.1.2]{book:JohnsonYau}

\begin{definition} \label{def:2-functor}
	Let $\cC$ and $\cC'$ be small $2$-categories as in \cite[Definition 2.3.2]{book:JohnsonYau}. Denote the set of objects by $\cC_0$, $\cC'_0$, respectively, and let $\cC_k$, $\cC'_k$ for $k \in \{1,2\}$ be the sets of $k$-morphisms. A strictly unital pseudofunctor $\theta = (\theta^0, \theta^1, \eta) \colon \cC \to \cC'$ consists of 
	\begin{enumerate}[i)]
		\item a map $\theta^0 \colon \cC_0 \to \cC'_0$,
		\item for each pair of objects $x,y \in \cC_0$ a functor 
		\[
			\theta^1_{x,y} \colon \cC(x,y) \to \cC'(\theta^0(x), \theta^0(y))\ ,
		\]
		with $\theta^1_{x,x}(\id{x}) = \id{\theta^0(x)}$,
		\item For all objects $x,y,z \in \cC_0$ and pair of morphisms $\beta \colon x \to y$, $\alpha \colon y \to z$ a natural isomorphism 
		\[
		  \eta_{\alpha, \beta} \colon \theta^1(\alpha \circ \beta) \Longrightarrow  \theta^1(\alpha) \circ \theta^1(\beta)
		\]
		between the two functors 
		\[
		  \theta^1 \circ c_{\cC},\ c_{\cC'} \circ (\theta^1 \times \theta^1) \colon \cC(y,z) \times \cC(x,y) \to \cC'(\theta^0(x),\theta^0(z)) \ ,
		\]  
		where $c_{\cC}$ and $c_{\cC'}$ denote the composition in $\cC$ and $\cC'$, respectively, and the following associativity diagram commutes
		\[
		\begin{tikzcd}[column sep=2cm]
		  \theta^1(\alpha) \circ \theta^1(\beta) \circ \theta^1(\gamma) & \theta^1(\alpha \circ \beta) \circ \theta^1(\gamma) \arrow[l,Rightarrow,"\eta_{\alpha,\beta} \circ_h \id{}" above] \\
			\theta^1(\alpha) \circ \theta^1(\beta \circ \gamma) \arrow[u,Rightarrow,"\id{} \circ_h \eta_{\beta,\gamma}" left]  & \theta^1(\alpha \circ \beta \circ \gamma) \arrow[l,Rightarrow, "\eta_{\alpha, \beta \circ \gamma}" below] \arrow[u,Rightarrow,"\eta_{\alpha \circ \beta, \gamma}" right]
		\end{tikzcd}\ .
		\]
	\end{enumerate}
	The set of all strictly unital pseudofunctors between $\cC$ and $\cC'$ will be denoted by $\Fun{\cC}{\cC'}$. Let $\theta \in \Fun{\cC}{\cC'}$ and $\sigma \in \Fun{\cC'}{\cC''}$. Then the composition $\sigma \circ \theta$ of the two is defined by
	\begin{align*}
		(\sigma \circ \theta)^0 &= \sigma^0 \circ \theta^0\ ,\\
		(\sigma \circ \theta)^1_{x,y} &= \sigma^1_{\theta^0(x),\theta^0(y)} \circ \theta^1_{x,y} \ . 
	\end{align*}
	If $\mu$ and $\eta$ denote the natural isomorphisms for $\sigma$ and $\theta$, respectively, then  
	\[
		\mu_{\theta^1(\alpha), \theta^1(\beta)} \circ \sigma^1(\eta_{\alpha, \beta})  \colon \sigma^1(\theta^1(\alpha \circ \beta)) \Longrightarrow 
        \sigma^1(\theta^1(\alpha)) \circ \sigma^1(\theta^1(\beta))
	\]	
	defines the natural isomorphism for $\sigma \circ \theta$.
\end{definition}

\par Any morphism of crossed modules $\Phi=(\phi_1,\phi_2):\cG_1=(\partial_1,H_1,G_1)\rightarrow \cG_2=(\partial_2,H_2,G_2)$ as in Definition \ref{defn:morphicrossedmodules} induces a (continuous) strictly unital pseudofunctor from the 2-category associated to $\cG_1$ to the 2-category associated to $\cG_2$. This functor is defined by the identity on the single object, $\phi_2$ at the level of 1-morphisms, at the level of 2-morphisms it is given by
\begin{align*}
    G_1\rtimes H_1&\rightarrow G_2\rtimes H_2\\ 
    (g,h)&\mapsto (\phi_2(g),\phi_1(h))
\end{align*}
   and the natural isomorphisms $\eta$ are chosen to be the identity. There are several ways of associating a classifying space to a topological $2$-group obtained from a crossed module $\cG = (G,H,\partial)$. We will describe two of them in more detail. All of them proceed by first constructing a simplicial space from $\cG$. The geometric realisation of this simplicial space is then the classifying space. 

The first construction goes back to Duskin \cite{paper:Duskin}: Let $[n]$ be the $2$-category associated to the ordered set $\{0, \dots, n\}$, i.e.\ there is a unique $1$-morphism from $j$ and $i$ if and only if $i \leq j$ and the category has only identity $2$-morphisms. Let 
\[
	N_n^D(\cG) = \Fun{[n]}{\cG}\ .
\] 
Since $\cG$ has only one object, a strictly unital pseudofunctor is determined by a choice of $1$-morphism for each pair $i < j$ and a $2$-morphism for $i < j < k$. Therefore  
\[
	\Fun{[n]}{\cG} \subseteq G^{\binom{n+1}{2}} \times H^{\binom{n+1}{3}}
\]
and we give it the subspace topology. Any order-preserving map $f \colon [m] \to [n]$ can be viewed as a functor and gives rise to a continuous map  
\[
	f^* \colon \Fun{[n]}{\cG} \to \Fun{[m]}{\cG}\ .
\]
With this definition $N^D_\ast(\cG)$ is a simplicial space, called the Duskin nerve of~$\cG$. Unpacking the definition, we have 
\begin{align*}
	N^D_0(\cG) &= \ast \ ,\\
	N^D_1(\cG) &= G \ .
\end{align*}
The space of $n$-simplices $N^D_n(\cG)$ consists of $(\alpha_{ij})_{0 \leq i < j \leq n}$ with $\alpha_{ij} \in G$ and $(u_{ijk})_{0 \leq i < j < k \leq n}$ with $u_{ijk} \in H$ such that
\begin{align*}
	\alpha_{ij} \, \alpha_{jk} &= \partial(u_{ijk})\, \alpha_{ik}\ , \\
	\alpha_{ij}(u_{jkl})\,u_{ijl} &= u_{ijk}\,u_{ikl}\ .	
\end{align*}
In particular, the space of $2$-simplices can be visualised as the space of $H$-labelled triangles with $G$-labelled edges, subject to the condition shown on the right in the following diagram: 
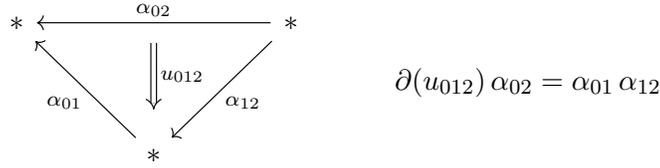
\begin{center}
\begin{tabular}{cc}
	\begin{tikzcd}[row sep=1.3cm, column sep=1.3cm]
		\ast & \arrow[d,Rightarrow,yshift=5pt,shorten=2mm,"u_{012}" right] & \ast \ar[ll, "\alpha_{02}" above] \arrow{dl}[below,xshift=8pt]{\alpha_{12}} \\
		& \ast  \arrow{ul}[below,xshift=-8pt]{\alpha_{01}}
	\end{tikzcd}
&
	$\qquad \partial(u_{012}) \, \alpha_{02} = \alpha_{01} \, \alpha_{12}$
\end{tabular}
\end{center}
The space of $3$-simplices consists of tetrahedra with triangular faces as above subject to a compatibility condition that arises from composing the $2$-morphisms on the sides. This is shown in the unfolded diagram below: 
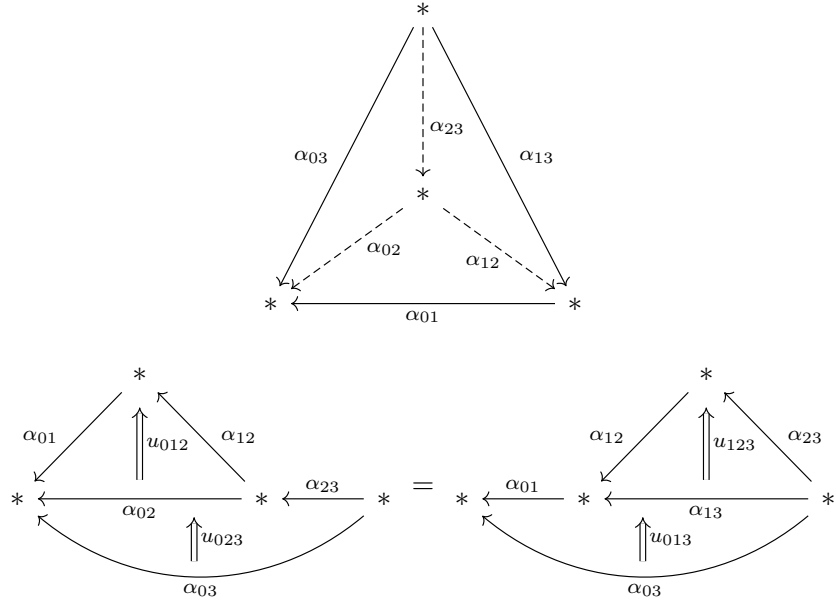
\[
\begin{tikzcd}[column sep=1.5cm, row sep=1cm]
& \ast 
\arrow{rddd}[right,xshift=5pt]{\alpha_{13}} 
\arrow[dashed]{dd}[yshift=-10pt]{\alpha_{23}} 
\arrow{lddd}[left,xshift=-4pt]{\alpha_{03}} \\
& & \\
& \ast 
\arrow[dashed]{rd}[below,xshift=-6pt]{\alpha_{12}} 
\arrow[dashed]{ld}[right,xshift=5pt]{\alpha_{02}} & \\
\ast  & & 
\arrow[ll,"\alpha_{01}" below] \ast   
\end{tikzcd}
\]
\[
\begin{tikzcd}[row sep=1.2cm, column sep=1.1cm]
	& \ast 
	\arrow{dl}[left,xshift=-5pt]{\alpha_{01}} 
	\\
	\ast 
	& \arrow[u,Rightarrow,shorten=2mm,"u_{012}" right] & 
	\ast 
	\ar[ll, "\alpha_{02}" below] 
	\arrow{ul}[right,xshift=5pt]{\alpha_{12}}  
	& \ast 
	\arrow{l}[above]{\alpha_{23}} 
	\arrow[bend left=40]{lll}[below]{\alpha_{03}} 
	\\
	& & 
  	\arrow[Rightarrow,shorten=3mm,xshift=-9mm,yshift=3mm]{u}[right]{u_{023}}
\end{tikzcd}
=
\begin{tikzcd}[row sep=1.2cm, column sep=1.1cm]
	& & \ast 
	\arrow{dl}[left,xshift=-5pt]{\alpha_{12}}
	\\
	\ast 
	&  
	\ast 
	\arrow{l}[above]{\alpha_{01}}
	& 
	\arrow[Rightarrow,shorten=2mm]{u}[right]{u_{123}}
	&
	\ast 
	\arrow{ul}[right,xshift=5pt]{\alpha_{23}}
	\arrow{ll}[below]{\alpha_{13}}
	\arrow[bend left=40]{lll}[below]{\alpha_{03}}
	\\
	& & 
        \arrow[Rightarrow,shorten=3mm,xshift=-24pt,yshift=3mm]{u}[right]{u_{013}}
\end{tikzcd}
\]
The simplicial structure maps arise from letting the order-preserving maps act on the indices: If $f \colon [m] \to [n]$ is an order-preserving map, then 
\[
	f^*\left( (\alpha_{ij}, u_{ijk})_{0 \leq i < j < k \leq n} \right) = (\beta_{rs}, v_{rst})_{0 \leq r < s < t \leq m}
\]
with 
\begin{align*}
	\beta_{rs} &= \begin{cases}
		\alpha_{f(r)f(s)} & \text{if } f(r) < f(s) \ ,\\
		1 & \text{else}
	\end{cases} \\
	v_{rst} &= \begin{cases}
		u_{f(r)f(s)f(t)} & \text{if } f(r) < f(s) < f(t) \ .\\
		1 & \text{else}
	\end{cases}
\end{align*}
\begin{remark}\label{rmk:k-space}
    Strictly speaking $N_{\ast}^D(\cG)$ is only a simplicial space when both $G$ and $H$ are compactly generated as topological spaces as by Definition \ref{def:simp_space}. We will apply the Duskin nerve construction to the operator algebraic crossed modules constructed in Section \ref{sec:crossedmodules} (or also the operator algebraic crossed modules constructed in Lemma \ref{lem:crossedmodulestrict}). All of the groups involved in these crossed modules will be Polish groups in the case of separable $A$, so their compact generation as topological spaces will be immediate in this case.
\end{remark}
\begin{definition} \label{def:Duskin_nerve}
	Let $\cG$ be a topological crossed module. The space $B^D\cG$ obtained as the fat geometric realisation of the Duskin nerve $N^D_\ast(\cG)$, i.e.
	\[
		B^D\cG := \lVert N^D_\ast(\cG) \rVert\ ,
	\]
	is called the \emph{classifying space} of $\cG$. This provides a functor $B^D$ from the category of crossed modules to $\Top$.
\end{definition}

It is worth pointing out a special case of this construction: Let $\Gamma$ be a topological group. It gives rise to a crossed module $\partial \colon 1 \to \Gamma$. The associated $2$-group~$\cC_{\Gamma}$ has trivial $2$-morphisms and $\Gamma$ as its space of $1$-morphisms. We can therefore consider it as a topological $1$-category. Projecting $(\alpha_{ij})_{0 \leq i < j \leq n}$ to $(\alpha_{01}, \alpha_{12}, \dots, \alpha_{n-1,n}) \in \Gamma^n$ provides a homeomorphism 
\[
	N^D_n(\cC_\Gamma) \cong \Gamma^n = N_n(\mathcal{C}_\Gamma)\ .
\]
identifying it with the ordinary ($1$-categorical) nerve of $\mathcal{C}_\Gamma$. It is straightforward to check that this homeomorphism is compatible with the simplicial structure. In particular,
\[
	B^D\cC_{\Gamma} = \lVert N^D_{\ast}(\cC_{\Gamma}) \rVert \cong \lVert N_{\ast}(\cC_{\Gamma}) \rVert = \cB\Gamma\ .
\]
%By \cite[Proposition 2]{paper:tomDieckBG} the space $\lVert N_{\ast}(\cC_{\Gamma}) \rVert$ is homotopy equivalent to the Milnor's classifying space \cite[$\S3$, page 107]{paper:SegalBG}, which explains the last equality.

Let $(\alpha,u)$ be a $1$-cocycle on the discrete group $\Gamma$ with values in a crossed module $\cG$. Up to cohomology we may assume that it is normalised, i.e.\ that $\alpha_{e} = 1_G$ and $u_{g,e} = u_{e,g} = 1_H$ for all $g \in \Gamma$. Let $\theta_{(\alpha,u)} \colon \cC_{\Gamma} \to \cG$ be the strictly unital pseudofunctor defined as follows: $\theta^0_{(\alpha,u)}$ is the trivial map on a single point, the functor $\theta^1_{(\alpha,u)}$ maps the $1$-morphism $g \in \Gamma$ to the $1$-morphism $\alpha_g$. Since there are only identity $2$-morphisms in $\cC_{\Gamma}$ this fixes $\theta^1_{(\alpha,u)}$ completely. The normalisation ensures that $\theta^1_{(\alpha,u)}(\id{}) = \alpha_e = 1_G = \id{}$. Let $\eta_{g,h} = u_{g,h}$, then this provides a natural transformation
\[
\begin{tikzcd}
    u_{g,h} \colon \theta^1_{(\alpha,u)}(g \circ h) = \alpha_{gh} \ar[r,Rightarrow] & \partial(u_{g,h}) \,\alpha_{gh} = \alpha_g\,\alpha_h = \theta^1_{(\alpha,u)}(g)\theta^1_{(\alpha,u)}(h)
\end{tikzcd}
\]
Composition with the pseudofunctor $\theta_{(\alpha,u)} \colon \cC_\Gamma \to \cG$ induces a morphism of simplicial spaces
\[
	N_\ast(\cC_{\Gamma}) \to N^D_\ast(\cG)\ .
\]
Hence, after geometric realisation we end up with a continuous map
\[
	\cB\Gamma = \lVert N_\ast(\cC_{\Gamma})\rVert \to \lVert N_\ast^D(\cG) \rVert = B^D\cG\ .
\]
Denote the set of all normalised $1$-cocycles on $\Gamma$ by $Z^1_{\text{nor}}(\Gamma, \cG)$, then the above construction provides a map 
\begin{equation}\label{eqn:maponcocycles}
	Z^1_{\text{nor}}(\Gamma, \cG) \to [B\Gamma, B^D\cG]\ .
\end{equation}
Which is natural in both $\Gamma$ and $\cG$.\footnote{Recall that a morphism of crossed modules $\Phi:\cG\rightarrow \cH$ induces a continuous strictly unital pseudofunctor of the underlying 2-categories. Thus it induces a map $[B\Gamma,B^D\cG]\rightarrow [B\Gamma,B^D\cH]$ by postcomposing a continuous function $\cB\Gamma\rightarrow B^D\cG$ with the continuous map $B^D\cG\rightarrow B^D\cH$ induced by $\Phi$. Naturality in $\cG$ hence follows from functoriality of the geometric realisation.} Note that there is no ambiguity in choosing $B\Gamma$ or $\cB\Gamma$ in \eqref{eqn:maponcocycles} as since $\Gamma$ is discrete both canonically homotopy equivalent. The next lemma shows that this map induces a well-defined map on cohomology classes. 
\begin{lemma} \label{lem:map_on_BG}
    Let $\cG$ be the topological $2$-group associated to a crossed module $\partial \colon H \to G$ and let $\Gamma$ be a discrete group. The above construction gives rise to a map of pointed sets
	\[
		H^1(\Gamma, \cG) \to [B\Gamma, B^D\cG]
	\]
	which is natural both in $\Gamma$ and in $\cG$.
\end{lemma}

\begin{proof}
    Let $I$ be the (discrete) category with object set $\{0,1\}$ and one non-identity morphism $0 \to 1$. It suffices to see that a $1$-coboundary $(\gamma, w)$ linking two $1$-cocycles $(\alpha^{(i)}, u^{(i)})$ for $i \in \{0,1\}$ as described in \eqref{eqn:alpha12} and \eqref{eqn:u12} gives rise to a strictly unital pseudofunctor 
	\[
		\Theta \colon \cC_{\Gamma} \times I \to \cG
	\]
	such that the restriction of $\Theta$ to the subcategories $\cC_\Gamma \times \{i\}$ for $i \in \{0,1\}$ coincides with $\theta_i := \theta_{(\alpha^{(i)}, u^{(i)})}$. Indeed, by \cite[Proposition 6.2]{book:MaySimplicial} the simplicial map $N_\ast \Theta$ induces a simplicial homotopy between $N_\ast \theta_0$ and $N_\ast \theta_1$ in the sense of \cite[Definition 5.1]{book:MaySimplicial}. By \cite[Lemma 1.15]{EBWI19} a simplicial homotopy induces a homotopy 
    \[
        H \colon B\Gamma \times [0,1] \to B^D\cG
    \]
    between $\lVert N_\ast \theta_0 \rVert$ and $\lVert N_\ast \theta_1 \rVert$.
    %Indeed, since the classifying space functor preserves products and $\lvert BI \rvert \cong [0,1]$ the functor $\Theta$ will give a continuous map
	%\[
		%B\Theta \colon B\Gamma \times [0,1] \cong B\cC_\Gamma \times BI \cong B(\cC_\Gamma \times I) \to B^D\cG \ , 
	%\]
	%such that $\left.B\Theta\right|_{B\Gamma \times \{i\}} = B\theta_{i}$, i.e.\ a homotopy between the two induced maps on classifying spaces. 
	
	The functor $\Theta$ is already fixed on the two subcategories $\cC_\Gamma \times \{i\}$ (in particular on all objects). Therefore we only need to define where $\Theta^1$ should map the morphisms $(g, 1 \leftarrow 0) \in \cC_\Gamma \times I$ and define the natural transformations
	\begin{align}
		\Theta^1(gh, 1 \leftarrow 0) &\Rightarrow \Theta^1(g, 1 \leftarrow 0) \circ \Theta^1(h, \id{0})  \ , \label{eqn:comp10-0}\\
		\Theta^1(gh, 1 \leftarrow 0) &\Rightarrow  \Theta^1(g, \id{1}) \circ \Theta^1(h, 1 \leftarrow 0) \label{eqn:comp1-10}
	\end{align}
	in such a way that associativity holds. Let 	$\Theta^1(g, 1 \leftarrow 0) = \gamma\,\alpha_g^{(0)}$. Note that   
	\begin{align*}
		\Theta^1(g, 1 \leftarrow 0) \circ \Theta^1(h, \id{0}) &= \gamma\,\alpha_g^{(0)}\,\alpha_h^{(0)} = \partial(\gamma(u^{(0)}_{g,h}))\,\gamma\,\alpha^{(0)}_{gh} \ ,\\ 
		\Theta^1(gh, 1 \leftarrow 0) &= \gamma\,\alpha_{gh}^{(0)}\ .
	\end{align*}
	Therefore we define \eqref{eqn:comp10-0} to be $\gamma(u_{g,h}^{(0)})$. In a similar spirit we have 
	\begin{align*}
		\Theta^1(g, \id{1}) \circ \Theta^1(h, 1 \leftarrow 0) &= \alpha_g^{(1)}\,\gamma\,\alpha^{(0)}_h = \partial(w_g)\,\gamma\,\alpha^{(0)}_g\,\alpha^{(0)}_h \\ 
		&= \partial(w_g\,\gamma(u^{(0)}_{g,h}))\,\gamma\,\alpha^{(0)}_{gh} \ .
	\end{align*}
	Hence, we define \eqref{eqn:comp1-10} to be $w_g\,\gamma(u_{g,h}^{(0)})$. Associativity needs to be checked for the following compositions:
	\begin{align}
		\Theta^1(g, \id{1}) \circ \Theta^1(h, 1 \leftarrow 0) \circ \Theta^1(k, \id{0}) \label{eqn:as-1-10-0} \\
		\Theta^1(g, 1 \leftarrow 0) \circ \Theta^1(h, \id{0}) \circ \Theta^1(k, \id{0}) \label{eqn:as-10-0-0}\\
		\Theta^1(g, \id{1}) \circ \Theta^1(h, \id{1}) \circ \Theta^1(k, 1 \leftarrow 0)
			\label{eqn:as-1-1-10} 
 	\end{align}
 	Associativity in \eqref{eqn:as-1-10-0} and \eqref{eqn:as-10-0-0} easily follows from \eqref{eqn:u} using \eqref{eqn:alpha12} where needed. Checking associativity in \eqref{eqn:as-1-1-10} leads to the equation
 	\begin{equation} \label{eqn:as-comparison}
            \alpha_g^{(1)}\!\left(w_h\gamma(u_{h,k}^{(0)})\right) w_g \gamma(u_{g,hk}^{(0)}) = u_{g,h}^{(1)}w_{gh} \gamma(u_{gh,k}^{(0)})
        \end{equation}
    Replacing $\alpha_g^{(1)}$ using \eqref{eqn:alpha12} gives
    \[
        w_g\gamma(\alpha_g^{(0)}(\gamma^{-1}(w_h)))\,\gamma(\alpha_g^{(0)}(u_{h,k}^{(0)}) u_{g,hk}^{(0)}) = w_g\gamma(\alpha_g^{(0)}(\gamma^{-1}(w_h))u_{g,h}^{(0)})\,\gamma(u_{gh,k}^{(0)})
    \]
    where we applied \eqref{eqn:u}. By \eqref{eqn:u12} this is the same as the right hand side of \eqref{eqn:as-comparison}. Therefore associativity holds in all cases, which finishes the proof. 
\end{proof}

\begin{remark}
	What we really achieved in the proof of Lemma \ref{lem:map_on_BG} is to show that two cohomologous cocycles $(\alpha^{(i)}, u^{(i)})$ give rise to naturally isomorphic pseudofunctors $\theta_0$ and $\theta_1$ in an appropriate $2$-categorical sense. This natural transformation then gives rise to the corresponding homotopy. We have just combined these two steps into a single one.
\end{remark}
The key point to constructing \eqref{eqn:maponcocycles} is that any $1$-cocycle $(\alpha,u):\cC_{\Gamma}\rightarrow \cG$ induces a strictly unital pseudofunctor $\theta_{(\alpha,u)}:\cC_{\Gamma}\rightarrow \cG$. The functor $\theta^{1}_{(\alpha,u)}$ is always faithful, and it is full if and only if $\alpha_g$ is not in the image of $\partial$ for all $g\neq 1$. This motivates the following variation of the $1$-cohomology sets with coefficients in a crossed module.
\begin{definition}\label{def:ffcohomology}
Let $\cG = (H,G, \partial)$ be a topological crossed module. We let $Z^1_{ff}(\Gamma,\cG)$ be the set of $(\alpha,u)\in Z^1(\Gamma,\cG)$ such that $\alpha_g$ is not in the image of $\partial$ for all $g\neq 1$. We let
\[
H^1_{ff}(\Gamma,\cG_A)= Z^1_{ff}(\Gamma,\cG)/\!\sim.
\]
\end{definition}
\begin{remark}\label{rmk: maptoBGff}
The inclusion $Z^1_{ff}(\Gamma,\cG_A)\rightarrow Z^1(\Gamma,\cG_A)$ passes to an injective, natural map of sets
\[
H^1_{ff}(\Gamma,\cG_A)\rightarrow H^1(\Gamma,\cG_A).
\]
Thus after composing with the map of Lemma \ref{lem:map_on_BG} we get a canonical map of sets
\[H^1_{ff}(\Gamma,\cG)\rightarrow [B\Gamma,B^D\cG]\]
which is natural in both $\Gamma$ and $\cG$.
\end{remark}
%which is natural in $\Gamma$.
%and so one also has a 
%By restricting oneself to $Z^1_{ff}(\Gamma,\cG)$ then Lemma \ref{lem:map_on_BG} also gives a canonical map of sets

There is a second way of defining the classifying space of a topological crossed module $\partial \colon H \to G$. Let $\cG_\otimes$ be the topological $1$-category obtained as the endomorphism category of the one object $\ast$ in the $2$-category $\cG$. Thus, $\cG_\otimes$ has object space $G$ and morphism space $G \rtimes H$. The composition of morphisms is given by $\circ_v$. Note that the composition in $G$ induces a strict monoidal structure on $\cG_\otimes$ by
\[
	\alpha \otimes \beta = \alpha \cdot \beta
\]
on objects and $(\alpha, u) \otimes (\beta,v) = (\alpha, u) \circ_h (\beta, v)$ on morphisms. This monoidal structure equips the nerve $N_\ast \cG_\otimes$ with the structure of a simplicial topological group. Hence, its geometric realisation is a topological monoid, in fact a group \cite[Corollary 11.7]{book:MayIterated}. 
\begin{definition} \label{def:monoidal_BG}
    Let $\cG$ be a topological crossed module. The space $B^\otimes \cG$ defined by 
    \[
		B^\otimes \cG = \lVert N_\ast\,\lvert N_\ast \cG_{\otimes} \rvert\,\rVert = \cB\,\lvert N_\ast \cG_{\otimes} \rvert
    \]
    is called the \emph{monoidal classifying space} of $\cG$. 
\end{definition}
Moreover, as any morphism of crossed modules gives a continuous monoidal functor of underlying $1$-categories the construction above yields a functor from the category of crossed modules to $\Top$. We defer the proof of the following theorem to the appendix.
\begin{theorem}\label{thm:homotpyeqclassifyingspaces}
    Let $\cG = (\partial \colon H \to G)$ be a topological crossed module with $H$ well-pointed. Then the classifying spaces $B^{D}\cG$ and $B^{\otimes}\cG$ are weakly homotopy equivalent.
\end{theorem}

 \subsection{The homotopy type of $B^{\otimes}\cG_A$} \label{sec:BGA}
\par We denote by $A^s:=A\otimes \bK$ the stabilisation of $A$. Let $e_{ij}$ for $i,j\in \bN$ be the matrix units in $\bK = \bK(\ell^2(\N))$, i.e. they satisfy $e_{ij}e_{lk}=\delta_{j,l}e_{ik}$ and $e_{ij}^*=e_{ji}$. We set $e=e_{11}$ and denote by 
 \[
 \Aute \As=\{\alpha\in \Aut \As : \alpha(1\otimes e)=1\otimes e\}
 \]
 and 
 \[
 \Autz {\As}=\{ \alpha\in \Aut{\As}: \alpha(1\otimes e)\sim_{h} 1\otimes e\}.
 \]
Note that the connected component of the identity in $\Aut{A^s}$ is contained in $\Autz{A^s}$ and that when $A$ is strongly self-absorbing the two coincide (see \cite[Theorem 2.5]{DAPE16}).
\par In this subsection we show that for any discrete group $\Gamma$ there is a natural bijection of pointed sets between $[B\Gamma,B^{D}\cG_A]$ and $[B\Gamma,\cB\Autz{A^s}]$. We first require some preliminary results. The following Lemma is a routine computation.
%\par Recall that for a C$^*$-algebra $A$ the strict topology on $M(A)$ is defined by the seminorms
%\[
	%\rho_S(T) = \lVert T S\rVert \qquad \text{and} \qquad \rho_S^*(T) = \lVert T^* S \rVert = \lVert S^* T \rVert.
%\]
 \begin{lemma} \label{lem:alpha_on_cpts}
 Let $\alpha \in \Aute{\As}$. The sum 
\begin{equation} \label{eqn:v-alpha}
	v_{\alpha} = \sum_{i \in \N} \alpha(1 \otimes e_{i1})\,(1 \otimes e_{1i})
\end{equation}
    is bounded and converges in the strict topology on $M(A^s)$. It defines a unitary element $v_{\alpha} \in UM(A^s)$. Moreover, for all $T \in \bK$ we have 
	\[
		v_{\alpha} (1 \otimes T) v_{\alpha}^* = \alpha(1 \otimes T)\ .
	\]
\end{lemma}

Let $\alpha \in \Aute {\As}$. Since $\alpha(a \otimes e)$ commutes with $1 \otimes e$, it maps the algebra $(1 \otimes e)(\As)(1 \otimes e)$ isomorphically onto itself. The corner inclusion 
\[
	A \to (1 \otimes e)(\As)(1 \otimes e) \quad , \quad a \mapsto a \otimes e 
\]
is an isomorphism. When restricted along this isomorphism, $\alpha$ gives rise to an element $\bar{\alpha} \in \Aut{A}$. The above lemma implies the following: 
\begin{corollary} \label{cor:conjugation}
	Let $\alpha \in \Aute {\As}$ and let $v_{\alpha} \in UM(A^s)$ be defined as in \eqref{eqn:v-alpha}. Then 
	\[
		\alpha = \Ad_{v_{\alpha}} \circ\, (\bar{\alpha} \otimes \id{\bK})
	\]
\end{corollary}

\begin{proof}
	Since the finite rank operators $\sum_{i,j=1}^n a_{ij} \otimes e_{ij}$ are dense in $\As$, it suffices to show the claim for elements of the form $a \otimes e_{kl}$. First consider the case $k = l = 1$. 
	\begin{align*}
		&(\Ad_{v_{\alpha}} \circ\, (\bar{\alpha} \otimes \id{\bK}))(a \otimes e_{11})\\
        =\ & v_{\alpha} \,(\bar{\alpha}(a) \otimes e_{11})\, v_{\alpha}^* \\
		=\ & \sum_{i,j=1}^\infty \alpha(1 \otimes e_{i1}) (1 \otimes e_{1i}) (\bar{\alpha}(a) \otimes e_{11}) (1 \otimes e_{j1}) \alpha(1 \otimes e_{1j}) \\
		=\ & \alpha(1 \otimes e_{11}) (\bar{\alpha}(a) \otimes e_{11}) \alpha(1 \otimes e_{11})\\
        =\ & \bar{\alpha}(a) \otimes e_{11}\ .
	\end{align*}
	The general case follows from the above together with Lemma \ref{lem:alpha_on_cpts}:
	\begin{align*}
		&(\Ad_{v_{\alpha}} \circ\, (\bar{\alpha} \otimes \id{\bK}))(a \otimes e_{kl}) = v_{\alpha} \,(1 \otimes e_{k1})(\bar{\alpha}(a) \otimes e_{11})(1 \otimes e_{1l})\, v_{\alpha}^*	 \\
		%=\ & \alpha(1 \otimes e_{k1})\,v_{\alpha}\,(\bar{\alpha}(a) \otimes e_{11})\,v_{\alpha}^*\,\alpha(1 \otimes e_{1l}) \\
		=\ & \alpha(1 \otimes e_{k1})\,\alpha(a \otimes e_{11})\,\alpha(1 \otimes e_{1l}) = \alpha(a \otimes e_{kl}) \qedhere
	\end{align*} 
\end{proof}

\begin{lemma}\label{lem:ctstostrict}
	Let $A$ be a unital separable $C^*$-algebra. The map 
	\[
		\psi \colon \Aute {\As} \to UM(A^s) \quad , \quad \alpha \mapsto v_{\alpha}
	\]
	with $v_{\alpha}$ defined in \eqref{eqn:v-alpha} is continuous when the domain is equipped with the point-norm topology and the codomain with the strict topology.
\end{lemma}

\begin{proof}
	Let $(\alpha_i)_{i \in I}$ with $\alpha_j \in \Aute {\As}$ be a net converging to $\alpha \in \Aute {\As}$ in the point-norm topology. To see that $v_{\alpha_i}x \to v_{\alpha}x$ in norm for all $x \in \As$, it suffices again to check this for $x = a \otimes e_{kl}$ by the density of the linear span of these operators in $\As$. But we have 
	\begin{align*}
		 & \lVert v_{\alpha_i}(a \otimes e_{kl}) - v_{\alpha}(a \otimes e_{kl}) \rVert \\
		 =\ & \lVert \alpha_i(1 \otimes e_{k1})\,(a \otimes e_{1l}) - \alpha(1 \otimes e_{k1})\,(a \otimes e_{1l}) \rVert
	\end{align*}
	and this converges to $0$, because $\lVert \alpha_i(1 \otimes e_{k1}) - \alpha(1 \otimes e_{k1}) \rVert$ does. This argument also shows $v_{\alpha_i}^*x \to v_{\alpha}^*x$ by unitarity of $v_{\alpha}$.
\end{proof}

\begin{lemma}\label{lem:fibrationstrict}
    Let $A$ be a unital separable C$^*$-algebra. Then the canonical map
    \begin{align*}
    \Aut{A}\times UM(\As) &\rightarrow \Autz {\As}\\
    (\alpha,u)&\mapsto \Ad_u\circ\ \alpha\otimes \id{\bK},
    \end{align*}
    where $UM(\As)$ is equipped with the strict topology, is a Hurewicz fibration.
\end{lemma}
\begin{proof}
    As $A$ is separable $\Autz{A^s}$ is metrisable and hence paracompact. Thus it suffices to show that the map $\pi: \Aut{A}\times UM(A^s)\rightarrow \Autz \As$ admits a continuous local section for an open set around $\id \As$. Let
    \[
    \cU=\{\alpha\in \Autz \As :\|\alpha(1\otimes e)-1\otimes e\|<1/2\}
    \]
    which is an open set around $\id \As$. By \cite[Proposition II.3.3.5]{BL06} there exists a unitary $u_{\alpha}\in UM(\As)$ for all $\alpha \in \cU$ such that $\Ad_{u_\alpha} \alpha\in \Aute{\As}$ and such that the assignment $\alpha\rightarrow u_{\alpha}$ is continuous from the point norm to the norm topology (and hence also to the strict topology). We denote by 
    \begin{align*}
    s:\cU&\rightarrow \Aute {\As}\times UM(A^s)\\
    \alpha&\mapsto (\Ad_{u_\alpha} \alpha,u_{\alpha}^*)
    \end{align*}
    which is continuous from the point norm topology to the product of the point norm and strict topology. Moreover, the map
    \begin{align*}
        f:\Aute{A^s}\times UM(\As)&\rightarrow \Aut{A}\times UM(\As)\\
        (\beta,u)&\mapsto (\overline{\beta},uv_\beta)
    \end{align*}
    is continuous by Lemma \ref{lem:ctstostrict}. We claim that $f\circ s$ is a section of $\pi$ on $\cU$. Indeed for $\alpha\in \cU$ and denoting by $\beta=\Ad_{u_\alpha}\alpha$ one has that
    \begin{align*}
        \pi(f\circ s(\alpha))&=\pi(\overline{\beta},u_\alpha^*v_{\beta})\\
        &=\Ad_{u_\alpha^*}\beta\\
        &=\alpha.
    \end{align*}\end{proof}
\begin{lemma}\label{lem:crossedmodulestrict}
    Let $A$ be a separable, unital C$^*$-algebra. Then the action of $\Autz \As$ on $UM(\As)$ is continuous when $UM(\As)$ is equipped with the strict topology. Thus the map
    \[
    \Ad:UM(\As)\rightarrow \Autz \As
    \]
    with $UM(\As)$ equipped with the strict topology defines a crossed module which we denote by $\cG_A^{0}$. Moreover $UM(A^s)$ is well pointed in the strict topology.
\end{lemma}
\begin{proof}
    Firstly, the map
    \begin{align}\label{eqn:strictcont1}
    \Aut{A}\times UM(A^s)&\rightarrow UM(A^s)\\
    (\alpha,u)&\mapsto (\alpha\otimes \id{\bK})(u) \notag
    \end{align}
    is jointly continuous. Indeed, let $\alpha_n\in \Aut{A}$ be a sequence converging to $\alpha\in \Aut{A}$ and $u_n\in UM(A^s)$ be a sequence converging to $u\in UM(A^s)$.\footnote{Note that as $A$ is separable both $\Aut{A}$ and $UM(A^s)$ are metrisable (also when $UM(A^s)$ is equipped with the strict topology).} Let $p_m$ be a sequence of projections in $\bK$ converging strongly to $1_{B(H)}$.  As $1\otimes p_m$ is an approximate unit for $\As$ and the sequence $\alpha_n\otimes \id{\bK}(u_n)$ is bounded, it suffices to show that for each $m\in \bN$ both
    \[
    \|\alpha_n\otimes \id{\bK}(u_n) (1\otimes p_m)-\alpha\otimes \id{\bK}(u) (1\otimes p_m)\|\longrightarrow 0
    \]
    and 
    \[
    \|(1\otimes p_m)\alpha_n \otimes \id{\bK}(u_n)-(1\otimes p_m)\alpha \otimes \id{\bK}(u)\|\longrightarrow 0
    \]
     as $n$ tends to infinity. We only show the first of these relations as the second follows in the same manner.
     \begin{align*}
        & \|\alpha_n\otimes \id{\bK}(u_n)1\otimes p_m-\alpha\otimes \id{\bK}(u)(1\otimes p_m)\|\\
        &\leq \|\alpha_n\otimes \id{\bK}((u_n-u)1\otimes p_m)\|+\|\alpha_n\otimes \id{\bK}(u(1\otimes p_m))-\alpha\otimes \id{\bK}(u(1\otimes p_m))\|\\
        &\leq \|(u_n-u)1\otimes p_m\|+\|\alpha_n\otimes \id{\bK}(u(1\otimes p_m))-\alpha\otimes \id{\bK}(u(1\otimes p_m))\|\\
        &\longrightarrow 0.
     \end{align*}
     Now, to show the lemma, we need to show that the map 
     \begin{align}\label{eqn:strictcont2}
     \Autz{A^s}\times UM(A^s)\rightarrow UM(A^s)\\
     (\alpha,u)\mapsto \alpha(u)\notag
     \end{align}
     is jointly continuous. For this it suffices to show that if $\alpha_n\rightarrow \id{A^s}\in \Autz{A^s}$ and $u_n\rightarrow u\in UM(A^s)$ strictly then $\alpha_n(u_n)\rightarrow u$ strictly. For large enough $n$ we may assume that the sequence $\alpha_n$ is contained in a small enough open neighbourhood $\cV$ such that there exists a continuous section $s:\cV\rightarrow \Aut{A}\times UM(A^s)$ to the map $\pi:\Aut{A}\times UM(A^s)\rightarrow \Autz{A^s}$ of Lemma \ref{lem:fibrationstrict}. 
     \par Let $\beta_n\in \Aut{A}$ and $v_n\in UM(A^s)$ be such that $s(\alpha_n)=(\beta_n,v_n)$. By continuity of $s$ one has that $\beta_n$ converges to $\id{A}$ and $v_n$ converges to $1$ strictly. Thus, for large enough $n$,
     \[
     \alpha_n(u_n)=\Ad_{v_n}\circ\beta_n\otimes\id{\bK}(u_n)\rightarrow u
     \]
    by the joint continuity of \eqref{eqn:strictcont1}.
    \par We now turn to the well pointedness of $UM(A^s)$. We will show that $(UM(A^s),1_{M(A^s)})$ is an NDR pair as in \cite[Definition 2.24]{DAPE16}. First note that as $A^s$ is separable then $UM(A^s)$ is metrisable in the strict topology with metric $d$. Moreover, there is a deformation retract from $UM(A^s)$ to $1_{M(A^s)}$ i.e. a homotopy $H:UM(A^s)\times [0,1]\rightarrow UM(A^s)$ such that $H(u,0)=u$, $H(u,1)=1_{M(A^s)}$ for all  $u\in UM(A^s)$ and  $H(1_{M(A^s)},t)=1_{M(A^s)}$ for all $t\in [0,1]$ (see \cite[Exercise 2.M]{WEGGE93}). Thus let 
    \begin{align*}
    f:UM(A^s)&\rightarrow [0,1]\\
    u&\mapsto \min\{d(u,1),1\}
    \end{align*}
    then the pair $f,H$ witnesses that $UM(A^s)$ is well pointed.
 \end{proof}
Before we proceed to the main result of this section, we recall the notion of a comma category associated to a topological functor. Let $\cC$ and $\cD$ be topological $1$-categories, $d$ an object in $\cD$ and $F:\cC\rightarrow \cD$ a continuous functor. The \emph{comma category $F/d$} is the topological category whose object space consists of pairs $(c,f)$ with $c$ an object in $\cC$ and $f:F(c)\rightarrow d$ a morphism in $\cD$. A morphism between $(c,f)$ and $(c',g)$ consists of a morphism $l$ in $\cC$ from $c$ to $c'$ such that the diagram
\[
\begin{tikzcd}
    F(c)\ar[r,"f"]\ar[d,"F(l)" swap] & d\\
    F(c')\ar[ur,"g" swap]
\end{tikzcd}
\]
commutes.
\begin{theorem}\label{thm:weakhomeq}
    Let $A$ be a unital, separable C$^*$-algebra and $\Gamma$ a discrete group. There is a zig-zag of weak homotopy equivalences
    \[
    B^{\otimes}\cG_A\rightarrow B^\otimes \cG_A^0\leftarrow \cB\Autz{A^s}.
    \]
    Thus there is a bijection of pointed sets
    \[[B\Gamma,B^{D}\cG_A]\rightarrow [B\Gamma, \cB\Autz{A^s}]
    \]
    which is natural in $\Gamma$.
\end{theorem}
\begin{proof}
    We will apply the topological version of Quillen's Theorem A (\cite[Theorem 4.7]{EBWI19}). Before we do so we discuss how to understand the hypothesis of Quillen's Theorem A in our setting. If $\cG_1$ and $\cG_2$ are crossed modules then any morphism 
    \[\Psi=(\psi_1,\psi_2):\cG_1=(\partial_1,H_1,G_2)\longrightarrow \cG_2=(\partial_2,H_2,G_2)\]
    induces a topological functor between the underlying $1$-categories which we still denote by $\Psi$. As $\Psi$ is a monoidal functor, it induces a continuous map $\Psi:\lVert N_*\cG_1\rVert\rightarrow \lVert N_*\cG_2\rVert$. It is straightforward to see that condition (v) of \cite[Theorem 4.7]{EBWI19} (and so also (ii) and (iii)) for the functor $\Psi$ coincides with the conditions that
    \begin{align*} 
            G_1\times H_1&\rightarrow G_1\\
            (\alpha,u)&\mapsto\partial_1(u)\alpha
    \end{align*}
    and
    \begin{align*}
            G_1\times H_2&\rightarrow G_2\\
            (\alpha,u)&\mapsto \partial_2(u)\phi_2(\alpha)
    \end{align*}
    are Hurewicz fibrations. 
    \par Consider the morphisms of crossed modules $\Phi=(\phi_1,\phi_2):\cG_A\rightarrow \cG_A^{0}$ with $\phi_1(u)=u\otimes 1_{B(\cH)}$ and $\phi_2(\alpha)=\alpha\otimes\id{\bK}$; and $\iota=(\iota_1,\iota_2):\Autz{A^s}\rightarrow \cG_A^0$ with $\iota_1(1)=1_{M(A^s)}$ and $\iota_2(\alpha)=\alpha$. We check that these satisfy the conditions of \cite[Theorem 4.7]{EBWI19}. By the proceeding discussion, $\iota$ satisfies conditions (ii) and (iii) of \cite[Theorem 4.7]{EBWI19}. That $\Phi$ satisfies conditions (ii) and (iii) follows from Lemma \ref{lem:fibrationstrict}. Also, for any $\alpha\in \Autz{A^s}$ the category $\iota/\alpha$ has as objects $u\in UM(A^s)$ and only identities as morphisms. The classifying space $\cB(\iota/\alpha)$ is hence homeomorphic to $UM(A^s)\times \Delta^\infty$ which is contractible as $UM(A^s)$ is contractible in the strict topology. Finally, since $\Phi$ is essentially surjective, then for any $\beta\in \Autz{A^s}$ there exists $\alpha\in \Aut{A}$ such that $\Phi/\beta$ is equivalent to $\Phi/(\alpha\otimes\id{\bK})$. Moreover, $\Phi/(\alpha\otimes \id{\bK})$ has object space $u\in U(A)$ and a unique morphism between any two objects. Any such category is equivalent to the trivial category with one object and one morphism and thus has contractible classifying space. 
    In particular $\Phi$ and $\iota$ satisfy the condition of \cite[Theorem 4.7]{EBWI19} and thus the continuous maps $\lVert\Phi\rVert\colon\lVert N_{\ast}\cG_A\rVert\rightarrow \lVert N_{\ast} \cG_A^0\rVert$ and $\lVert\iota\rVert \colon \lVert N_{\ast} \cC_{\Autz{A^s}}\rVert\rightarrow \lVert N_{\ast}\cG_A^0\rVert$ are weak homotopy equivalences. The unitary $U(A)$ is well-pointed as it is a Banach Lie group. Similarly, $UM(A^s)$ is well pointed in the strict topology as a consequence of Lemma \ref{lem:crossedmodulestrict}. Therefore the group homomorphisms $\lvert \Phi\rvert \colon \lvert N_{\ast}\cG_A\rvert\rightarrow \lvert N_{\ast}\cG_A^0\rvert$ and $\lvert \iota \rvert \colon \Autz{A^s}\rightarrow \lvert N_{\ast}\cG_A^0\rvert$ are also weak homotopy equivalences. As the classifying space construction preserves weak homotopy equivalences, we get the desired zig-zag of weak homotopy equivalences
    \[
    B^{\otimes}\cG_A\xrightarrow{B^{\otimes}\Phi} B^{\otimes}\cG_A^0\xleftarrow{B^{\otimes}\iota} \cB\Autz{A^s}
    \]
    which combined with Theorem \ref{thm:homotpyeqclassifyingspaces} and \cite[Proposition 4.22]{HAT02} implies the required bijection.
\end{proof}
\begin{remark}\label{rmk:compatibilityPR}
    After postcomposing the bijection of Theorem \ref{thm:weakhomeq} with the map of Lemma \ref{lem:map_on_BG}, we get a map of pointed sets from $H^1(\Gamma,\cG_A)$ to $[B\Gamma,\cB\Autz{A^s}]$. There is another canonical mapping from $H^1(\Gamma,\cG_A)$ to $[B\Gamma,\cB\Autz{A^s}]$ 
    considered in \cite{IZ23}. One can take a cocycle action on $A$, stabilize it and then apply the Packer--Raeburn stabilization trick (\cite{PARA89}) to get an action $\beta:\Gamma\rightarrow \Autz{A^s}$. After applying the classifying space functor to $\beta$, one gets the desired map $H^1(\Gamma,\cG_A)\rightarrow [B\Gamma,\cB\Autz{A^s}]$. These two maps coincide. Indeed, the second map is the composition 
    \[
    H^1(\Gamma,\cG_A)\xrightarrow{\Phi^*}H^1(\Gamma,\cG_A^0)\xrightarrow{PR} [B\Gamma,\cB\Autz{A^s}]
    \]
    where $\Phi:\cG_A\rightarrow \cG_A^0$ is the morphism of crossed modules in the proof of Theorem \ref{thm:weakhomeq} and $PR$ is induced by the Packer--Raeburn map that takes a cocycle action through $\Autz{A^s}$ and gives a unique group action through $\Autz{A^s}$ up to cocycle conjugacy (this map exists by combining \cite{PARA89} and \cite[Proposition 1.4]{paper:GabeSzabo}, it can be easily checked that the cocycle conjugacies needed to pass to this action are through $\Autz{A^s}$). By naturality of the map of Lemma \ref{lem:map_on_BG} the following diagram commutes
    \begin{equation}\label{diag:compatibilityPR}
        \begin{tikzcd}
            H^1(\Gamma,\cG_A)\ar[r]\ar[d,"\Phi^*"] & \left[B\Gamma,B^D\cG_A\right]\ar[d,"B^D\Phi"]\\
            H^1(\Gamma,\cG_A^0)\ar[r]\ar[rd,"PR" swap] &\left[B\Gamma,B^D\cG_A^0\right]\\
            & \left[B\Gamma,\cB\Autz{A^s} \right] \ar[u,"B^D\iota" swap]
        \end{tikzcd}
    \end{equation}
    which implies the required equality.
\end{remark}

\subsection{Strongly self-absorbing C$^*$-algebras}
For the remainder of this section let $A$ be a self-absorbing C$^*$-algebra and $\varphi:A\otimes A\rightarrow A$ an isomorphism. Firstly note that when $A$ is strongly self-absorbing in the sense of \cite{TOWI07} then $B^{\otimes}\cG_A$ and $\cB\Autz{A^s}$ are homotopy equivalence.
\begin{prop}
    Let $A$ be a strongly self-absorbing C$^*$-algebra. Then $B\cG_A$ and $\cB\Autz{A^s}$ are homotopy equivalent.
\end{prop}
\begin{proof}
    By the proof of Theorem \ref{thm:weakhomeq} there are weak homotopy equivalences $B^{\otimes}\cG_A\rightarrow B\cG_A^0$ and $\cB\Autz{A^s}\rightarrow B\cG_A^0$. As $A$ is strongly self-absorbing it follows from \cite[Corollary 2.9]{DAPE16} that $\cB\Autz{A^s}$ has the homotopy type of a CW-complex. As $U(A)$ and $UM(A^s)$ are both well-pointed, $N_\ast\cG_A$ and $N_\ast\cG_A^0$ are good, and hence proper simplicial spaces in the sense of Definition \ref{def:good_proper}. Note that $U(A)$ has the homotopy type of a CW-complex as it is a Banach Lie group, while $UM(A^s)$ is contractible and therefore has the homotopy type of a CW-complex as well. By \cite[Corollary A.6]{MAY74} so do $\lvert N_\ast \cG_A \rvert$ and $\lvert N_\ast \cG_A^0 \rvert$. Thus by \cite[Proposition A.1 (i)]{SE74} both $B^\otimes \cG_A$ and $B^\otimes \cG_A^0$ have the homotopy type of a CW-complex. It follows from Whitehead's theorem that $B^{\otimes}\cG_A$ and $\cB\Autz{A^s}$ are homotopy equivalent.
\end{proof}
When $A$ is self-absorbing there is a morphism of crossed modules
\[
\Phi=(\phi_1,\phi_2):\cG_A\times \cG_A\rightarrow \cG_A
\]
where $\phi_1:U(A)\times U(A)\rightarrow U(A)$ and $\phi_2:\Aut{A}\times \Aut{A}\rightarrow \Aut{A}$ with $\phi_1(u\times v)=\varphi(u\otimes v)$ and $\phi_2(\alpha\times \beta)\rightarrow \varphi\circ \alpha\otimes \beta\circ \varphi^{-1}$. In the same way one can define a morphism of crossed modules $\cG\times \cG\rightarrow \cG$ for the crossed modules $\cG=P\cG_A,S\cG_A^\tau,\tilde{\cG}_A,\cG_A^{0}$ whereby in the last case one chooses an isomorphism $\psi:(A\otimes \bK)\otimes(A\otimes \bK)\rightarrow A\otimes \bK$. The morphism $\Phi$ induces a monoid structure on $H^1(\Gamma,\cG_A)$ by the composition
\[
H^1(\Gamma,\cG_A)\times H^1(\Gamma,\cG_A)\rightarrow H^1(\Gamma,\cG_A\times \cG_A)\xrightarrow{\Phi^*} H^1(\Gamma,\cG_A)
\]
which preserves $H^1_{ff}(\Gamma,\cG_A)$. We also have a monoid structure on $[B\Gamma,B^D\cG_A]$ by
\[
[B\Gamma,B^D\cG_A]^{\times 2}\rightarrow [B\Gamma,(B^D\cG_A)^{\times 2}]\rightarrow [B\Gamma,B^D(\cG_A\times \cG_A)]\xrightarrow{B^D \Phi}[B\Gamma,B^D\cG_A].
\]
Where the third map exists by \cite[Theorem 7.2]{EBWI19} and \cite[Proposition 4.22]{HAT02}. Similarly one has monoid structures on $[B\Gamma,B^D\cG]$ and on $[B\Gamma,B^\otimes \cG]$ for $\cG=P\cG_A,S\cG_A^{\tau},\tilde{\cG}_A$ and $\cG_A^{0}$.
\begin{prop}\label{prop:homotopygroup}
Let $A$ be a strongly self-absorbing C$^*$-algebra. Then the monoid $[B\Gamma,B^D\cG_A]$ is a group.
\end{prop}
\begin{proof}
    Let $\Phi=(\phi_1,\phi_2):\cG_A\rightarrow \cG_A^0$ be the map defined by $\phi_1(u)=u\otimes 1_{B(\cH)}$ and $\phi_2(\alpha)=\alpha\otimes \id{\bK}$ and $\iota:(1\rightarrow \Autz{A^s})\rightarrow \cG_A^0$ the map which is defined as the identity on objects. Using the notation of Appendix \ref{appendix} it follows from Corollary \ref{cor:homotopyequivalence} and its proof that we have a commuting diagram
\begin{equation*}\label{diag:commutingclassifyingspaces}
\begin{tikzcd}
    \left[B\Gamma,B^{\otimes} \cG_A \right] \ar[r] \ar[d,"B^\otimes \Phi"] & \left[B\Gamma, B\cF_{\cG_A}\right]  \ar[d,"B\cF(\Phi)"]  & \left [B\Gamma,B^D\cG_A\right ] \ar[l] \ar[d,"B^D\Phi"]\\
   \left[B\Gamma,B^{\otimes} \cG_A^0\right] \ar[r] &\left[B\Gamma, B\cF_{\cG_A^0}\right]  & \left[B\Gamma,B^D\cG_A^0\right] \ar[l]\\
   \left[B\Gamma, B\Autz{A^s}\right]\ar[r] \ar[u,"B^\otimes \iota" swap] & \left[ B\Gamma,B\cF_{\cC_{\Autz{A^s}}}\right]\ar[u,"B\cF(\iota)" swap] &\left[B\Gamma,\cB\Autz{A^s}\right]\ar[l] \ar[u,"B^D\iota" swap]
\end{tikzcd}
\end{equation*}
where the horizontal arrows are isomorphisms (see also \cite[Proposition 4.22]{HAT02}). As the leftmost vertical arrows are also isomorphisms by Theorem \ref{thm:homotpyeqclassifyingspaces}, it follows that all of the maps in the diagram are isomorphisms. The set $[B\Gamma,B\cF_\cG]$ for $\cG=\cG_A,\cG_A^0,\cC_{\Autz{A^s}}$ has a canonical monoid structure by using the corresponding morphism of crossed modules $\cG\times\cG\rightarrow \cG$ and the functoriality of $B\cF$. It follows from naturality that all of the maps in the diagram above are maps of monoids. Thus, since $[B\Gamma,\cB\Autz{A^s}]$ is a group by \cite[Corollary 3.9]{DAPE16}, it follows that $[B\Gamma,B^D\cG_A]$ is also a group. \end{proof}
When $\Gamma$ is a countable discrete amenable torsion-free group with a finite CW-complex model for $B\Gamma$ then the monoid $H^1_{ff}(\Gamma,\cG_A)$ is a group by \cite[Section 4.3]{IZ23}.
\begin{theorem} \label{thm:BG_group_hom}
    Let $A$ be a strongly self-absorbing Kirchberg algebra and $\Gamma$ be a countable discrete amenable torsion-free group with a finite CW-complex model $B\Gamma$. Then the map
    \[
    H^1_{ff}(\Gamma,\cG_A)\rightarrow [B\Gamma,B^D\cG_A]
    \]
    of Remark \ref{rmk: maptoBGff} is a group isomorphism.
\end{theorem}
\begin{proof}
    Firstly note that by the naturality in the coefficient crossed module, the map $H^1_{ff}(\Gamma,\cG_A)\rightarrow [B\Gamma,B^D\cG_A]$ is a unital map of monoids (and hence of groups). Consider the corresponding commuting diagram to \ref{diag:compatibilityPR} when restricted to $H^1_{ff}(\Gamma,\cG_A)$. This map consists of homomorphisms of groups with the rightmost column consisting of isomorphisms as shown in Theorem \ref{thm:weakhomeq}. Also the composition
    \[
    H^1_{ff}(\Gamma,\cG_A)\xrightarrow{\Phi^*}H^1(\Gamma,\cG_A^0)\xrightarrow{PR} [B\Gamma,\cB\Autz{A^s}]
    \]
    is an isomorphism by \cite[Corollary 4.6]{IZ23}. The result now follows from a diagram chase.
\end{proof}
 \section{The universal cover of $SU_\tau(A)$}\label{sec:covSU}
 It is well known that for unitary groups of matrix algebras one has an isomorphism $\covU{n}\cong \R\times SU(n)$. In this section we show that, under some conditions, one also has an analogous isomorphism $\covU A\cong \R \times SU_\tau(A)$. Throughout the section we assume that $A$ is a separable, unital C$^*$-algebra with a connected unitary group and that $\tau$ is a tracial state of $A$ satisfying $SU_\tau(A)=\ker(\Delta_\tau)\cap U(A)_0$.
\begin{lemma}\label{lem:continuity}
    Let $X$ be a locally path connected space and $A$ be a separable C$^*$-algebra. Let $f:X\rightarrow SU_\tau(A)$ be a map then $f$ is continuous with respect to the norm topology in $SU_\tau(A)$ if and only if it is continuous with respect to the $d_\tau$ topology.
\end{lemma}
\begin{proof}
   As the $d_\tau$ topology is stronger than the norm topology the if direction follows. Assume $f$ is continuous with respect to the norm topology. To complete the proof it suffices to show that for any $0<\varepsilon<1$ and $x_0\in X$ the set $\{x\in X: d_{\tau}(f(x_0),f(x))<\varepsilon\}$ contains an open neighbourhood of $x_0$. By norm continuity of $f$, there exists a path connected neighbourhood $V$ of $x_0$ such that $\|f(x_0)-f(x)\|<\frac{2\epsilon}{\pi}$ for all $x\in V$. Now, for $x\in V$ one may choose a continuous path $x(t):[0,1]\rightarrow V$ with $x(0)=x_0$ and $x(1)=x$. The function
    $$t\mapsto \frac{\tau(\log(f(x_0)^{-1}f(x(t)))}{2\pi i}$$
    is norm continuous and valued in $\tau_{*}(K_0(A))$. As $A$ is separable, the subgroup $\tau_{*}(K_0(A))$ is a countable subset of $\R$ and hence completely disconnected. Therefore, this function is constant and $\tau(\log(f(x_0)^{-1}f(x)))=0$ for all $x\in V$. Moreover, as an application of spectral mapping in functional calculus, it is clear that $\|\log(f(x_0)^{-1}f(x))\|< \pi/2$ and thus one may apply Lemma \ref{lem:estimatenorms} to yield that for all $x\in V$
    $$d_\tau(f(x_0),f(x))\leq \frac{\pi}{2}\|f(x_0)-f(x)\|< \epsilon$$
    as required.
\end{proof}
\begin{lemma}\label{lem: representingunitaries}
    Let $X$ be a compact, contractible, locally path connected and let $x_0\in X$. Let $f:(X_,x_0)\rightarrow (U(A)_0,1_A)$ be a continuous map. Then there exists a unique pair of continuous maps $g:(X,x_0)\rightarrow (SU_\tau(A),1_A)$ and $h:(X,x_0)\rightarrow (\R,0)$ such that $f(x)=e^{2\pi i h(x)}g(x)$.
\end{lemma}
\begin{proof}
As $X$ is compact we may consider $f$ as an element in $U(C(X,A))$. Moreover, as $X$ is contractible and $f(x_0)=1_A$, it follows that $f$ is homotopic to $1_A$ inside $U(C(X,A))$. Therefore, there exists continuous maps $\varphi_i \in C(X,A^{sa})=C(X,A)^{sa}$ for $i=1,2,\ldots ,n$ such that $f(x)=e^{i\varphi_1(x)}\ldots e^{i\varphi_n(x)}$ and
$$f(x)=e^{i(\varphi_1(x)-\tau(\varphi_1(x)))}\ldots e^{i(\varphi_n(x)-\tau(\varphi_n(x)))}e^{i\sum_{j=1}^n\tau(\varphi_j(x))}.$$
The mappings $g_0(x)=e^{i(\varphi_1(x)-\tau(\varphi_1(x)))}\ldots e^{i(\varphi_n(x)-\tau(\varphi_n(x)))}$ and $h_0(x)=\frac{\sum_{j=1}^n\tau(\varphi_j(x))}{2\pi}$ for $x\in X$ are continuous from $X$ to $SU_\tau(A)$ ($g_0$ is by construction continuous in norm topology, it will hence be continuous in $d_\tau$ topology by Lemma \ref{lem:continuity}) and $X$ to $\R$ respectively and satisfy the conditions of the lemma apart from preserving the basepoints.
\par As $1=f(x_0)=g_0(x_0)e^{2\pi i h_0(x_0)}$ one has that $e^{2\pi i h_0(x_0)}\in SU_\tau(A)$. Thus
$$f(x)=g_0(x)e^{2\pi ih_0(x_0)}e^{2\pi i (h_0(x)-h_0(x_0))}$$
with $g(x)=g_0(x)e^{2\pi ih_0(x_0)}\in SU_\tau(A)$ and $h(x)=h_0(x)-h(x_0)$ now also preserving the basepoints.
\par If $g,g':(X,x_0)\rightarrow (SU_\tau(A),[1_A])$ and $h,h':(X,x_0)\rightarrow (\R,1)$ are continuous maps such that 
$$g(x)e^{2\pi i h(x)}=f(x)=g'(x)e^{2\pi i h'(x)}$$
then $g(x)g'(x)^{-1}=e^{2\pi i(h(x)-h'(x))}\in \ker(\Delta_\tau)$. So $g(x)g'(x)^{-1}$ is a continuous map into $e^{2\pi i\tau_{*}(K_0(A))}\subset SU_\tau(A)$. As $e^{2\pi i \tau_{^*}(K_0(A))}$ is discrete (see Remark \ref{rem:discrete}) and $X$ is connected, then $g(x)g'(x)^{-1}$ is constant and everywhere $1_A$. So $h(x)-h'(x)=h(x_0)-h'(x_0)=0$ for all $x\in X$ too.
\end{proof}
\begin{theorem}\label{thm:universalcover1}
 Let $A$ be a unital, separable C$^*$-algebra and $\tau\in T(A)$ such that $U(A)=U(A)_0$ and $SU_\tau(A)=\ker(\Delta_\tau)\cap U(A)_{0}$ then $\covU{A}\cong \R\times \widetilde{SU_\tau}(A)$.
\end{theorem}
\begin{proof}
    Firstly as $\R\times \widetilde{SU_\tau}(A)\cong \widetilde{\R}\times \widetilde{SU_\tau}(A)$ we will show that the homomorphism of groups
    \begin{align*}
        \Phi:\widetilde{\R}\times \widetilde{SU_\tau}(A)&\rightarrow \covU{A}\\
        h(t)\times u(t)&\mapsto e^{2\pi i h(t)}u(t)
    \end{align*}
    is an isomorphism. For this it suffices to show that any path homotopy in $\covU{A}$ is induced by a path homotopy under $\Phi$.
    \par Let $u\in U(A)$ and $f:[0,1]^2\rightarrow U(A)$ a homotopy of paths starting at $1$ and ending at $u$ i.e. that $f(s,0)=1$ and $f(s,1)=u$ for all $s\in [0,1]$.
    \par By Lemma \ref{lem: representingunitaries}, with $X=[0,1]^2/[0,1]\times \{0\}$ and $x_0=[[0,1]\times \{0\}]$, there exists continuous maps $g:[0,1]^2\rightarrow SU_\tau(A)$ and $h:[0,1]^2\rightarrow \R$ with 
$$f(s,t)=e^{2\pi ih(s,t)}g(s,t)$$
and $g(s,0)=1$ and $0=h(s,0)$. Also $u=e^{2\pi i h(s,1)}g(s,1)$ for all $s\in [0,1]$ and thus $g(s,1)g(0,1)^{-1}=e^{2\pi i(h(0,1)-h(s,1))}$ is a continuous path from $[0,1]$ into $e^{2\pi i \tau_{^*}(K_0(A))}\subset SU_\tau(A)$. As $e^{2\pi i \tau(K_0(A))}$ is discrete in the $d_\tau$ topology it follows that $g(s,1)g(0,1)^{-1}$ is constant and equal $1$ for all $s\in [0,1]$. Similarly $h(0,1)=h(s,1)$ for all $s\in [0,1]$. In particular $h$ and $g$ are path homotopies and $f$ is the image of $h\times g$ under $\Phi$.
\end{proof}
In some cases of interest the group $SU_\tau(A)$ will automatically be simply connected and so $\covU{A}\cong \R\times SU_\tau(A)$.
\begin{theorem}\label{thm:simplyconnected}
    Suppose that $Z(A)=\C$ and $\pi_1(U(A))\cong K_0(A)$.\ Then $\pi_1(SU_\tau(A))\cong \{x\in K_0(A): \tau_{*}x=0\}$. In particular if $\tau_{*}:K_0(A)\rightarrow \R$ is injective then $SU_\tau(A)$ is simply connected.
\end{theorem}
\begin{proof}
    The short exact sequences
    \begin{align*}
        \bT\longrightarrow &U(A)\longrightarrow PU(A),\\
        e^{2\pi i \tau_{*}(K_0(A))}\longrightarrow &SU_\tau(A)\longrightarrow PU(A)
    \end{align*}
    are Hurewicz fibrations. The long exact sequence of homotopy groups arising from the first fibration yields  %(this follows as they arise as coset spaces of a Lie group with regard to an action of a closed subgroup see \cite[II.7 Example 4, II.7 Corollary 14]{SPA66}) 
    \[
    \ldots \rightarrow \pi_1(\bT)\xrightarrow{\iota} \pi_1(U(A)) \rightarrow \pi_1(PU(A))\rightarrow 0.
    \]
    As $A$ is tracial, then after identifying $\pi_1(\bT)\cong \Z$ and $\pi_1(U(A))\cong K_0(A)$ the map $\iota$ sends $\Z$ injectively into $K_0(A)$ and so $\pi_1(PU(A))\cong K_0(A)/\Z\langle[1_A]\rangle$. The long exact sequence of homotopy groups associated to the second fibration yields
    $$0\rightarrow \pi_1(SU_\tau(A))\rightarrow \pi_1(PU(A))\xrightarrow{\Phi} \pi_0(e^{2\pi i \tau_{*}(K_0(A))})\rightarrow 0.$$
    Where the mapping $\Phi$  is defined by sending a class $[\lambda]\in \pi_1(PU(A))$, of a based loop $\lambda$ in $PU(A)$, to the path component of the endpoint $f(1)$ of any continuous lift $f:[0,1]\rightarrow SU_\tau(A)$ for $\lambda$ with $f(0)=1$. 
    \par By the surjection $K_0(A)\rightarrow \pi_1(PU(A))$, every loop in $\pi_1(PU(A))$ is represented by a loop $\lambda_{p,q}(t)=\eta(t)+\bT$ with $\eta(t)$ differentiable and homotopic to $e^{2\pi itp}e^{-2\pi i tq}$ in $U_n(A)$ with $p,q\in P_n(A)$ and $n\in \N$. By the proof of Lemma \ref{lem: representingunitaries} there are continuous, differentiable paths $g\in SU_\tau(A)$ starting at $1$ and $h\in \R$ starting at $0$ such that $\eta(t)=e^{2\pi i h(t)}g(t)$ and $\tilde{\Delta}_\tau(g)=0$. Then $g(t)$ is a continuous lift of $\lambda_{p,q}(t)$ to $SU_\tau(A)$ and 
    \[g(1)=e^{-2\pi ih(1)}=e^{-2\pi i \tilde{\Delta}_\tau(\eta)}=e^{2\pi i (\tau_*(p)-\tau_*(q))}
    \]
Thus $\Phi(\lambda_{p,q})=f_{p,q}(1)=e^{2\pi it \tau_{*}(p)-\tau_{*}(q)}$. Now $\pi_1(SU_\tau(A))\cong \{x\in \pi_1(PU(A)):\Phi(x)=1\}$, which is precisely $\{x\in K_0(A):\tau_{*}x=0\}$.
\end{proof}
\begin{theorem}\label{thm:universalcover}
    Let $A$ be a unital, simple, separable C$^*$-algebra $A$ with $U(A)=U(A)_0$, $\pi_1(U(A))\cong K_0(A)$. Suppose that $A$ admits a trace $\tau$ with $\tau_{*}:K_0(A)\rightarrow \R$ faithful. Then there is an isomorphism $\covU{A}\cong \R\times SU_\tau(A)$.
\end{theorem}
\begin{proof}
    This follows from Theorems \ref{thm:universalcover1} and \ref{thm:simplyconnected}.
\end{proof}
The conditions of the theorem above are satisfied in the cases of UHF algebras, the Jiang--Su algebra or irrational rotation algebras. As an application of Theorem \ref{thm:universalcover} we can now recover instances of \cite[Theorem 3.9]{IZ23} from the crossed module formalism of Section \ref{sec:crossedmodules}. We argue this in the following remark.
\begin{remark}
    Let $(A,\tau)$ be a unique trace C$^*$-algebra satisfying the conditions of Theorem \ref{thm:universalcover}. There is an isomorphism $\phi$ from the crossed module given by 
    \begin{align*}
    \Ad:\R\times SU_\tau(A) &\rightarrow \Aut A\\
    (\lambda,u)&\rightarrow \Ad(u)
    \end{align*}
    to the crossed module $\widetilde \cG_A$.\footnote{The action of $\Aut A$ on $\R\times SU_\tau(A)$ is given by the canonical action on the second coordinate.} This isomorphism is defined by 
    \begin{align*}
    \phi:\R \times SU_\tau(A)&\rightarrow \widetilde{U}(A)\\
    (\lambda,u)&\mapsto e_{\lambda}\widetilde{u}
    \end{align*}
    where $\widetilde{u}$ is any path in $SU_\tau(A)$ starting at $1$ and ending at $u$. Note that this is independent of the choice of path $\tilde{u}$ precisely because $SU_\tau(A)$ is simply connected.
    \par Restricted to $\R \times e^{2\pi i \tau(K_0(A))}$, the map $\phi$ maps surjectively onto $K_0^{\#}(A)$ by sending $(\lambda,e^{2\pi i\tau(x)})\mapsto e_{\lambda+\tau(x)}-x$. Indeed, for any projection $p$ in $A$ it is clear that $\phi(\lambda,e^{2\pi i \tau(p)})(t)=e^{2\pi i \lambda t}e^{2\pi i (\tau(p)-p)t}=e_{\lambda+\tau(p)}-[p]$, an argument as in the end of Theorem \ref{thm:simplyconnected} shows the case for general $x\in K_0(A)$. Denote by $\psi:K_0^{\#}(A)\rightarrow \R\times e^{2\pi i \tau(K_0(A))}$ the inverse of $\phi$ restricted to $K_0^{\#}(A)$. It is easy to see that $\psi(x)=(\tau(x),e^{-2\pi i \tau(x)})$ for all $x\in K_0(A)$ and that $\psi([e_\lambda])=(0,\lambda)$. Thus $\phi$ and $\psi$ give an isomorphism between the short exact sequences of crossed modules
    \[
    K_0^{\#}(A)\longrightarrow \widetilde{\cG}_A\longrightarrow P\cG_A
    \]
    given in Lemma \ref{lem:ses-covU} and
    \begin{equation} \label{eqn:short-ex-RSU(A)}
	\begin{tikzcd}
		\R \times e^{2\pi i \tau(K_0(A))} \ar[r] \ar[d] & \R \times SU_\tau(A)\ar[d,"\Ad{}"] \ar[r] & PU(A) \ar[d,"\Ad{}"] \\
		1 \ar[r] & \Aut A \ar[equal]{r}	 & \Aut A.
	\end{tikzcd}
\end{equation}
By Lemma \ref{lem:longexact}, the crossed module in \eqref{eqn:short-ex-RSU(A)} induces an exact sequence of the associated cohomology pointed sets whose last term is the mapping $H^1(\Gamma, P\cG_A)\rightarrow H^3(\Gamma,\R \times e^{2\pi i\tau(K_0(A))})$ sending a conjugacy class of a $\Gamma$-kernel $\alpha$ to the three cocycle $(0,\ob_\tau(\alpha))$. Now, as $\tau$ is preserved by any automorphism of $A$ and $\tau_{*}$ is faithful on $K_0(A)$, the induced action of any $\Gamma$-kernel on $K_0^{\#}(A)$ is trivial. Thus the rightmost term of the exact sequence in Theorem \ref{thm:operatoralgcrossedmoduleseq} is simply $H^3(\Gamma,K_0^{\#}(A))$. By the naturality of the associated long exact sequences, the isomorphism of crossed modules above induces an isomorphism of the associated exact sequences of the cohomology pointed sets where the last vertical isomorphism is $\psi_*:H^3(\Gamma,K_0^{\#}(A))\rightarrow H^3(\Gamma,\R \times e^{2\pi i \tau(K_0(A))})$. Hence we have that $\psi_*(\widetilde \ob(\alpha))=(\ob_\tau(\alpha),0)$ recovering \cite[Theorem 3.9]{IZ23}.
\end{remark}

\appendix

\section{Proof of \cref{thm:homotpyeqclassifyingspaces} ($B^D\cG \simeq B^\otimes \cG$)}\label{appendix}

Let $\cG = (\partial \colon H \to G)$ be a crossed module. Recall that $B^D\cG = \lVert N_\ast^D(\cG) \rVert$, where the Duskin nerve $N_\ast^D(\cG)$ is the simplicial space defined by
\[
    N_n^D(\cG) = \Fun{[n]}{\cG}\ .
\]
The pseudofunctors $\Fun{[n]}{\cG}$ form the object space of a topological category $\lxF{[n]}{\cG}$, in which the morphisms are the natural transformations, where we will only consider transformations that map to the identity on objects. Let $(\alpha^{(k)}, u^{(k)})$ for $k \in \{0,1\}$ be pseudofunctors $[n] \to \cG$. A natural transformation between them corresponds to elements $(w_{ij})_{i < j}$ with $w_{ij} \in H$ such that
\begin{equation} \label{eqn:nat}
    \begin{tikzcd}[column sep=2cm]
        \alpha^{(0)}_{ik} \arrow[r, Rightarrow, "u_{ijk}^{(0)}"] \arrow[d, Rightarrow, "w_{ik}" left] & \alpha^{(0)}_{ij} \circ \alpha^{(0)}_{jk} \arrow[d, Rightarrow, "w_{ij}\,\circ_h\, w_{jk}"] \\
        \alpha^{(1)}_{ik} \arrow[r, Rightarrow, "u_{ijk}^{(1)}" below] & \alpha^{(1)}_{ij} \circ \alpha^{(1)}_{jk}
    \end{tikzcd}
\end{equation}
The commutation of diagram \eqref{eqn:nat} implies that $\alpha_{ij}^{(1)} = \partial(w_{ij}) \alpha_{ij}^{(0)}$ and 
\(
    w_{ij}\alpha_{ij}^{(0)}(w_{jk})u_{ijk}^{(0)} = u_{ijk}^{(1)}w_{ik}
\),
which should be compared to \eqref{eqn:alpha12} and \eqref{eqn:u12} with $\gamma=\id{}$. The spaces $\lxF{[n]}{\cG}$ assemble to a simplicial topological category. By taking the nerve for each $n$ we end up with a bisimplicial space
\[
    X_{m,n} = N_m\lxF{[n]}{\cG} 
\]
and its geometric realisation $B\Flax_{\cG} = \lVert\,\lVert X_{\ast,\ast} \rVert_m\rVert_n$. Here, the subscripts $m$ and $n$ denote the two simplicial directions and we are using the fat geometric realisation for both of them. By \cite[Equation (1.9)]{EBWI19} the order of the two geometric realisations does not matter and only changes the outcome up to homeomorphism.

Similarly, the space $B^\otimes \cG$ can be obtained as the geometric realisation of a bisimplicial space in the following way: Let $\stFun{[n]}{\cG}$ be the space of strict functors $[n] \to \cG$. Each such functor corresponds to a chain of elements $(\alpha_{i(i+1)})_{0 \leq i \leq n-1}$ with $\alpha_{ij} \in G$. It is the object space of the topological category $\stF{[n]}{\cG}$, where the morphisms are again the natural transformations that map to the identity on objects. With this definition we have a canonical isomorphism of topological categories
\[
    \stF{[1]}{\cG} \cong \cG_{\otimes} \qquad , \qquad \stF{[n]}{\cG} \cong (\cG_{\otimes})^n
\]
and the monoidal structure of $\cG_{\otimes}$ corresponds to the composition of functors in $\stF{[1]}{\cG}$. In particular,
\[
    \lvert N_\ast \stF{[1]}{\cG} \rvert \cong \lvert N_\ast \cG_{\otimes} \rvert \quad \text{and} \quad \lVert\,\lvert N_\ast \stF{[\ast]}{\cG} \rvert\,\rVert = B^{\otimes} \cG\ .
\]
If $\cG = (\partial \colon H \to G)$ and $H$ is well-pointed, then for each fixed $n \in \N_0$ the simplicial space $k \mapsto N_k \stF{[n]}{\cG}$ is good and the quotient map
\[
    \lVert N_\ast \stF{[n]}{\cG} \rVert \to \lvert N_\ast \stF{[n]}{\cG} \rvert 
\]
is a weak homotopy equivalence by \cite[Proposition A.1 (iv)]{SE74}. The above map is compatible with the simplicial structure. Therefore
\begin{equation} \label{eqn:BfG}
    B^{\otimes}_f\cG := \lVert\,\lVert N_\ast \stF{[\ast]}{\cG} \rVert\,\rVert \to \lVert\,\lvert N_\ast \stF{[\ast]}{\cG} \rvert\,\rVert = B^\otimes \cG
\end{equation}
is a weak equivalence as well by \cite[Theorem 2.2]{EBWI19}.

By construction we now have two maps to $B\Flax$: We can consider the Duskin nerve $N_n^D(\cG) = \Fun{[n]}{\cG}$ as a bisimplicial space that is constant in one direction and obtain a map $B^D\cG \to B\Flax$. We also have the inclusion of the category of strict functors into the category of pseudofunctors, which gives $B_f^\otimes\cG \to B\Flax$. We will finish the proof by showing that both of these are weak homotopy equivalences.  

\begin{lemma} \label{lem:BtensorBF}
    The map $B_f^\otimes\cG \to B\Flax_{\cG}$ induced by the inclusion of the strict functor category into the category of pseudofunctors is a weak equivalence.
\end{lemma}

\begin{proof}
    By \cite[Theorem 2.2]{EBWI19} it suffices to show that for every $n \in \N_0$ the map
    \[
        \lVert N_\ast \stF{[n]}{\cG} \rVert \to \lVert N_\ast \lxF{[n]}{\cG} \rVert
    \]
    induced by the inclusion functor $\Phi \colon \stF{[n]}{\cG} \to \lxF{[n]}{\cG}$ is a weak equivalence. 
    
    Consider the continuous forgetful functor $\Psi \colon \lxF{[n]}{\cG} \to \stF{[n]}{\cG}$ that maps a pseudofunctor $(\alpha, u) \colon [n] \to \cG$ to the strict one given by the sequence $(\alpha_{01}, \alpha_{12}, \dots, \alpha_{(n-1)n})$ (i.e.\ $i < j$ is mapped to $\alpha_{i(i+1)} \circ \dots \circ \alpha_{(j-1)j})$. The functor $\Psi$ maps any natural transformation $w:(\alpha^{(0)},u^{(0)})\rightarrow (\alpha^{(1)},u^{(1)})$ to the natural transformation 
    \[
    \Psi((\alpha^{(0)},u^{(0)}))_{ij}=\alpha_{i(i+1)}^{(0)}\circ \ldots \circ\alpha_{(j-1)j}^{(0)}\xRightarrow{w_{i(i+1)}\circ_h\ldots \circ_h w_{(j-1)j}}\Psi((\alpha^{(1)},u^{(1)}))_{ij} 
    \]
    for $i<j$. We have $\Psi \circ \Phi = \id{}$. There is a natural transformation $\eta \colon \id{} \Rightarrow \Phi \circ \Psi$, which can be inductively defined as follows: On $\alpha_{ij}$ with $j - i = 1$ we define $\eta_{(\alpha,u)} = \id{}$. If $j - i = 2$, then we define it to be
    \[
        v_{i,(i+1),j} \colon u_{i(i+1)j} \colon \alpha_{ij} \Rightarrow \alpha_{i(i+1)} \circ \alpha_{(i+1)j}
    \]
    If $j -i = 3$, then we take 
    \[
       v_{i, (i+1), (j-1), j} = u_{i(i+1)(j-1)} u_{i(j-1)j} \colon \alpha_{ij} \Rightarrow \alpha_{i(i+1)} \circ \alpha_{(i+1)(i+2)} \circ \alpha_{(i+2)j}
    \]
    and in general
    \[
        v_{i, \dots, j} = v_{i, \dots, j-1} u_{i(j-1)j} \colon \alpha_{ij} \Rightarrow \alpha_{i(i+1)} \circ \dots \circ \alpha_{(j-1)j}
    \]
    This defines a natural transformation. With $I$ as in \cref{lem:map_on_BG} we have a functor $\lxF{[n]}{\cG} \times I \to \lxF{[n]}{\cG}$. Combining \cite[Proposition 6.2]{book:MaySimplicial} and \cite[Lemma 1.15]{EBWI19} as in the proof of \cref{lem:map_on_BG} this natural transformation gives rise to a homotopy 
    \[
        \lVert N_\ast \lxF{[n]}{\cG} \rVert \times [0,1] \to \lVert N_\ast \lxF{[n]}{\cG} \rVert
    \]
    between $\lVert N_{\ast} \Phi\rVert \circ \lVert N_{\ast} \Psi \rVert$ and the identity.
\end{proof}

\begin{remark}
    Note that the functor $\Psi$ is not compatible with the simplicial space structure on both sides, which is why we only get a levelwise equivalence and not a homotopy inverse to $\Phi$.
\end{remark}

\begin{lemma} \label{lem:BDBF}
    Let $(\partial \colon H \to G)$ be a crossed module such that $H$ is well-pointed. The map $B^D\cG \to B\Flax_{\cG}$ induced by the inclusions 
    \[
        \Fun{[n]}{\cG} \to \lxF{[n]}{\cG}
    \]
    is a weak homotopy equivalence.
\end{lemma}

\begin{proof}
    Consider $Y_{n,m} = \Fun{[n]}{\cG}$ as a bisimplicial space, which is constant in the $m$-direction, then $\lVert Y_{n,\ast} \rVert \simeq \Fun{[n]}{\cG} \times \Delta^\infty$ and the simplicial map $\Fun{[n]}{\cG} \to \lVert Y_{n,\ast} \rVert$ induces a homotopy equivalence 
    \[
        \lVert \Fun{[\ast]}{\cG} \rVert \simeq \lVert\,\lVert Y_{\ast,\ast} \rVert \,\rVert\ . 
    \]
    By \cite[Theorem 2.2]{EBWI19} and \cite[Equation (1.9)]{EBWI19} it suffices to show that for each $m \in \N_0$ the map
    \[
        \lVert Y_{\ast,m} \rVert \to \lVert N_{m} \lxF{[\ast]}{\cG} \rVert
    \]
    is a weak homotopy equivalence. Since $Y_{n,0} = N_0\lxF{[n]}{\cG}$, this will follow once we have shown that for $m \geq 1$ the first degeneracy map
    \[
        s_0^h \colon N_{m-1} \lxF{[n]}{\cG} \to N_{m} \lxF{[n]}{\cG}
    \]
    is a simplicial deformation retract (of the simplicial spaces in $n$). We call this degeneracy map the horizontal one. In contrast to that, the vertical ones~$s_k^v$ keep $m$ fixed and raise $n$. The homotopy inverse of $s_0^h$ is provided by the first horizontal face map 
    \[
        d_0^h \colon N_{m} \lxF{[n]}{\cG} \to N_{m-1} \lxF{[n]}{\cG}
    \]
    which satisfies $d_0^h \circ s_0^h = \id{}$. By \cite[Lemma 1.15]{EBWI19} the result will follow once we have constructed a simplicial homotopy between $s_0^h \circ d_0^h$ and $\id{}$. An $n$-simplex in $N_m\lxF{[n]}{\cG}$ takes the form 
    \begin{equation} \label{eqn:n-simpl}
        \begin{tikzcd}[column sep=1.5cm]
            x^0 \ar[r,Rightarrow,"v^0"] & x^1 \ar[r,Rightarrow,"v^1"] & \dots \ar[r,Rightarrow,"v^{m-2}"] & x^{m-1} \ar[r,Rightarrow,"v^{m-1}"] & x^m 
        \end{tikzcd}
    \end{equation}
    where $x^i = (\alpha^i, u^i) \in \Fun{[n]}{\cG}$ are objects in the functor category. The $k$th degree part $H_k$ of the simplicial homotopy $H$ will send this to
    \[
        \begin{tikzcd}[column sep=1.5cm]
             \mu_k v^0 \ar[r,Rightarrow,"h_k v^0"] & s_k^v x^1 \ar[r,Rightarrow,"s_k^v v^1"] & \dots \ar[r,Rightarrow,"s_k^v v^{m-2}"] & s_k^v x^{m-1} \ar[r,Rightarrow,"s_k^v v^{m-1}"] & s_k^v x^m            
        \end{tikzcd}
    \]
    in $N_m\lxF{[n+1]}{\cG}$ with $\mu_k v^0$ and $h_k v^0$ still to be defined. Let 
    \[
        \sigma_k(j) = \begin{cases}
            j & \text{if } j \leq k \ ,\\
            j-1 & \text{if } j > k
        \end{cases}
    \]
    and let $\mu_k v^0 = (\alpha^{\mu,k}, u^{\mu,k}) \in \Fun{[n+1]}{\cG}$ with
    \[
        \alpha^{\mu,k}_{i,j} = \begin{cases}
            \alpha^1_{i,j} & \text{if } j \leq k\ ,\\[.2cm]
            \alpha^0_{\sigma_k(i),j-1} & \text{if } j > k
        \end{cases} 
        \quad \text{and} \quad
        u^{\mu,k}_{i,j,\ell} = \begin{cases}
            u_{i,j,\ell}^1 & \text{if } \ell \leq k\ ,\\[.2cm]
            v_{i,j}^0 u_{i,j,\ell-1}^0 & \text{if } j \leq k < \ell\ , \\[.2cm]
            u^0_{\sigma_k(i),j-1,\ell-1} & \text{if } k < j
        \end{cases}
    \]
    It is straightforward to check that this indeed defines a valid element of $\Fun{[n+1]}{\cG}$ and that $\mu_0v^0 = s_0^v(x^0)$ (and hence $d_0^v(\mu_0v^0) = x^0)$ and $d^v_{n+1}(\mu_nv^0) = x^1$. The transformation $h_k v^0 \colon \mu_k v^0 \Rightarrow s_k^v x_1$ is defined by
    \[
        (h_k v^0)_{i,j} = \begin{cases}
            \id{\alpha_{i,j}^1} & \text{if } j \leq k\ , \\
            v_{\sigma_k(i),j-1}^0 & \text{if } j > k\ .
        \end{cases}
    \]
    Using \eqref{eqn:nat} we see that this is indeed natural. Note that $h_0 v^0 = s_k^v v^0$ (and hence $d_0^v (h_0 v^0) = v^0$) and $d_{n+1}^v(h_nv^0) = \id{}$. It is not hard to check that the map $H_k$ defined in this way satisfies the other conditions in \cite[Lemma 1.15]{EBWI19} as well and therefore provides a simplicial homotopy between \eqref{eqn:n-simpl} and 
    \[
        \begin{tikzcd}[column sep=1.5cm]
            x^1 \ar[r,Rightarrow,"\id{}"] & x^1 \ar[r,Rightarrow,"v^1"] & \dots \ar[r,Rightarrow,"v^{m-2}"] & x^{m-1} \ar[r,Rightarrow,"v^{m-1}"] & x^m 
        \end{tikzcd}    
    \]
    which is the image of \eqref{eqn:n-simpl} under $s_0^h \circ d_0^h$. This completes the proof.
\end{proof}

\begin{corollary}\label{cor:homotopyequivalence}
    Let $\cG = (\partial \colon H \to G)$ be a topological crossed module with $H$ well-pointed. There is a weak-equivalence 
    \[
        B^D\cG \simeq B^\otimes \cG
    \]
    that is natural in $\cG$.
\end{corollary}

\begin{proof}
The weak equivalence is induced by the following zig-zag of continuous maps:
\[
    \begin{tikzcd}[column sep=1.6cm]
        B^\otimes \cG & \ar[l,"\simeq" above,"\text{eq. }\eqref{eqn:BfG}"] B^\otimes_f\cG \ar[r,"\text{Lem. } \ref{lem:BtensorBF}" below, "\simeq" above] & B\Flax_{\cG} & \ar[l,"\text{Lem. } \ref{lem:BDBF}" below, "\simeq" above] B^D\cG 
    \end{tikzcd}
\]
All of these are induced by constructions that are natural in $\cG$.
\end{proof}

\bibliographystyle{plain}
\bibliography{G-Kernels}

\end{document}